\documentclass{amsart}%elsarticle}%amsart}%
\usepackage{hyperref}
\usepackage{amssymb}
\usepackage{amsmath}
\usepackage{amsfonts}
\usepackage{amsthm}
\usepackage{framed}
\usepackage{graphicx}%
\DeclareMathAlphabet{\mathpzc}{OT1}{pzc}{m}{it}
\theoremstyle{plain}
\newtheorem{thm}{Theorem}[section]

\newtheorem{cor}[thm]{Corollary}
\newtheorem{lem}[thm]{Lemma}

\newtheorem{prop}[thm]{Proposition}

\renewenvironment{proof}[1][Proof]{\textbf{#1.} }{\ \rule{0.5em}{0.5em}}
\theoremstyle{definition}

\theoremstyle{remark}

\newtheorem*{rmk*}{Remark}
\newtheorem{nota}[thm]{Notation}

\theoremstyle{plain}
\newtheorem{claim}[]{Claim}
\theoremstyle{definition}
\newtheorem{assumption}[]{Assumption}
\numberwithin{equation}{section}

\makeatletter
\renewcommand*\env@cases[1][1.2]{%
  \let\@ifnextchar\new@ifnextchar
  \left\lbrace
  \def\arraystretch{#1}%
  \array{@{}l@{\quad}l@{}}%
}
\makeatother

\newcommand{\bb}[1]{\mathbb{#1}}
\newcommand{\mc}[1]{\mathcal{#1}}
\newcommand{\pzc}[1]{\mathpzc{#1}}

\newcommand{\bbR}{\mathbb{R}}

\newcommand{\bbN}{\mathbb{N}}
\newcommand{\mcL}{\mathcal{L}}

\newcommand{\E}{\mathbb{E}}
\newcommand{\W}{\Omega}
\newcommand{\w}{\omega}
\newcommand{\ed}{\stackrel{\text{dist}}{=}}

\newcommand{\Vol}{\operatorname{Vol}}
\newcommand{\dVol}{\operatorname{dVol}}
\newcommand{\Scal}{\operatorname{Scal}}
\newcommand{\Ric}{\operatorname{Ric}}
\newcommand{\Law}{\operatorname{Law}}
\newcommand{\Hom}{\operatorname{Hom}}
\newcommand{\ex}{\operatorname{exp}}
\newcommand{\tr}{\operatorname{tr}}
\newcommand{\mesh}{\operatorname{mesh}}

\newcommand{\epsi}{\varepsilon}
\newcommand{\T}{\mc{T}}

\newcommand{\p}{\mathcal{P}}
\newcommand{\Hp}{H_{\p}}
\newcommand{\Hpe}{H_{\p}^{\epsi}}
\newcommand{\mLp}{\mcL_{\p}}
\newcommand{\Up}{\mc{U}_{\p}}
\newcommand{\ep}{\mc{E}_{\p}}
\newcommand{\bp}{b^{\p}}
\newcommand{\up}{u^{\p}}
\newcommand{\Ap}[1]{A^{\p}_{#1}}
\newcommand{\Sp}[1]{S^{\p}_{#1}}
\newcommand{\Cp}[1]{C^{\p}_{#1}}
\newcommand{\Vp}[1]{V^{\p}_{#1}}
\newcommand{\Fp}[1]{F^{\p}_{#1}}
\newcommand{\Kp}[1]{K^{\p}_{#1}}
\newcommand{\Rp}{\mc{R}_{\p}}
\newcommand{\betap}{\beta^{\p}}
\newcommand{\tp}{\theta^{\p}}
\newcommand{\mAp}{\mc{A}_{\p}}

\newcommand{\nuG}[1]{\nu_{G_{#1}}}
\newcommand{\nuS}[1]{\nu_{S_{#1}}}
\newcommand{\muS}[1]{\mu_{S_{#1}}}

\newcommand{\re}{\mathbb{R}}
\newcommand{\fancyR}{\pzc{R}}
\newcommand{\fancyS}{\pzc{S}}
\newcommand{\fancyA}{\pzc{A}}
\newcommand{\fancyL}{\pzc{L}}
\newcommand{\ddim}{\operatorname{d}}
\newcommand{\rhot}{\tilde{\rho}}

\newcommand{\F}{\mc{F}}
\newcommand{\GFp}{G_{\p}^{\F_{\p}}}
\newcommand{\SFp}{S_{\p}^{\F_{\p}}}
\newcommand{\D}{\mc{D}}
\newcommand{\mD}{\D}
\newcommand{\tk}{\theta_k}

\newcommand{\nats}{\bbN}
\newcommand{\chs}[2]{ \left( \begin{array}{c} #1 \\ #2 \end{array}  \right)}
\newcommand{\HpeR}{H_{\p}^{\epsi}(\re^d)}
\newcommand{\rem}{\operatorname{rem}}
\newcommand{\bdry}[1]{\partial_{#1}}
\newcommand{\bdryc}[1]{\Omega_{#1}}

\newcommand{\subsect}[1]{$\newline$ \textbf{#1}. }

\title[Approximation to Wiener Measure]{An Approximation to Wiener Measure and Quantization of the Hamiltonian on Manifolds with Non-positive Sectional Curvature}
\author{Thomas Laetsch}
\thanks{This research was supported in part by NSF Grants DMS-0804472 and DMS-1106270}
\keywords{Path integrals, finite dimensional approximations, Wiener measure, infinite dimensional analysis}
\subjclass[2010]{28C20, 58D30, 60H99}
\address{Department of Mathematics, University of Connecticut, Storrs, CT 06269}
\email{thomas.laetsch@uconn.edu}
\urladdr{http://www.math.uconn.edu/~laetsch}

\begin{document}
\maketitle

\begin{abstract}
This paper gives a rigorous interpretation of a Feynman path integral on a Riemannian manifold $M$ with non-positive sectional curvature. A $L^2$ Riemannian metric $G_{\p}$ is given on the space of piecewise geodesic paths $\Hp(M)$ adapted to the partition $\p$ of $[0,1]$, whence a finite-dimensional approximation of Wiener measure is developed. It is proved that, as $\mesh(\p) \to 0$, the approximate Wiener measure converges in a $L^1$ sense to the measure $\ex\{ -\frac{2 + \sqrt{3}}{20\sqrt{3}} \int_0^1 \Scal(\sigma(s)) ds \} d\nu(\sigma)$ on the Wiener space $W(M)$ with Wiener measure $\nu$. This gives a possible prescription for the path integral representation of the quantized Hamiltonian, as well as yielding such a result for the natural geometric approximation schemes originating in \cite{AndDrive:1999} and followed by \cite{ALim:2007}.
\end{abstract}

\tableofcontents

\section{Introduction}
Let  $M$ be a $d$-dimensional Riemannian manifold with metric $g$, fixed point $o \in M$, and Levi-Civita covariant derivative $\nabla$. For the remainder of this paper we will assume that the curvature and its derivative are bounded on $M$. We will eventually also require that the sectional curvature of $M$ is non-positive, making sure to mention when we impose this restriction. 

The Wiener space of $M$ consists of the continuous paths starting at $o$ and parameterized on $[0,1]$, 
\begin{align}
W(M) = \{ \sigma \in C([0,1] \to M) : \sigma(0)=o \}. \label{eqn.WM}
\end{align}
The Wiener measure associated to $M$ is the unique probability measure $\nu$ on $W(M)$ such that,
\begin{align}
&\int_{W(M)} f(\sigma) d\nu(\sigma)  = \int_{M^n} F(x_1, ..., x_n) \prod_{i=1}^n P_i(dx_i)
\end{align}
whenever $f$ has the form $f(\sigma) = F(\sigma(s_1), ..., \sigma(s_n))$ where $\p = \{ 0=s_0 < s_1 < \cdots < s_n = 1\}$ is a partition of $[0,1]$ and $F$ is a bounded and measurable function. The measures $P_i(dx_i)$ are defined as $P_i(dx_i) := p_{\Delta_i s} (x_{i-1}, x_i) dx_i$, where $p_s(x,y)$ denotes the fundamental solution to the heat equation on $M$, $\Delta_i s = s_i - s_{i-1}$, and $dx_i$ is the Riemannian volume form on $M$.

The purpose of this paper is to give a rigorous interpretation of a heuristic path integral on $M$ having the form,
\begin{align}
\frac{1}{Z} \int_{W(M)} f(\sigma(1)) \ex \left\{  \int_0^1\left( -\frac{1}{2}  \| \sigma'(s) \|^2 + V(s) \right)ds \right\} \mc{D}\sigma \label{eqn.heurpath}
\end{align}
via a finite dimensional approximation to Wiener measure. The ``derivation'' of Eq. (\ref{eqn.heurpath}) follows from an application of Trotter's product formula and a limiting argument from which $Z$ arises as a ``normalization'' constant that can either be interpreted as $0$ or $\infty$, and $\mc{D}\sigma$ is an infinite-dimensional Lebesgue type measure which, in truth, does not exist.    Moreover, $V$ is a potential and $-\frac{1}{2} \| \sigma'(s) \|^2 + V(s)$ yields an energy term which is problematic since the weight of the space $W(M)$ lands on nowhere differentiable paths.

In spite of the need to give a rigorous interpretation, heuristic path integrals such as those in Eq. (\ref{eqn.heurpath}) have proven themselves useful and arise often in physics literature. Particularly, one can interpret this path integral as the path integral quantization of the Hamiltonian on $M$. Much of the current interest concerning path integrals in physics began with Feynman in \cite{MR0026940} and has since grown deeply. The role of path integrals in quantum mechanics is surveyed by Gross in \cite{MR838562} and detailed more by Feynman and Hibbs in \cite{MR2797644} as well as Glimm and Jaffe in \cite{QuantumPhysics:1981}.

\section{Main Result}
For the partition $\p = \{ 0=s_0 < s_1 < \cdots < s_n = 1\}$, define the finite dimensional subspace $\Hp(M)$ of $W(M)$ by
\begin{align}
\Hp(M) = \{ \sigma \in W(M) : \sigma \text{ is piecewise geodesic with respect to } \p \}. \label{eqn.HpM}
\end{align}
We make $\Hp(M)$ into a Riemannian manifold by endowing it with the $L^2$ metric $G_{\p}$, defined by,
\begin{align}
G_{\p}(X,Y) = \int_0^1 g\left( X(s), Y(s) \right) ds, \label{eqn.G}
\end{align}
where we are making the natural identification of the tangent space $T_{\sigma}\Hp(M)$ with the piecewise Jacobi fields along $\sigma$ in $M$. From here we define the approximate Wiener measure $\nu_{G_{\p}}$ on $\Hp(M)$ by,
\begin{align}
d\nu_{G_{\p}} = \frac{1}{Z_{G_{\p}}} e^{ -\frac{1}{2} \int_0^1\| \sigma'(s) \|^2 ds } \dVol_{G_{\p}}, \label{eqn.nuG}
\end{align}
where $\Vol_{G_{\p}}$ is the Riemannian volume form given by $G_{\p}$ and $Z_{G_{\p}}$ is a normalization constant which forces $\nu_{G_{\p}}$ to be a probability measure in the case that $M = \re^d$. With the matrix $\mLp$ introduced below in Eq. (\ref{eqn:Lp}),
\begin{align}
Z_{G_{\p}} = \sqrt{\det{\mLp}} \prod_{i=1}^n (2\pi \Delta_i s)^{d/2}. \label{eqn.ZGp}
\end{align}

We can now state the main result of this paper.

\begin{thm} 
\label{thm.mainthm}
Suppose that $M$ has non-positive sectional curvature and $f : W(M) \to \re$ is bounded and continuous. Then, 
\begin{align}
\lim_{|\p| \to 0} \int_{\Hp(M)} f(\sigma) d\nuG{\p}(\sigma) = \int_{W(M)} f(\sigma) e^{-\frac{2+\sqrt{3}}{20\sqrt{3}} \int_0^1 \Scal(\sigma(s)) ds} d\nu(\sigma) \label{eqn.mainthm}
\end{align}
where $\Scal$ is the scalar curvature on $M$ and $\p$ is taken to be the equally-spaced partition $\p = \{ 0, 1/n, 2/n, ..., 1\}$.
\end{thm}

Interpreting Eq (\ref{eqn.mainthm}) as the path integral quantization of the hamiltonian $H$ with
\begin{align*}
e^{-t H} :=  \int_{W(M)} f(\sigma) e^{-\frac{2+\sqrt{3}}{20\sqrt{3}} \int_0^1 \Scal(\sigma(s)) ds} d\nu(\sigma),
\end{align*}
an application of the Feynman-Kac formula gives $H = - \frac{1}{2} \Delta +V + \frac{2+\sqrt{3}}{20\sqrt{3}} \Scal $, where we have again included the potential $V$. 

\subsect{Remarks}
	This paper views $\mc{D}\sigma$ as a volume form on $W(M)$ and approximates Wiener measure on the piecewise geodesic path space $\Hp(M)$; an approach which started with \cite{AndDrive:1999} and has further been explored by \cite{ MR2074770, MR2353701, MR2482215, MR2133965, MR2314128, MR2423533, MR2020214, MR1845218, MR1991495, MR1843773, MR2195994, MR2030207, MR2290140, ALim:2007,MR2285584, MR2386721, MR2006190, MR1991497, MR2276523, MR2401870}. A wealth of literature pertaining to Eq. (\ref{eqn.heurpath}), which considers a product measure on $M^n$, is also available, see \cite{SKCheng1, MR799932, MR1714351, MR780661, Um1} for a short list. The ``derivation'' of Eq. \ref{eqn.heurpath} follows from an application of the Trotter product formula, which the interested reader is directed to \cite{Cachia2001, Cachia2002,  Exner2011, Ichinose2002, Ichinose2004, MR1627505, Neidhardt1999b, Neidhardt1999c, Neidhardt1999a} for reference. 
	
	In \cite{AndDrive:1999}, Andersson and Driver consider two natural geometric schemes for approximating Wiener measure which result from $L^2$ and $H^1$ Riemann sum type metrics on $\Hp(M)$, 
\begin{align*}
	&S_0(X,Y) = \sum_{i=1}^n g\left( X(s_{i}), Y(s_{i}) \right) \Delta_i s, \\  
	&S_1(X,Y) = \sum_{i=1}^n g\left( \frac{\nabla}{ds} X(s_{i-1}+), \frac{\nabla}{ds}Y(s_{i-1}+) \right).
\end{align*}
 Another natural geometric scheme approximating Wiener measure is developed in  \cite{ALim:2007}, which results from the  $H^1$ metric,
	\[ G_1(X,Y) = \int_0^1 g\left( \frac{\nabla}{ds} X(s), \frac{\nabla}{ds} Y(s) \right) ds. \]

	It has been asserted that the correct form of the quantization of the Hamiltonian $\frac{1}{2}g^{ij} p_i p_j + V$ is given by $- \hbar^2(\frac{1}{2}\nabla - \tau \Scal)+V$ where $\hbar$ is Planck's constant and $\tau\in \re$ is a constant which depends on the interpretation of the path integral. For example, in \cite{AndDrive:1999}, $\tau = 0$ or $\tau=\frac{1}{6}$. Our work gives the value $\tau = (2 + \sqrt{3})/(20\sqrt{3})$. However, in \cite{ALim:2007}, Lim derives a form that is dissimilar and does not lend itself to the Feynman-Kac formula for interpretation of the quantized Hamiltonian.
	
\section{Background and Notation}
\subsect{More on Wiener Spaces}
We have already introduced the Wiener space $W(M)$ in Eq. (\ref{eqn.WM}) as well as the space of piecewise geodesics $\Hp(M)$ in Eq. (\ref{eqn.HpM}). It is well known that the Wiener measure on $W(\re^d)$ is the law of a $\re^d$-valued Brownian motion, and conversely, the evaluation maps $b_s(\w) = \w(s)$ on $W(\re^d)$ are an $\re^d$-valued Brownian motion under the Wiener measure. The analogous statements can be said for the Wiener measure on $W(M)$ and an $M$-valued Brownian motion, although we do not explore this further. The interested reader is referred to \cite{MR2090752, MR675100, MR1882015} for the definition and treatment of a manifold-valued Brownian motion.

In what follows we use the symbols $\mu$ and $\nu$ to denote the Wiener measures on $W(\bbR^d)$ and $W(M)$ respectively. Although we will consider several probability spaces, the symbol $\E$ will be used solely for expectation on the probability space $(W(\re^d), \mu)$. Further, we reserve $(b_s : s \in [0,1])$ as the $\re^d$-valued Brownian motion defined as the evaluation maps on $W(\re^d)$.

The piecewise approximation of Brownian motion with respect to the partition $\p$ are the maps $\bp_s : W(\re^d) \to \Hp(\re^d)$  with $s \in [0,1]$ given by,
\begin{align}
b^{\p}_s &:= \sum_{i=1}^n 1_{J_i}(s)\left[ \frac{\Delta_i b}{\Delta_i s}(s - s_{i-1}) + b_{s_{i-1}} \right]. 
\label{eqn:bp}
\end{align}
Here and forevermore $\Delta_i b = b_{s_i} - b_{s_{i-1}}$, $\Delta_i s = s_i - s_{i-1}$, and $J_i = (s_{i-1}, s_{i}]$ when $i > 1$ and $J_1 = [0, s_1)$. It is important to note that $b_s |_{\Hp(\re^d)} = \bp_{s} |_{\Hp(\re^d)}$.

This is a convenient place to introduce the Cameron-Martin subspace $H(M)$ of the Wiener space, which is the collection of absolutely continuous paths with finite energy,
\begin{align}
H(M) = \{ \sigma \in W(M) : \sigma \text{ is absolutely continuous}, \int_0^1 \| \sigma'(s) \|^2 ds < \infty \}. \label{eqn.HM}
\end{align}
The Cameron-Martin space is a Hilbert space and $(i, H(\re^d), W(\re^d))$ is the prototype for an abstract Wiener space, where $i : H(\re^d) \to W(\re^d)$ is the canonical injection. The full beauty of abstract Wiener spaces will not come to light in this paper, but for a short list of references, see \cite{MR0212152, MR1267569}. Volumes of work have also come to move these ideas onto the manifold setting, for example \cite{MR2090752, MR675100, MR1882015, MR1070361}.

\subsect{Geometric Basics}
\label{section:GeometricBasics}
For a path $\sigma \in H(M)$ and $s \in [0,1]$, we use the symbol $//_s(\sigma) : T_o M \to T_{\sigma(s)}M$ to represent parallel translation along $\sigma$ with respect to $\nabla$. Further, for a vector field $X$ along $\sigma$, define $\frac{\nabla}{ds}$ by,
\begin{align*}
\frac{\nabla}{ds} X(s) &:= //_s(\sigma) \frac{d}{ds} \left\{ //^{-1}_s(\sigma)X(s) \right\}.%\label{eqn:covariantderiv}
\end{align*}
The curvature tensor $R$ on $M$ is defined by
\begin{align*}
R(X,Y)Z := \nabla_X \nabla_Y Z - \nabla_Y \nabla_X Z - \nabla_{[X,Y]} Z
\end{align*}
for vector fields $X, Y$ and $Z$ on $M$. The Ricci tensor is then defined as
\begin{align*}
\Ric(X) := \sum_{i=1}^d R(X,e_i)e_i
\end{align*}
for the vector field $X$ on $M$ and orthonormal frame $\{ e_i \}_{i=1}^d$. The scalar curvature on $M$ is given by
\begin{align*}
\Scal := \sum_{i=1}^d g( \Ric(e_i),e_i ).
\end{align*}
Notice that for a given $p \in M$, $\Ric|_{T_p(M)}$ is a linear map $T_p(M) \to T_p(M)$. Therefore, $\Scal(p) = \tr(\Ric|_{T_p(M)})$.

We fix an isometry $u_0 : \re^d \to T_o(M)$ and from henceforth identify $T_o(M)$ with $\re^d$. Some of the work of this paper will be translating statements between the spaces $W(M), H(M)$, and $H_{\p}(M)$ and the spaces $W(\re^d), H(\re^d)$, and $H_{\p}(\re^d)$. In doing so many proofs become tractable; however, this does lead us to introduce more notation.

\begin{nota}
\label{nota:geometric}
If $\pi: TM \to M$ is the projection, $f : \re^d (= T_o (M)) \to T_p M$ is an isometry, we define
\begin{enumerate} 
\item $\W_f : \re^d \times \re^d \times \re^d \to \re^d$ by $\W_f (a,b) c = f^{-1} R(f a, f b) f c$.
\item $\Ric_f :\re^d \to \re^d$ is the linear map defined by $\Ric_f(v) = \sum_{i=1}^d \W_f (v, \varepsilon_i) \varepsilon_i$ where $\{ \varepsilon_i \}_{i=1}^d$ is an orthonormal basis for $\re^d$.
\item Given a curves $h:[0,1] \to \re^d$ and $\sigma:[0,1] \to M$, we define the vector field $X_{\sigma}^h$ along $\sigma$ by $X_{\sigma}^h(s) = //_s(\sigma)h(s)$.
\end{enumerate}
\end{nota}

As previously mentioned, the tangent space $T_{\sigma} H_{\p}(M)$ is identified with the continuous piecewise Jacobi fields along $\sigma$. The following proposition is a statement of this fact using the notation introduced in this section.
\begin{prop} 
\label{prop.HpJacobi}
Let $\sigma \in \Hp(M)$ and $X \in T_{\sigma}H(M)$. Then, $X \in T_{\sigma}\Hp(M)$ if and only if $X$ satisfies,
\begin{align}
\frac{\nabla^2}{ds^2}  X(s) = R(\sigma'(s), X(s))\sigma'(s), \label{eqn:JacobisEqn}
\end{align}
on $[0,1]\backslash \p$. 
\end{prop}
\begin{proof} This is a direct consequence of the fact that $\Hp(M)$ consists of piecewise geodesics.
\end{proof}

\subsect{Cartan's Development Map}
\label{section:CartanDevelopment}
%\begin{defi}
%\label{defi:cartan}
Cartan's development map is a diffeomorphism $\phi:H(\re^d) \to H(M)$, where $\sigma = \phi(\w)$ with,
\[ \sigma'(s) = //_s(\sigma) \w'(s), \qquad \sigma(0)=o. \label{eqn:Development} \] 
%\end{defi}
By smooth dependence on parameters, $\phi$ is smooth, and uniqueness of solutions implies that $\phi$ is injective. 
%\begin{defi}
%\label{defi:anticartan}
The anti-development map, $\phi^{-1} : H(M) \to H(\re^d)$ is defined by $\w = \phi^{-1}(\sigma)$ where
\[ \w(s) = \int_0^s //_r^{-1}(\sigma)\sigma'(r)dr. \]
%\end{defi}
By similar arguments as above, $\phi^{-1}$ is smooth and injective and shows that $\phi:H(\re^d) \to H(M)$ is a diffeomorphism, as asserted.  We also make note of the following important facts,
\begin{enumerate}
\item $\phi$ is a bijection between $H_{\p}(\re^d)$ and $H_{\p}(M)$, $\phi(H_{\p}(\re^d)) = H_{\p}(M)$.
\item This further implies that $T_{\sigma}H_{\p}(M)$ is an embedded submanifold of $T_{\sigma}H(M)$, since $T_{\sigma}H_{\p}(\re^d)$ is an embedded submanifold of $T_{\sigma}H(\re^d)$.
\end{enumerate}
A more detailed account of the development map can be found in \cite{MR2090752, MR1882015}. Throughout the remainder, the symbol $\phi$ will be reserved solely for the development  map.

\section{A Previous Result and Consequences}
In this section we introduce the measure $\nuS{\p}$ in Eq. (\ref{eqn.nuS}) and $\muS{\p}$, where $\muS{\p}$ is simply realization of $\nuS{\p}$ in the flat case $M = \re^d$. To prove Theorem \ref{thm.mainthm}, it is our approach is to compare the measure $\nuG{\p}$ with $\nuS{\p}$. To this end we will define the derivative $\rho_{\p}$  in Eq. (\ref{eqn.firstrhop}) and its companion map $\rhot_{\p}$ in Eq. (\ref{eqn.firstrhotp}), which reduces many of our calculations involving $\rho_{\p}$ to more tractable ones in flat space. However, before arriving at that point we first examine some properties of $\nuS{\p}$ and give size estimates that will be needed later. 

\subsect{The Measure $\nuS{\p}$} In \cite{AndDrive:1999}, Andersson and Driver give the approximation to the Wiener measure $\nuS{\p}$ defined  analogously to Eq. (\ref{eqn.nuG}),
\begin{align}
d\nuS{\p} = \frac{1}{Z_{S_{\p}}} e^{ -\frac{1}{2} \int_0^1\| \sigma'(s) \|^2 ds } \dVol_{S_{\p}}. \label{eqn.nuS}
\end{align}
Here the metric $S_{\p}$ on $\Hp(M)$ given by,
\begin{align}
S_{\p}(X,Y) =  \sum_{i=1}^n g\left( \frac{\nabla}{ds}X(s_{i-1}+), \frac{\nabla}{ds}Y(s_{i-1}+) \right) \Delta_i s, \label{eqn.Sp}
\end{align}
and normalization constant,
\begin{align}
Z_{S_{\p}} = (2 \pi)^{nd/2}. \label{eqn.ZSp}
\end{align}
With this approximation to the Wiener measure, they prove the following theorem.

\begin{thm}[{\cite[Theorem 1.8]{AndDrive:1999}}]
\label{thm.AndDrive}
 If $f:W(M) \to \re$ is bounded and continuous then,
\begin{align}
\lim_{|\p| \to 0} \int_{\Hp(M)} f d\nu_{S_{\p}} = \int_{W(M)} f d\nu.
\end{align}
\end{thm}

 As previously mentioned, we will distinguish the measure $\muS{\p}$ on $\Hp(\re^d)$ as the flat case realization of $\nuS{\p}$. From here we state as a Theorem another fact proved in  \cite[Theorem 4.10 and Corollary 4.13]{AndDrive:1999}. 
\begin{thm}
\label{thm:lawofbp} 
For $\bp$ defined in Eq. (\ref{eqn:bp}), $\muS{\p} = \Law_{\mu}(\bp)$ and $\nuS{\p} = \Law_{\mu}(\phi(\bp))$. In particular, $\muS{\p}$ is the pullback of $\nuS{\p}$ by $\phi$, $\muS{\p} = \phi^{*}\nuS{\p}$. That is, for any Borel set $A \subset \Hp(\re^d)$, $\muS{\p}(A) = \nuS{\p}(\phi(A))$.
\end{thm}

\begin{rmk*}
Using Eq \ref{eqn:bp}, $\Delta_i \bp := \bp_{s_i} - \bp_{s_{i-1}} = \Delta_i b$.  This fact will be used  in future calculations alongside Theorem \ref{thm:lawofbp} as follows. If $f : (\bbR^d)^n \to \bbR^d$ is integrable, then 
\begin{align*} 
\int_{\Hp(\bbR^d)} f( \Delta_1 b, ..., \Delta_n b) d\muS{\p} = \int_{\Hp(\bbR^d)} f( \Delta_1 \bp, ..., \Delta_n \bp) d\muS{\p} \\
= \int_{W(\bbR^d)} f( \Delta_1 b, ..., \Delta_n b) d\mu = \E[f(\Delta_1 b, ..., \Delta_n b)].
\end{align*}
\end{rmk*}

Before exploring some of the consequences of this theorem, we need to agree upon some notation. 
\begin{nota}
Define the map $\up_s$ on $W(\re^d)$ by,
\begin{align}
\up_s := //_s(\phi(\bp)). \label{eqn:up}
\end{align}
That is, for $\w \in W(\re^d)$ and $\sigma^{\p}:= \phi(\bp(\w))$,  $\up_s(\w) : T_o M \to T_{\sigma^{\p}(s) } M$ is the linear isometry given by,
\begin{align}
\up_s(\w) = //_s(\sigma^{\p}).
\end{align}
In turn this let us define the random variables $\fancyR_{\p}, \fancyS_{\p} :W(\re^d) \to \re$ by
\begin{align}
\fancyR_{\p} &= \sum_{i=1}^n \langle \Ric_{\up_{s_{i-1}}} \Delta_i b, \Delta_i b \rangle, \label{eqn:fancyRp}\\
\fancyS_{\p} &= \sum_{i=1}^n \Scal( \phi(\bp)|_{s_{i-1}})\Delta_i s, \label{eqn:fancySp}.
\end{align}
\end{nota} 
\noindent Here  $\phi(\bp)|_{s_{i-1}}(\w) := \phi( \bp(\w))(s_{i-1})$. A direct consequence of these definitions is that $\tr( \Ric_{\up_s} ) = \Scal(\phi(\bp)|_{s})$.

\begin{lem}
\label{lem:fancyRS1}
For $p \in \re$  there is a constant $C$ depending only on $\ddim$, $p$, and the bound on the curvature of $M$ such that 
\begin{align}
1 \leq \int_{\Hp(\re^d)} e^{p(\fancyR_{\p} - \fancyS_{\p})} d\muS{\p}  \leq e^{C | \p |},
\end{align}
immediately implying that 
\begin{align}
\left| \int_{\Hp(\re^d)} (e^{p(\fancyR_{\p} - \fancyS_{\p})} - 1) d\muS{\p} \right|  \leq e^{C | \p |} - 1.
\end{align}
\end{lem}
\begin{proof} Since $\fancyR_{\p} = \fancyR_{\p}(\bp)$ and $\fancyS_{\p} = \fancyS_{\p}(\bp)$, Theorem \ref{thm:lawofbp} implies, 
\begin{align*}
\int_{\Hp(\re^d)}e^{p(\fancyR_{\p} - \fancyS_{\p})} d\muS{\p} = \E\left[ e^{p(\fancyR_{\p} - \fancyS_{\p})} \right].
\end{align*}
The result then follows as a direct application of Proposition \ref{prop.ItoTraceOnlyImportant} below.
\end{proof}

\begin{lem}
\label{lem:fancyS}
For $p \in \re$  there is a constant $C$ depending only on $\ddim$, $p$, and the bound on the curvature of $M$ such that 
\begin{align}
\int_{\Hp(\re^d)}  \left(\ex\left\{p\left|\fancyS_{\p}(\w) - \int_0^1 \Scal(\phi(\w)(s)) ds\right| \right\} - 1 \right) d\muS{\p}(\w)  \leq C |\p|^{1/2}.
\end{align}
\end{lem}
\begin{proof} There is a bound $\Lambda = \Lambda(\text{curvature}) < \infty$ such that for $\w \in \Hp(\re^d)$,
\begin{align*}
\left| \fancyS_{\p}(\w) - \int_0^1 \Scal(\phi(\w)(s)) ds \right| &= \left| \sum_{i=1}^n \int_{s_{i-1}}^{s_i} \left\{ \Scal(\phi(\w)(s)) - \Scal(\phi(\w)(s_{i-1})) \right\}ds\right| \\
&\leq  \sum_{i=1}^n \int_{s_{i-1}}^{s_i} \left| \Scal(\phi(\w)(s)) - \Scal(\phi(\w)(s_{i-1})) \right| ds \\
& \leq \Lambda  \sum_{i=1}^n \int_{s_{i-1}}^{s_i} \left\| \w(s) - \w(s_{i-1}) \right\| ds \\
&= \Lambda  \sum_{i=1}^n \int_{s_{i-1}}^{s_i} \left\| \bp_s(\w) - \bp_{s_{i-1}}(\w) \right\| ds \\
&\leq  \Lambda  \sum_{i=1}^n \Delta_i s \left\| \Delta_i b \right\| (\w), 
\end{align*}
where we recall that $\bp$ is the identity on $\Hp(\re^d)$.
From here,
\begin{align*}
&\int_{\Hp(\re^d)}  \left(\ex\left\{p\left|\fancyS_{\p}(\w) - \int_0^1 \Scal(\phi(\w)(s)) ds\right| \right\} - 1 \right) d\muS{\p}(\w) \\
&\leq  p \Lambda \sum_{i=1}^n \Delta_i s \int_{\Hp(\re^d)} \| \Delta_i b \| \ex\left\{p \Lambda \sum_{j=1}^n \Delta_j s \| \Delta_j b \| \right\} d\muS{\p}\\
& = p \Lambda \sum_{i=1}^n \Delta_i s  \E \left[ \| \Delta_i b \| \ex\left\{p \Lambda \sum_{j=1}^n \Delta_j s \| \Delta_j b \| \right\}   \right] \\
& = p \Lambda \sum_{i=1}^n \Delta_i s  \E \left[ \| \Delta_i b \| e^{ p \Lambda \Delta_i s \| \Delta_i b \|} \right]  \prod_{j \neq i} \E \left[e^{p\Lambda \Delta_j s \| \Delta_j b \| }   \right] \\
& = p \Lambda \sum_{i=1}^n (\Delta_i s)^{3/2}  \E \left[ \| b_1 \| e^{ p \Lambda (\Delta_i s)^{3/2} \| b_1 \|} \right]  \prod_{j \neq i} \E \left[e^{p\Lambda (\Delta_j s)^{3/2} \| b_1 \| }   \right] \\
&\leq p \Lambda |\p|^{1/2} \E \left[ \| b_1 \| e^{ p \Lambda |\p|^{3/2} \| b_1 \|} \right]  \prod_{i=1}^n \left( 1 + p \Lambda (\Delta_i s)^{3/2} \E \left[\|b_1\|e^{p\Lambda (\Delta_i s)^{3/2} \| b_1 \| }   \right] \right) \\
& \leq C |\p|^{1/2}.
\end{align*}
Here the first and penultimate inequality come from Eq.\,(\ref{eqn.eaminus1}), and the final inequality follows from the fact that 
\begin{align*}
 &\prod_{i=1}^n \left( 1 + p \Lambda (\Delta_i s)^{3/2} \E \left[\|b_1\|e^{p\Lambda (\Delta_i s)^{3/2} \| b_1 \| }   \right] \right) \\
 &\qquad \leq \ex\left\{ p \Lambda |\p|^{1/2} \E\left[ \| b_1 \| e^{p \Lambda |\p|^{3/2} \| b_1 \|}\right]\right\}.
\end{align*}
\end{proof}

\subsect{The Maps $\rho_{\p}$ and $\rhot_{\p}$}
\label{subsection.rhop} The final result discussed in this section introduces the maps $\rho_{\p}$ and $\rhot_{\p}$ which are a major focus throughout the sequel. Given the measure $\nuG{\p}$ in Eq. (\ref{eqn.nuG}), we let $\rho_{\p} : \Hp(M) \to \re$ be the Lebesgue-Radon-Nikodym derivative with respect to $\nuS{\p}$,
\begin{align}
d\nuG{\p} = \rho_{\p} d\nuS{\p}. \label{eqn.firstrhop}
\end{align}
Given how our measures are defined,
\begin{align}
\rho_{\p} = \frac{Z_{S_{\p}}}{Z_{G_{\p}}}\sqrt{\frac{\det\left(G_{\p}(X_i, X_j)\right)}{\det\left(S_{\p}(X_i, X_j)\right)}} \label{eqn.rhop2}
\end{align}
with $\{ X_i \}$ a frame in $T\Hp(M)$.
We then define $\rhot_{\p}: W(\re^d) \to \re$ as
\begin{align}
\rhot_{\p} = \rho_{\p} ( \phi(\bp) ).\label{eqn.firstrhotp}
\end{align}
The usefulness of introducing $\rhot_{\p}$ is illustrated in the following proposition and is, in essence, why our focus turns to understanding $\rhot_{\p}$ as $|\p|\to 0$.
\begin{prop}
\label{prop:nuGvsmu}
Let $f : W(M) \to \re$ be bounded and continuous. Then, 
\begin{align}
\int_{\Hp(M)} f d\nuG{\p} = \int_{\Hp(\re^d)} f(\phi) \rhot_{\p} d\muS{\p} = \int_{W(\re^d)} f(\phi(\bp)) \rhot_{\p} d\mu.
\end{align}
\end{prop}
\begin{proof}
By the definition of $\rho_{\p}$, 
\begin{align*}
\int_{\Hp(M)} f d\nuG{\p} = \int_{\Hp(M)} f \rho_{\p} d\nuS{\p}.
\end{align*}
By Theorem \ref{thm:lawofbp},
\begin{align*}
\int_{\Hp(M)} f \rho_{\p} d\nuS{\p} = \int_{\Hp(\re^d)} f(\phi) \rho_{\p}(\phi) d\muS{\p},
\end{align*}
and we finish by another application of Theorem \ref{thm:lawofbp} and by noticing that $\rho_{\p}(\phi) = \rho_{\p}(\phi(\bp)) = \rhot_{\p}$ on $\Hp(\re^d)$.
\end{proof}

\section{Setup for the Proof of Theorem \ref{thm.mainthm}}
This paper sets out to find the limit of the measure $d\nuG{\p}$ as $|\p| \to 0$. Based on the definition of $\rho_{\p}$ in Eq. (\ref{eqn.firstrhop}), this problem turns into understanding the limit of $\rho_{\p} d\nuS{\p}$, which leads us to understand how $\rho_{\p}$ behaves in the limit, since Theorem \ref{thm.AndDrive} shows us how to deal with $d\nuS{\p}$. Eq. (\ref{eqn.rhop2})  tells us that we can describe $\rho_{\p}$ if we choose a frame in $T\Hp(M)$ to consider. 

In this section we choose such a frame in Eq. (\ref{eqn:fbasis}) by applying Proposition \ref{prop.HpJacobi2} with the basis given in Eq. (\ref{eqn:Definef}). This then lets us have an explicit description of $\rhot_{\p}$ in Eq. (\ref{eqn.rhotp2}), which essentially  is a description of $\rho_{\p}$. The operator in Eq. (\ref{eqn:Si}) and its inverse are used in the construction of the chosen frame. However, to guarantee that the inverse exists, we make the assumption of non-positivity of the sectional curvature from here forward. 

\begin{assumption}
\label{assumption.nonpossect}
$M$ has non-positive sectional curvature.
\end{assumption}

\subsect{Defining the basis $\{ f_{i,a} \}$}
We start this section by restating Proposition \ref{prop.HpJacobi} in an equivalent form that will be useful. The proof is just a reformulation of the definitions in Notation \ref{nota:geometric}, or for a direct proof, see \cite[Proposition 4.4]{AndDrive:1999}.

\begin{prop}
\label{prop.HpJacobi2}
If $\w = \phi^{-1}(\sigma)$, $u(s) = //_s(\sigma)$, $h \in H(\re^d)$, and $X^h_{\sigma} \in T_{\sigma}H(M)$. Then $X^h_{\sigma} \in T_{\sigma}H_{\p}(M)$ if and only if 
\begin{align}
h''(s) = \W_{u(s)}(\w'(s), h(s))\w'(s). \label{eqn.FlatJacobi}
\end{align}
\end{prop}

Following the notation of Proposition \ref{prop.HpJacobi2}, if $\w \in \Hp(\bbR^d)$ then $\w'(s) = \frac{\Delta_i \w}{\Delta_i s}= \frac{\Delta_i b(\w)}{\Delta_i s}$ when $s \in (s_{i-1},s_i)$. Moreover, $u(s) = //_s( \phi(\w)) = \up_s(\w)$, where $\up$ is defined in Eq. (\ref{eqn:up}). Hence, for each $i \in \{1, ..., n\}$ and $s \in (s_{i-1}, s_i)$, we can rewrite Eq. (\ref{eqn.FlatJacobi}) as,
\begin{align}
h''(s) &= \W_{\up_s (\w)}\left( \frac{\Delta_i b(\w)}{\Delta_i s}, h(s) \right)  \frac{\Delta_i b(\w)}{\Delta_i s} \label{eqn:FlatJacobi2}.
\end{align}
This motivates the definition of the following operators for $1\leq i \leq n$,
\begin{align}
\Ap{i}(s) := \W_{\up_{s+s_{i-1}}}\left( \frac{\Delta_i b}{\Delta_i s}, \cdot \right)  \frac{\Delta_i b}{\Delta_i s} \label{eqn:Ai}
\end{align}
when $s \in (s_{i-1}, s_i)$. Using this notation,
\begin{align}
\tr(\Ap{i}(0)) = - \left\langle \Ric_{\up_{s_{i-1}}} \frac{\Delta_i b}{\Delta_i s}, \frac{\Delta_i b}{\Delta_i s} \right\rangle.
\label{eqn:trAp}
\end{align}
Moreover, the assumption that the curvature and its derivative are bounded on $M$ is equivalent to the existence of some $\kappa < \infty$ such that 
\begin{align}
\sup_s \| \Ap{i} (s) \| &\leq \kappa \frac{\| \Delta_i b \|^2}{(\Delta_i s)^2} \label{eqn:Aibound} \\
\sup_s \| \frac{d}{ds}\Ap{i} (s) \| &\leq \kappa \frac{\| \Delta_i b \|^3}{(\Delta_i s)^3}. \label{eqn:Aipbound}
\end{align}

With this in mind, for each $s \in [0,\Delta_i s]$ with $1 \leq i \leq n$, consider the differential equation, 
\begin{align}
\frac{d^2}{ds^2} Z^{\p}_i(s) = \Ap{i}(s) Z^{\p}_i(s) \label{eqn:OpFlatJacobi}
\end{align}
with $Z^{\p}_i(s) : \re^d \to \re^d$ a linear map. Applying existence and uniqueness of ordinary differential equations, we define the following solutions to Eq. (\ref{eqn:OpFlatJacobi}):
\begin{align}
\Sp{i} &= Z_i^{\p} \text{ with initial conditions } \Sp{i}(0) = 0, \frac{d}{ds}\Sp{i}(0)=I \label{eqn:Si},\\
C_i^{\p} &= Z_i^{\p} \text{ with initial conditions } C_i^{\p}(0) = I, \frac{d}{ds}C_i^{\p}(0)=0 \label{eqn:Ci},
\end{align}
and for $1 \leq i < n$,
\begin{align}
\Vp{i+1} &= Z_{i+1}^{\p} \text{ with initial conditions } \Vp{i+1}(0) = \Sp{i}(\Delta_{i} s), \frac{d}{ds}\Vp{i+1}(0) = - \Fp{i} \label{eqn:ViODE}
\end{align}
where
\begin{align}
\Fp{i} = \Sp{i+1}(\Delta_{i+1}s)^{-1} \Cp{i+1}(\Delta_{i+1}s) \Sp{i}(\Delta_i s) \label{eqn:Fi}.
\end{align}
With the above definitions, 
\begin{align}
\Vp{i+1}(s) = \Cp{i+1}(s) \Sp{i}(\Delta_i s) - \Sp{i+1}(s) \Fp{i}. \label{eqn:Vi}
\end{align}

The fact that $\Sp{i}(\Delta_i s)$ has an inverse is not immediate. However, we are guaranteed the inverse exists when we restrict ourselves to manifolds of non-positive sectional curvature.

We now define the maps $\{ f^{\p}_{i,a} : 1\leq i \leq n, ~ 1\leq a \leq d \}$ which satisfy Eq. (\ref{eqn.FlatJacobi}) with the boundary conditions 
\begin{align}
f^{\p}_{i,a}(0) &= 0\\
\frac{d}{ds}f^{\p}_{i,a}(s_j+) &=
\begin{cases}
e_a & j = i-1 \\
-\Fp{i} e_a & j= i \\
0 & \text{otherwise}.
\end{cases}
\end{align}
Using  Eq. (\ref{eqn:Si}), Eq. (\ref{eqn:Ci}), and Eq. (\ref{eqn:Vi}),
\begin{align}
f^{\p}_{i,a}(s) = 1_{J_i}(s)\Sp{i}(s-s_{i-1})e_a + 1_{J_{i+1}}(s)\Vp{i+1}(s-s_i)e_a \label{eqn:Definef}
\end{align}
where $\{ e_a  \}_{a=1}^d$ is the standard basis for $\re^d$.
This set of maps induces a basis 
\begin{align}
\F_{\p} := \left\{ X_{\phi(\bp)}^{f^{\p}_{i,a}} : 1 \leq i \leq n,~ 1 \leq a \leq d  \right\} \label{eqn:fbasis}
\end{align} of $T_{\phi(\bp)}H_{\p}(M)$ where the meaning of $X_{\sigma}^{h}$ was established in Notation \ref{nota:geometric}.

\subsect{The matrix $\GFp$}
Define $\GFp$ as the $n\times n$ block diagonal matrix with $d \times d$ blocks given by
\begin{align}
\GFp &:= \left[ G_{\p}\left( X_{\phi(\bp)}^{f^{\p}_{i,a}} , X_{\phi(\bp)}^{f^{\p}_{j,c}}\right): 1\leq i, j \leq n,~ 1\leq a, c \leq d \right]. \label{eqn:GfMatrix}
\end{align}
That is, $\GFp$ is the matrix representation of the metric $G_{\p}$ under the basis $\F_{\p}$ in Eq. (\ref{eqn:fbasis}). We write the $(i,j)^{th}$ block of $\GFp$ as
\begin{align}
\left[ \GFp \right]_{i,j} := \left[ G_{\p}\left( X_{\phi(\bp)}^{f^{\p}_{i,a}} , X_{\phi(\bp)}^{f^{\p}_{j,c}}\right): 1\leq a, c \leq d \right].
\end{align}
 
Let $i,j \in \{1, ..., n\}$, $a,c \in \{1, ..., d\}$, and define $S_{n+1}^{\p} = V_{n+1}^{\p} \equiv 0$, 
\begin{align*}
G_{\p}\left(  X_{\phi(\bp)}^{f^{\p}_{i,a}},  X_{\phi(\bp)}^{f^{\p}_{j,c}} \right) &= \int_0^1 g\left( X_{\phi(\bp)}^{f^{\p}_{i,a}}(s), X_{\phi(\bp)}^{f^{\p}_{j,c}}(s) \right) ds \\
&= \int_0^1 \left\langle f^{\p}_{i,a}(s), f^{\p}_{j,c}(s) \right\rangle ds
\end{align*}
where $\langle \cdot, \cdot \rangle$ indicates the inner-product on $T_o(M) = \re^d$ defined by $g$. Using Eq. (\ref{eqn:Definef}),
\begin{align*}
&\int_0^1 \left\langle f^{\p}_{i,a}(s), f^{\p}_{j,c}(s) \right\rangle  ds \\
&\qquad  = \delta_{i,j} \left\langle e_a, \left[ \int_0^{\Delta_i s}(\Sp{i}(s)^{tr} \Sp{i}(s) + \Vp{i+1}(s)^{tr}\Vp{i+1}(s) )ds \right]  e_c \right\rangle \\
&\qquad \qquad   + \delta_{i,j+1} \left\langle e_a, \left[ \int_0^{\Delta_i s} \Sp{j+1}(s)^{tr} \Vp{j+1}(s) ds\right] e_c \right\rangle \\
&\qquad  \qquad+ \delta_{i+1,j} \left\langle e_a, \left[  \int_0^{\Delta_{i+1}s} \Vp{i+1}(s)^{tr} \Sp{i+1}(s) ds\right] e_c \right\rangle,
\end{align*}
which implies that,
\begin{align}
\GFp &=
\left(
\begin{array}{cccc}
D_1 & M_2 & 0 & 0 \\
M_2^{tr} & D_2 &  M_3 & 0 \\
0 & \ddots & \ddots & M_n \\
0 & 0 & M_n^{tr} & D_n
%\int_0^{\Delta} \left(S_1^{tr}(s)S_1(s)+V_2^{tr}(s)V_2(s) \right)ds & \int_0^{\Delta} V_2^{tr}(s) S_2(s) ds & 0 & 0 \\
%\int_0^{\Delta} S_2^{tr}(s) V_2(s) ds & \int_0^{\Delta} \left(S_2^{tr}(s)S_2(s)+V_3^{tr}(s)V_3(s) \right)ds & \int_0^{\Delta} V_3^{tr}(s) S_3(s) ds & 0 \\
%0 & \ddots & \ddots & 0 \\
%0 & 0 & \int_0^{\Delta} S_n^{tr}(s) V_n(s) ds & \int_0^{\Delta} S_n^{tr}(s) S_n(s) ds
\end{array}
\right)
\label{eqn:GFp}
\end{align}
where 
\begin{align}
D_i &= 
\begin{cases}
\int_0^{\Delta_i s} \Sp{i}(s)^{tr}\Sp{i}(s)ds + \int_0^{\Delta_{i+1} s}\Vp{i+1}(s)^{tr}\Vp{i+1}(s) ds & 1\leq i < n \\
\int_0^{\Delta_n s} \Sp{n}(s)^{tr}\Sp{n}(s)ds & i = n
\end{cases}
\label{eqn:GFpDiag}
\end{align}
and 
\begin{align}
M_{i} &= \int_0^{\Delta_i s} V^{\p}_{i}(s)^{tr} \Sp{i}(s) ds \quad 2 \leq i \leq n.
\label{eqn:GFpOffDiag}
\end{align}

In the case that $M = \re^d$, Eq. (\ref{eqn:Definef}) greatly simplifies to, 
\begin{align}
f_{i,a} = \left\{ 1_{J_i}(s)(s-s_{i-1}) + 1_{J_{i+1}}(s)(s_{i+1} - s)  \right\} e_a, \label{eqn:fiaflat}
\end{align}
and hence $\GFp(\re^d)$ is given by the block matrix,
\begin{align}
 \left(
\begin{array}{cccc}
\frac{1}{3}[(\Delta_1 s)^3 + (\Delta_{2} s)^3] I& \frac{1}{6}(\Delta_2 s)^3 I & 0 & 0 \\
\frac{1}{6}(\Delta_2 s)^3 I  & \frac{1}{3}[(\Delta_2 s)^3 + (\Delta_{3} s)^3] I &  \frac{1}{6}(\Delta_3 s)^3 I & 0 \\
0 & \ddots & \ddots & \frac{1}{6}(\Delta_n s)^3 I  \\
0 & 0 & \frac{1}{6}(\Delta_n s)^3 I & \frac{1}{3}(\Delta_n s)^3 I
%\int_0^{\Delta} \left(S_1^{tr}(s)S_1(s)+V_2^{tr}(s)V_2(s) \right)ds & \int_0^{\Delta} V_2^{tr}(s) S_2(s) ds & 0 & 0 \\
%\int_0^{\Delta} S_2^{tr}(s) V_2(s) ds & \int_0^{\Delta} \left(S_2^{tr}(s)S_2(s)+V_3^{tr}(s)V_3(s) \right)ds & \int_0^{\Delta} V_3^{tr}(s) S_3(s) ds & 0 \\
%0 & \ddots & \ddots & 0 \\
%0 & 0 & \int_0^{\Delta} S_n^{tr}(s) V_n(s) ds & \int_0^{\Delta} S_n^{tr}(s) S_n(s) ds
\end{array}
\right).
\label{eqn:Lp}
\end{align}
Since this matrix will prove important to understand throughout the remainder, we use the notation $\mLp := \GFp(\re^d)$. 

 \subsect{A simplified expression for $\rhot_{\p}$}
We now return to the non-flat case.  If we define $\SFp$ analogously to $\GFp$, where $\SFp$ is the matrix representation of the metric $S_{\p}$ using the basis $\F_{\p}$, then we have, 
\begin{align}
\rhot_{\p} = \rho_{\p}(\phi(\bp)) &= \frac{Z_{S_\p}}{Z_{G_{\p}}}\sqrt{\frac{ \det \left( \GFp \right)}{ \det \left( \SFp \right)}}. %\label{eqn:Definerhot}.
\end{align}

An analogous argument as that in \cite[Theorem 4.7]{ALim:2007} shows that $\det(\SFp) = \prod_{i=1}^n (\Delta_i s)^d$. This in combination with Eqs. (\ref{eqn.ZGp}) and (\ref{eqn.ZSp}) yields, 
\begin{align}
\rhot_{\p} = \sqrt{\frac{ \det \left( \GFp \right)}{ \det (\mLp)}} \label{eqn.rhotp2}.
\end{align}

\section{Properties of $\mLp$}
\label{section:PropertiesOfLp}
The main theorem of this paper, Theorem \ref{thm.mainthm}, has the constant $\frac{2 + \sqrt{3}}{20\sqrt{3}}$ multiplied by the scalar curvature term once the limit has been taken. While the physical interpretation of this constant is open, that constant largely owes its present form to the eigenvalues and eigenvectors of $\mLp$ defined in Eq. (\ref{eqn:Lp}). This section collects together  the necessary properties of $\mLp$ in Theorem \ref{thm:Lpthm} followed by corollaries that give insight into how these properties manifest themselves in the future. 

The analysis of $\mLp$ is greatly simplified by assuming that $\p$ is equally spaced, and so we will make this assumption from here on.

\begin{assumption}
\label{assumption.Pequal}
We will assume that $\p = \{0, 1/n, 2/n, ..., 1\}$ is the equally spaced partition of $[0,1]$ and use the notation $\Delta  := |\p| = 1/n$. 

\end{assumption}
\noindent Under Assumption \ref{assumption.Pequal}, Eq. (\ref{eqn:Lp}) becomes,
\begin{align}
 \mLp = \frac{\Delta^3}{6} \left(
\begin{array}{cccc}
4I&  I & 0 & 0 \\
 I  &4 I &   I & 0 \\
0 & \ddots & \ddots & I  \\
0 & 0 &  I & 2 I
\end{array}
\right). \label{eqn.Lp2}
\end{align} 

\begin{thm}
\label{thm:Lpthm}
Let $\{ e_a : 1 \leq a \leq d \}$ denote the standard basis in $\re^d$. There exists an orthonormal basis of eigenvectors of $\mLp$, $\{ \up_{k,a} : 1\leq k \leq n, 1\leq a \leq d \}$ with,
\begin{align}
\up_{k,a} := \betap_k \left(
\begin{array}{c}
\alpha_k^1 e_a \\
\alpha_k^2 e_a \\
\vdots \\
\alpha_k^n e_a
\end{array}
\right).
\end{align}
 Here, for $1 \leq k < n$, $\alpha_k^m = \sin(m \tp_k)$ with $\{ \tp_k : 1 \leq k \leq n-1\} \subset (0, \pi)$ a monotonically increasing sequence given by,
\begin{align}
\label{eqn.tk}
\tp_k = \frac{\pi(k+r_k)}{n+1}
\end{align}
where there is a smooth map $\varphi : [0,\pi] \to [0,2\pi]$ with $1 \leq \varphi' \leq 4$ such that $r_k = \varphi(\tp_k)/2\pi$. For $k=n$, $\alpha_n^m = \gamma_n^m - \gamma_n^{-m}$ where for $n \geq 2$, $\gamma_n \in (-2, -3/2)$, and $\gamma_n \to -2$ as $n \to \infty$. If $1\leq k <n$, the normalization constants $\betap_k$ are given by
\begin{align}
(\betap_k)^2 = \frac{2}{n}\left( \frac{1}{1-\epsilon_k} \right) \label{eqn.bk}
\end{align}
where $| \epsilon_k | = O\left( 1/n  \right)$ and thusly $\left( \betap_k \right)^2  = 2/n + O(1/n^2)$. For $k=n$, 
\begin{align}
(\betap_n)^2 = \left( - 2(n + 1) + \frac{\gamma_n^{2(n+1/2)} - \gamma_n^{-2(n+1/2)}}{\gamma_n-\gamma_n^{-1}} \right)^{-1}. \label{eqn.bn}
\end{align}
In particular,  $(\betap_n)^2 = O(\gamma_n^{-2n})$. 

The eigenvalues $\lambda^{\p}_{k,a}$, defined so that $\mLp u_{k,a} = \lambda^{\p}_{k,a} \up_{k,a}$, are given by $\lambda^{\p}_{k,a} = \frac{\Delta^3}{3}\left( 2 + \cos(\tp_k) \right) =: \lambda_k^{\p}$  for $1 \leq k <n$, and $\lambda^{\p}_{n,a} = \frac{\Delta^3}{6}(4 + \gamma_n + \gamma_n^{-1}) =: \lambda_n^{\p}$ when $k=n$. This further implies that
\begin{align}
\frac{\Delta^3}{4} \leq \| \mLp \| \leq \Delta^3. \label{eqn.SizeOfLp}
\end{align}

Finally, there exists an upper triangular matrix $\mAp$ such that 
\begin{align}
\mLp = \mAp \mAp^{tr}. \label{eqn:LpLU}
\end{align}
Here $\mAp$ is invertible and 
\begin{align}
\| [ \mAp^{-1}]_{i,j} \|^2 \leq
\begin{cases}
\frac{3}{\Delta^3} \left( \frac{1}{2} \right)^{j-i} & j\geq i \\
0 & j < i
\end{cases}. \label{eqn:LpLUsize}
\end{align}
\end{thm}

\begin{proof}
A first step into understanding $\mLp$ is to write $\mLp = \frac{\Delta^3}{6} L_n \otimes I_d$ where $L_n$ be the $n \times n$ matrix given by 
\begin{align}
L_n = 
\left(
\begin{array}{c c c c c}
4 & 1 & 0 & 0 & 0\\
1 & 4 & 1 & 0 & 0\\
0 & 1 & \ddots & 1 & 0 \\
0 & 0 & 1 & 4 & 1 \\
0 & 0 & 0 & 1 & 2
\end{array}
\right),
\label{eqn:ln}
\end{align}
$I_d$ is the $d \times d$ identity matrix, and $\otimes$ denotes the Kronecker product. 

From here, our work simplifies to understanding $L_n$ which in turn can be understood by writing $L_n = \mD_n + 6 I_n$ with
\begin{align*}
\mD_n &= 
\left(
\begin{array}{c c c c c}
-2 & 1 & 0 & 0 & 0\\
1 & -2 & 1 & 0 & 0\\
0 & 1 & \ddots & 1 & 0 \\
0 & 0 & 1 & -2 & 1 \\
0 & 0 & 0 & 1 & -4
\end{array}
\right),
\end{align*}
and turning our attention to $\mD_n$. We are tempted to do this since $\mD_n$ is a discretization of the Laplacian acting on a functions $f : \{1, 2, ..., n\} \to \bb{C}$ (i.e. $f \in \bb{C}^n$) with appropriate ``boundary conditions.'' Indeed, motivated by the fact that for a twice differentiable function $g$, $g''(x) \approx (\Delta x)^{-2}\{g(x+\Delta x) - 2 g(x) + g(x-\Delta x)\}$ for small $\Delta x$, we consider the expression $(\Delta j)^{-2}\{ f(j + \Delta j) - 2 f(j) + f(j - \Delta j)\}$ for $f \in \bb{C}^n$ and small $\Delta j$. However, the smallest value of $\Delta j$ we can hope for in this case is $\Delta j = 1$, and we therefore define $f''(j) := f(j + 1) - 2 f(j) + f(j-1)$ for $2 \leq j \leq n-1$. Noting that $f''(j) = \mD_n f(j)$ for $2 \leq j \leq n-1$, we now extend  $f''$ to $j=1$ and $j=n$. If we define the boundary conditions $f(0):= 0$ and $f(n+1):=-2f(n)$, then $\mD_n f(1) = -2f(1) + f(2) = f(0) - 2f(1) + f(2) =: f''(1)$, and $\mD_nf(n) = f(n-1)  - 4f(n) = f(n-1) - 2f(n) + f(n+1) =: f''(n)$. 

We now have a clue into the spectral behavior of $\mD_n$. Eigenvectors of the Laplacian should have the form $f(j) = a z^j + b z^{-j}$ for some $a, b, z \in \bb{C}$. Applying the boundary condition $f(0)=0$ yields  $f(j) = a(z^j - z^{-j})$. Further applying the boundary condition $f(n+1) = -2f(n)$ gives $a(z^{n+1} - z^{-(n+1)}) = -2a(z^{n} - z^{-n})$, which is satisfied whenever
\begin{align}
 z^{2(n+1)}+2 z^{2n+1} -2z -1 = 0. \label{eqn.zbdry}
\end{align}
 One can now check that for any $a \in \bb{C}$, $z \in \bb{C}$ satisfying Eq. (\ref{eqn.zbdry}), and $j \in \{1, 2, ..., n\}$,
 \begin{align}
  \label{eqn.Dneigen}
 \mc{D}_nf (j) = f''(j) = a(z^{j+1} - z^{-(j+1)}) - 2a(z^j - z^{-j}) + a(z^{j-1} - z^{-(j-1)}) \notag \\ 
 = a(z + z^{-1} - 2) (z^j - z^{-j}) = (z + z^{-1} -2) f(j).
 \end{align}

The remainder of the proof will now be broken into the following claims.
\begin{claim}
\label{claim.lpc1}
There is a strictly increasing sequence $\{ \tp_1 < \tp_2 < \cdots < \tp_{n-1} \}\subset (0, \pi)$ such that $z_k := e^{i\tp_k}$ satisfies Eq. (\ref{eqn.zbdry}) for each $1 \leq k \leq n-1$. Moreover, each of these $\tp_k$ satisfies Eq. (\ref{eqn.tk}).
\end{claim}
\begin{proof}
If we make the assumption that $z = e^{i \theta}$ for some $\theta \in \re$, Eq. (\ref{eqn.zbdry}) implies $e^{2(n+1)i \theta }(1 + 2 e^{-i \theta}) = (1+ 2 e^{i \theta})$, which rewritten gives, $e^{2(n+1) i \theta} = \zeta / \bar{\zeta}$ with $\zeta = 1+ 2 e^{i \theta}$. Taking advantage of the fact that $ \zeta / \bar{\zeta}$ lies on the unit circle, we can find $\varphi(\theta)$ satisfying 
\begin{align}
e^{i \varphi(\theta)} &= \zeta / \bar{\zeta}  = \frac{1 + 2e^{i \theta}}{1 + 2e^{-i \theta}} = Z\{ 1 + 4(\cos(\theta) + \cos(2\theta)) + i 4 ( \sin(\theta) + \sin(2 \theta)) \} \label{eqn.varphidef}
\end{align}  
where $Z = |\zeta|^{-2} = (5 + 4\cos(\theta))^{-1}$ is the normalization constant ensuring the total magnitude is 1. The right hand side of Eq.\,(\ref{eqn.varphidef}) is smooth as a function of $\theta$ and we can therefore assume that $\varphi$ is smooth. Moreover,
\begin{align*}
i \varphi'(\theta) e^{i \varphi(\theta)} = \frac{d}{d\theta}  \left(\frac{1 + 2 e^{i \theta}}{1+2e^{-i \theta}} \right)= i\left( \frac{4(\cos(\theta)+2)}{5+4\cos(\theta)}\right) \left( \frac{1+2e^{i\theta}}{1+2e^{-i\theta}} \right)
\end{align*}
implying that $\varphi'(\theta) = 4(\cos(\theta)+2)/(5+4\cos(\theta))$. In particular, $\varphi$ is strictly increasing with $1 \leq \varphi' \leq 4$. The only values of $\theta \in [0,\pi]$ giving a value of $1$ when evaluated in Eq.\,(\ref{eqn.varphidef}) are $\theta=0$ and $\theta=\pi$. Hence we can choose $\varphi(0) = 0$, which then implies that $\varphi(\pi) = 2\pi$.

Now, if $\theta \in [0, \pi]$ is chosen such that $e^{2(n+1)i \theta} = e^{i\varphi(\theta)}$, then $z = e^{i \theta}$ satisfies Eq. (\ref{eqn.zbdry}).  This condition is satisfied whenever 
\begin{align}
\theta - \frac{1}{2(n+1)} \varphi(\theta) = \frac{\pi k}{n+1} \label{eqn.thetadef}
\end{align} 
for some $k$. For any $k \in \{1, ..., n-1\}$, 
\[ \frac{\pi k}{n+1} - \frac{1}{2(n+1)} \varphi\left( \frac{\pi k}{n+1} \right) < \frac{\pi k}{n+1} < \frac{\pi(k+1)}{n+1} -\frac{1}{2(n+1)} \varphi\left( \frac{\pi(k+1)}{n+1}\right) \]
since $0 < \varphi(\theta) < 2 \pi$ when $\theta \in (0, \pi)$. Therefore, by the continuity of $\varphi$ and the intermediate value theorem, we are guaranteed a solution $\tp_k$ satisfying Eq. (\ref{eqn.thetadef}) such that $\frac{\pi k}{n+1} < \tp_k < \frac{\pi(k+1)}{n+1}$ for each $1 \leq k \leq n-1$. Rearranging Eq. (\ref{eqn.thetadef}) in terms of $\tp_k$, 
\begin{align}
\tp_k = \frac{\pi k}{n+1} + \frac{\varphi(\tp_k)}{2(n+1)} = \frac{\pi \left(k + [\varphi(\tp_k)/2\pi]\right)}{n+1} =: \frac{\pi(k+r_k)}{n+1} \label{eqn.thetarkdef}
\end{align}
That is, $r_k = \varphi(\tp_k)/2\pi$. This proves the claim.
\end{proof}

\begin{claim}
\label{claim.lpc2}
For $n \geq 2$, there exists some $\gamma_n \in (-2, -3/2)$ solving Eq. (\ref{eqn.zbdry}) such that $\gamma_n \to -2$ as $n \to \infty$.
\end{claim}
\begin{proof}
Manipulating the left hand side of Eq. (\ref{eqn.zbdry}) we find $z^{2(n+1)} + 2z^{2n+1} - 2z - 1 = (z^2 + 2z) z^{2n} - (2z + 1) =: a_n(z)$. Define the map $g_n(z)$ for $z \in (-2, -1)$ by $g_n(z) := z^{2n}(z^2 + 2z) + 3.$ Then $g_n(z) > a_n(z)$ for every $z \in (-2,-1)$, and on this interval, the minimum of $g_n(z)$ is easily found at $z^{(n)}_{min} := -(2n+1)/(n+1)$. Evaluating, 
\begin{align*}
g_n(z^{(n)}_{min}) &= 4^n\left(1 - \frac{1}{2(n+1)}\right)^{2n}\left( 4\left[ 1 - \frac{1}{2(n+1)} \right]^2 - 4 \left[ 1 - \frac{1}{2(n+1)} \right]\right) + 3 \\
&= - \frac{4^{n+1}\left(1 - \frac{1}{2(n+1)}\right)^{2n+1}}{2(n+1)} + 3 \approx -\frac{4^{n+1}}{2e(n+1)} + 3.
\end{align*} 
We can see that for sufficiently large $n$, $g_n(z^{(n)}_{min}) < 0$, implying that for these values of $n$,  $a_n(z^{(n)}_{min}) < 0$. Since $a_n(-2) = +3$, the intermediate value theorem then guarantees a solution $\gamma_n$ to $a_n(z)= 0$ for some $z \in (-2, z^{(n)}_{min})$. Moreover, since $a_n(z)$ is strictly decreasing in $n$ for a fixed $z \in (-2, -1)$, and $a_2(-3/2) < 0$, we can always ensure that $\gamma_n \in (-2, -3/2)$ for $n \geq 2$.
\end{proof}

\begin{claim}
\label{claim.lpc3}
For $1 \leq k \leq n$ and $1 \leq m \leq n$, let $\alpha_{k}^{m}$, $\beta^{\p}_k$, and $\lambda_{k}^{\p}$ be given as in the statement of Theorem \ref{thm:Lpthm}. Define $f_k = \beta_k^{\p} (\alpha^1_k, \alpha^2_k, ..., \alpha^{n}_k)^{tr}$. Then $\{ f_k : 1 \leq k \leq n\}$ is an orthonormal set of eigenvectors of $L_n$ such that $L_n f_k = \frac{6}{\Delta^3} \lambda^{\p}_k f_k$. Moreover, Eqs. (\ref{eqn.bk}) and (\ref{eqn.bn}) hold.
\end{claim}
\begin{proof}
For $1 \leq k \leq n-1$, define $z_k := e^{i \tp_k}$ where $\tp_k$ are given in Claim \ref{claim.lpc1}. Also define $z_n := \gamma_n$ where $\gamma_n$ is given in Claim \ref{claim.lpc2}. Let $\tilde{f_k} \in \bb{R}^n$ be defined by $\tilde{f}_k(m) = a_k (z_k^m - z_k^{-m})$, where $a_k = (2 i)^{-1}$ for $k = 1, 2, ..., n-1$, and $a_n = 1$. Then $\tilde{f}_k(m) = \alpha_k^m$ for each $k, m \in \{1, 2, ..., n\}$. By our choices for $z_k$ and the discussion leading to Eq.\,(\ref{eqn.Dneigen}), $\mD_n \tilde{f}_k = (z_k + z_k^{-1} - 2)\tilde{f}_k$ for each $k$, which implies that $L_n \tilde{f}_k = (z_k + z_k^{-1} + 4) \tilde{f}_k$ since $L_n = \mD_n + 6I_n$. For $1 \leq k \leq n-1$, $z_k + z_k^{-1}+4 = 2\cos(\tp_k) + 4 = \frac{6}{\Delta^3} \lambda_k^{\p}$. Also $z_n + z_n^{-1} + 4 = \gamma_n + \gamma_n^{-1} + 4 = \frac{6}{\Delta^3} \lambda_n^{\p}$. Since $L_n$ is symmetric and the eigenvalues of $\{ \tilde{f}_k \}$ are distinct, then this collection of eigenvectors must be orthogonal. 

By assumption $f_k = \betap_k \tilde{f}_k$. The $\betap_k$ were chosen as normalizing constants, so it only remains to show that   Eqs. (\ref{eqn.bk}) and (\ref{eqn.bn}) hold for each $k$. To this end,
\begin{align*}
&(\betap_k)^{-2} = \| \tilde{f}_k \|^2 \\
&= a_k^2 \sum_{m=1}^n (z_k^{m} - z_k^{-m})^2 = a_k^2 \sum_{m=1}^n (z_k^{2m} + z_k^{-2m} - 2)\\
&= a_k^2 \left( \frac{1-z_k^{2(n+1)}}{1-z_k^2} + \frac{1-z_k^{-2(n+1)}}{1-z_k^{-2}} - 2n - 2 \right) \\
&= a_k^2 \left( -2(n+1) + \frac{z_k^{2(n+1/2)} - z_k^{-2(n+1/2)}}{z_k-z_k^{-1}} \right)
\end{align*}
Eq. (\ref{eqn.bn}) now follows immediately with $z_n = \gamma_n$ and $a_n = 1$. For $k \neq n$,
\begin{align*}
(\betap_k)^{-2} = -\frac{1}{4} \left( -2(n+1) + \frac{e^{i 2(n+1/2) \tp_k} - e^{-i2(n+1/2)\tp_k}}{e^{i\tp_k}-e^{-i \tp_k}} \right) \\
= \frac{n}{2}\left( 1 - \frac{1}{n}\left[ \frac{\sin((2n+1)\tp_k)}{ \sin(\tp_k)}  - 1  \right] \right)
\end{align*}
From Eq.\,(\ref{eqn.tk}), $2(n+1) \tp_k = 2 \pi k + 2 \pi r_k$ and hence
\begin{align*}
\sin((2n+1)\tp_k) = \sin(2(n+1)\tp_k - \tp_k) =  &\sin(2 \pi r_k) \cos(\tp_k) - \cos(2 \pi r_k)\sin(\tp_k).
\end{align*}
 Therefore,
\begin{align*}
\left| \frac{\sin((2n+1)\tp_k)}{\sin(\tp_k)} \right|  = \left| \frac{\sin(2\pi r_k)}{\sin(\tp_k)}\cos(\tp_k) - \cos(2 \pi r_k) \right| \leq 1+  \left| \frac{\sin(2\pi r_k)}{\sin(\tp_k)}\right| 
\end{align*}
From Eq.\,(\ref{eqn.thetarkdef}), $r_k = (2 \pi)^{-1} \varphi(\tp_k)$. This gives the estimate,
\begin{align}
\left| \frac{\sin(2\pi r_k)}{\sin(\tp_k)}\right| \leq \sup_{x \in (0, \pi)} \left| \frac{\sin(\varphi(x))}{\sin(x)}\right| \label{eqn.sinbd}
\end{align}
In the proof of Claim \ref{claim.lpc1} we established that $\varphi(0) = 0$, $\varphi(\pi) =2 \pi$, and $1 \leq \varphi'(x) \leq 4$ for all $x \in [0,\pi]$. Hence L'Hopital tells us that $\sin(\varphi(x))/\sin(x)$  remains bounded as $x$ approaches $0$ or $\pi$, implying that the right hand side of Eq. (\ref{eqn.sinbd}) will be bounded independent of $k$ and $n$.
This shows that $(\betap_k)^2 = (2/n)(1- O(1/n))^{-1}$ and finishes the proof of Eq. (\ref{eqn.bk}).
\end{proof}

\begin{claim}
\label{claim.lpc4}
There exists an invertible $n \times n$ upper triangular matrix $A_n$ such that $L_n = A_n A_n^{tr}$. Moreover, if $i \leq j \leq n$, $\big| [A_n^{-1}]_{i,j} \big|^2 \leq 2^{-(j-i+1)}$.
\end{claim}
\begin{proof}
Define $d_n := \det(L_n)$ where for $n = 1$, we define $L_1$ to be the $1 \times 1$ matrix with entry 2. Using the cofactor expansion to calculate the determinant, we find the recursive formula 
\begin{align}
d_{n+2} = 4 d_{n+1} - d_{n}. \label{eqn.dnrecurse}
\end{align} 
By defining $d_0 := 1$, this equation holds true for every $n \geq 0$. With initial values $d_0 = 1$ and $d_1 = 2$, we can use elementary linear algebra to find the closed formula
\begin{align}
d_n = \frac{1}{2}\left\{ (2 + \sqrt{3})^n + (2-\sqrt{3})^n \right\}, \label{eqn.dnform}
\end{align}
valid for every $n \geq 0$. Define $A_n$ by,
\renewcommand\arraystretch{1.5}
\begin{align*}
[A_n]_{i,j} = 
\begin{cases}[1.5]
\sqrt{\frac{d_{n-i+1}}{d_{n-i}}} & i = j \\
\sqrt{\frac{d_{n-i-1}}{d_{n-i}}} & i+1 = j \leq n \\
0 & \text{ otherwise}
\end{cases}
\end{align*}

%\begin{align}
%A_n = 
%\left(
%\begin{array}{ccccc}
%\sqrt{\frac{d_n}{d_{n-1}}} & \sqrt{\frac{d_{n-2}}{d_{n-1}}} & 0 & 0 & 0 \\
%0 &  \sqrt{\frac{d_{n-1}}{d_{n-2}}} &  \sqrt{\frac{d_{n-3}}{d_{n-2}}} & 0 & 0 \\
%0 & 0 & \ddots  &  \ddots & 0 \\
%0 & 0 & 0 & \sqrt{\frac{d_2}{d_1}} & \sqrt{\frac{d_0}{d_1}} \\
%0 & 0 & 0 & 0 & \sqrt{\frac{d_1}{d_0}}
%\end{array}
%\right).
%\end{align}
In particular, using Eq. (\ref{eqn.dnrecurse}), 
\[ [ A_n A_n^{tr} ]_{k,k} = \frac{d_{n-k+1}}{d_{n-k}} + \frac{d_{n-k-1}}{d_{n-k}} = \frac{4 d_{n-k}}{d_{n- k}} = 4\]
for every $k \in \{1, 2, ..., n-1\}$, and also $[A_n A_n^{tr}]_{n,n} = d_1/d_0 = 2$. Moreover, the off-diagonal terms of $A_n A_n^{tr}$ are 1 when there is a product of the form $\sqrt{\frac{d_{k}}{d_{k-1}}}\sqrt{\frac{d_{k-1}}{d_k}}$, which occur exactly on the super- and sub-diagonal entries. All other elements are 0.  Therefore, $A_n A_n^{tr} = L_n$.

A Gaussian elimination allows us deduce the inverse,
\[
[A_n^{-1}]_{i,j} = \begin{cases}[1.5]
\sqrt{\frac{d_{n-i}}{d_{n-i+1}}} & i = j \\
(-1)^{j-i} \frac{d_{n-j}}{\sqrt{d_{n-i} d_{n-i+1}}} & i < j \\
0 & \text{ otherwise}.
\end{cases}
\]
To confirm,
\begin{align*}
[A_n A_n^{-1}]_{i,j} =  \sum_{k=1}^n [A_n]_{i,k} [A_n^{-1}]_{k,j} = [A_n]_{i,i} [A_n^{-1}]_{i,j} + 1_{\{i < n \}} [A_n]_{i,i+1} [A_n^{-1}]_{i+1,j}\\ 
= \begin{cases}[2]
\sqrt{\frac{d_{n-i+1}}{d_{n-i}}} \sqrt{\frac{d_{n-i}}{d_{n-i+1}}} = 1  &  i = j \\
\sqrt{\frac{d_{n-i+1}}{d_{n-i}}} \frac{d_{n-i-1}}{\sqrt{d_{n-i} d_{n-i+1}}} - \sqrt{\frac{d_{n-i-1}}{d_{n-i}}} \sqrt{\frac{d_{n-i-1}}{ d_{n-i}}} = 0 & i+1=j \\
\sqrt{\frac{d_{n-i+1}}{d_{n-i}}} \frac{d_{n-j}}{\sqrt{d_{n-i} d_{n-i+1}}} - \sqrt{\frac{d_{n-i-1}}{d_{n-i}}} \frac{d_{n-j}}{\sqrt{d_{n-i-1} d_{n-i}}} = 0 & i+1<j \\
0 & \text{ otherwise}
\end{cases}
\end{align*}
 Hence $A_n A_n^{-1} = I$. 
 
 To finish the proof of the claim, we need to get the appropriate size estimates for the entries in $A_n^{-1}$, which can be done by using Eq. (\ref{eqn.dnform}). Notice that $(2 + \sqrt{3})^{-1} = 2 - \sqrt{3}$. So, for any $k, l \geq 0$, 
\begin{align*}
&\frac{(2 + \sqrt{3})^k + (2-\sqrt{3})^k}{(2 + \sqrt{3})^l + (2-\sqrt{3})^l} = (2+\sqrt{3})^{k-l} \left( \frac{ 1 + (2-\sqrt{3})^{2k}}{1 + (2-\sqrt{3})^{2l}}\right) \\
&\qquad \leq 
\begin{cases}
(2+\sqrt{3})^{k-l} \left(1 + (2 - \sqrt{3})^2\right) = (2+\sqrt{3})^{k-l}(8-4\sqrt{3}) & \text{ for } k \geq 1 \\
2 (2+\sqrt{3})^{k-l} & \text{ for } k \geq 0
\end{cases}
\end{align*}
This implies
\begin{align*}
\left| [A_n^{-1}]_{i,i} \right|^2 = \frac{d_{n-i}}{d_{n-i+1}} \leq \frac{8-4\sqrt{3}}{2+\sqrt{3}} < \frac{1}{2}
\end{align*}
for $i < n$, and
\begin{align*}
\left| [A_n^{-1}]_{i,j} \right|^2 = \frac{d_{n-j}^2}{d_{n-i}d_{n-i+1}} = \frac{d_{n-j}}{d_{n-i}}\cdot\frac{d_{n-j}}{d_{n-i+1}}  \leq \frac{2}{2 + \sqrt{3}} \cdot \frac{2}{(2+\sqrt{3})^{2(j-i)}} < \frac{1}{2} \left( \frac{1}{2}\right)^{j-i}
\end{align*}
whenever $i < j \leq n$. With $|[A_n^{-1}]_{n,n}|^2 = \frac{1}{2}$, the claim is proved.
\end{proof}\\

\textbf{Conclusion of the Proof of Theorem \ref{thm:Lpthm}}. Recall that $\mLp = \frac{\Delta^3}{6} L_n \otimes I_d$. Claim \ref{claim.lpc3} along with standard results concerning the Kronecker product tell us that $\{ u_{k,a}^{\p} = f_{k} \otimes e_a : 1 \leq k \leq n, 1 \leq a \leq d\}$ forms a collection of orthonormal eigenvectors of $\mLp$ with respective eigenvalues $\{ \lambda_{k}^{\p} = \lambda_{k,a}^{\p} : 1 \leq k \leq n, 1 \leq a \leq d\}$. Further, from Claim \ref{claim.lpc4}, $\mc{A}_{\p} := \sqrt{(\Delta^3/6)}A_n \otimes I_d$ is an upper triangular matrix such that $\mLp = \mc{A}_{\p} \mc{A}_{\p}^{tr}$ with $| [\mc{A}_{\p}^{-1}]_{i,j}|^2 = (6/\Delta^3) |[A_n^{-1}]_{i,j}|^2$.  Thus, armed with Claims \ref{claim.lpc1} through \ref{claim.lpc4}, we have shown the validity of the assertions of Theorem \ref{thm:Lpthm}.
\end{proof}

\begin{cor}
\label{cor:BetaInt}
Suppose that $0 \leq r < s \leq1$ and  $f:[0,\pi]\to \re$ is Lipshitz. Then,
\begin{align}
\lim_{n\to \infty} \sum_{\{k : r < k/n < s \}} (\betap_k)^2 f(\tp_k) = 2 \int_r^s f(\pi t)dt.
\end{align}
\end{cor}
\begin{proof}
 Notice that $\sum \frac{2}{n} f\left( \frac{\pi k}{n} \right)$ is the Riemann sum approximation to $2\int f(\pi t) dt$. Hence it suffices to show that $| \sum \left\{ (\betap_k)^2 f(\tp_k) - \frac{2}{n} f\left( \frac{\pi k}{n} \right) \right\}| \to 0$, which will be done by showing that the summand $ (\betap_k)^2 f(\tp_k) - \frac{2}{n} f\left( \frac{\pi k}{n} \right)$ is $O(1/n^2)$. By assumption, $|f(x) - f(y)| = O(|x-y|)$ and $\sup_{x \in [0,\pi]} |f(x)| < \infty$.  Also by assumption, $k \neq n$, thus Eqs.\,(\ref{eqn.tk}) and (\ref{eqn.bk}) imply  $|f(\tp_k) - f(\pi k / n) | = O(1/n)$, $|(\betap_k)^2 - 2/n| = O(1/n^2)$, and $(\betap_k)^2 = O(1/n)$. Therefore,
\begin{align*}
&\left| (\betap_k)^2 f(\tp_k) - \frac{2}{n} f\left( \frac{\pi k}{n} \right) \right| \\
&\quad \leq \underbrace{| (\betap_k)^2 |}_{O(1/n)} \cdot \underbrace{| f(\tp_k) - f( \frac{\pi k}{n}) |}_{O(1/n)} + \underbrace{| f( \frac{\pi k}{n})|}_{O(1)} \cdot \underbrace{| (\betap_k)^2 - \frac{2}{n} |}_{O(1/n^2)} = O(\frac{1}{n^2}).
\end{align*}
from which the result follows.
\end{proof}

\begin{cor}
\label{cor:ksumsize}
Take $0 <  \delta < 1$ and define $\bdry{\delta} = \bdry{\delta}(n) := \{ k \in \nats : \frac{\pi k}{n+1} \leq \delta \text{ or }  \frac{\pi k}{n+1} \geq \pi - \delta\}$ and $\bdryc{\delta} = \bdryc{\delta}(n) := \{ k \in \nats : \delta < \frac{\pi k}{n+1} < \pi - \delta\}$. Let $f:[0,\pi]\to \re$ be Lipshitz. Set $\Lambda<\infty$ with  $\sup_x \left| f(x) \right| \leq \Lambda$ and $|f(x) - f(y) | \leq \Lambda |x-y|$ for each $x,y \in [0,\pi]$. Then there exists a constant $C = C(\Lambda) < \infty$ such that if $j \in \bdryc{\delta}$,
\begin{align}
\left| \sum_{k=1}^{n-1}  (\betap_{k})^2 f(\tp_k) e^{i 2j\tp_k} \right| \leq \frac{C}{n \sin(\delta)}.
\end{align}
And for any $j = 1, 2, ..., n$, 
\begin{align}
\left| \sum_{k=1}^{n-1} (\betap_{k})^2 f(\tp_k) e^{i 2j \tp_k} \right| \leq C
\end{align}
where $C$ is independent of $n$.
\end{cor}
\begin{proof}
From Eq. (\ref{eqn.tk}), for $1 \leq k \leq n-1$, $\zeta_k := \tp_k - \pi k / (n+1) = \pi r_k / (n+1)$. In particular, if $k \leq n-2$, then 
\begin{align*}
\zeta_{k+1} - \zeta_k =  \frac{\pi(r_{k+1} - r_k)}{n+1} = \frac{\varphi(\tp_{k+1}) - \varphi(\tp_k)}{2(n+1)} \leq \frac{4 (\tp_{k+1} - \tp_k)}{2(n+1)} = O(\frac{1}{n^2}).
\end{align*} 
This further implies that for any $1\leq j\leq n$, $|e^{i 2 j \zeta_{k+1}} - e^{i 2 j \zeta_{k}}| = O(1/n)$. Combining this with Eqs. (\ref{eqn.tk}) and (\ref{eqn.bk}),  
\begin{align*}
&\left| (\betap_{k+1})^2 f(\tp_{k+1}) e^{i 2j \zeta_{k+1}} - \right.  \left. (\betap_k)^2 f(\tp_k) e^{i 2j \zeta_k} \right| \\
&\quad \leq \overbrace{|f(\tp_{k+1})e^{i 2j \zeta_{k+1}}|}^{O(1)}\cdot \overbrace{|(\betap_{k+1})^2 - (\betap_k)^2 |}^{O(1/n^2)} + \overbrace{|(\betap_k)^2 e^{i2j\zeta_k}|}^{O(1/n)}\cdot \overbrace{|f(\tp_{k+1}) - f(\tp_k)|}^{O(1/n)} \\
&\qquad \qquad + \underbrace{|(\betap_k)^2 f(\tp_k)|}_{O(1/n)} \cdot \underbrace{|e^{i2j\zeta_{k+1}} - e^{i2j\zeta_{k}}|}_{O(1/n)} \\
&\quad = O(\frac{1}{n^2})
\end{align*}
where the implicit constant in the last equality is independent of $k$, but certainly depends on $\Lambda$.

Now, define the partial sum $S_m := \sum_{k=1}^m e^{i2j \frac{\pi k}{n+1}}$ and $S_0 := 0$. Using the fact that $S_m$ is a geometric series,
\begin{align*}
S_m &= \frac{ 1 - e^{i 2j \frac{\pi (m+1)}{n+1}} }{1- e^{i 2j \frac{\pi}{n+1}}} - 1 = \frac{1}{2i} \frac{e^{i 2j \frac{\pi (m+1/2)}{n+1}} - e^{i j \frac{\pi}{n+1}}}{\sin\left(  \frac{\pi j}{n+1} \right)}-1
\end{align*}
so that $\left| S_m \right| \leq (\sin(\delta))^{-1} + 1.$ 
Applying summation by parts,
\begin{align*}
&\sum_{k=1}^{n-1} (\betap_k)^2 f(\tp_k) e^{i 2j \tp_k} = (\betap_{n-1})^2 f(\tp_{n-1})e^{i2j\tp_{n-1}} + \sum_{k=1}^{n-2} (\betap_k)^2 f(\tp_k) e^{i 2j \zeta_k} e^{i 2j \frac{\pi k}{n+1}}\\
&\quad = (\betap_{n-1})^2 f(\tp_{n-1})e^{i2j\tp_{n-1}}  + \sum_{k=1}^{n-2} (\betap_k)^2 f(\tp_k) e^{i 2j \zeta_k} (S_k - S_{k-1})\\
&\quad = (\betap_{n-1})^2 f(\tp_{n-1}) (e^{i2j\tp_{n-1}} + e^{i 2 j \zeta_{n-1}}S_{n-2}) \\
&\qquad \qquad - \sum_{k=1}^{n-2} \left( (\betap_{k+1})^2 f(\tp_{k+1}) e^{i 2j \zeta_{k+1}} - (\betap_k)^2 f(\tp_k)e^{i2j\zeta_k} \right) S_{k}.
\end{align*}
Using the above estimates with Eq. (\ref{eqn.bk}),
\begin{align*}
\left| \sum_{k=1}^{n-1} (\betap_k)^2 f(\tk) e^{i 2j \tk} \right| &\leq  \frac{1}{\sin(\delta)}  \left( O( \frac{1}{n}) + \sum_{k=1}^{n-2} O(\frac{1}{n^2}) \right) \leq \frac{C}{n \sin(\delta)}.
\end{align*}
The second claim is nearly immediate since $(\betap_k)^2 \leq O(1/n)$ for any $1 \leq k \leq n-1$. Hence, 
\begin{align*}
\left| \sum_{k=1}^{n-1} (\betap_k)^2 f(\tp_k) e^{i 2j \tk} \right| &\leq  \sum_{k=1}^{n-1} O( \frac{1}{n} ) \leq C.
\end{align*}
\end{proof}

\begin{cor}
\label{cor:trLUL}
Let $U$ be an $n \times n$ symmetric tri-diagonal block matrix with $d \times d$ blocks. Then, 
\begin{align}
\tr(\mLp^{-1/2} U \mLp^{-1/2}) = \sum_{m,k=1}^n \frac{(\betap_k)^2}{\lambda^{\p}_k}\left[ (\alpha^m_k)^2 \tr([U]_{m,m}) + 2 \alpha_k^{m} \alpha_k^{m+1} \tr([U]_{m,m+1}) \right], \label{eqn:trLUL}
\end{align}
where $\betap_k, \lambda^{\p}_k,$ and $\alpha_k^m$ are as in Theorem \ref{thm:Lpthm}.
Moreover, this implies that there exists a constant $C = C(d) < \infty$ such that
\begin{align}
|\tr(\mLp^{-1/2} U \mLp^{-1/2})| &\leq \frac{C}{\Delta^3} \sum_{m=1}^n \left( |\tr([U]_{m,m})| + |\tr([U]_{m,m+1})| \right)\\
&\leq \frac{C}{\Delta^3} \sum_{m=1}^n \left( \|[U]_{m,m}\| + \|[U]_{m,m+1}\| \right), \label{eqn:trLULsize}
\end{align}
where we define $[U]_{n,n+1} := 0$.
\end{cor}
\begin{proof}
The orthonormal basis of eigenvectors $ u_{k,a}^{\p} $ of $\mLp$ from Theorem \ref{thm:Lpthm} are also eigenvectors for $\mLp^{-1}$ with respective eigenvalues $1/\lambda^{\p}_{k}$. Let for $i = 1, ..., n$ let $u_{k,a}^i := \betap_k \alpha^i_k e_a$ so that $u_{k,a}^{\p} = (u_{k,a}^1, u_{k,a}^2, ..., u_{k,a}^n)^{tr}$. To keep our manipulations succinct, we also define $\alpha^{0}_k = \alpha^{n+1}_k = 0$ for each $k = 1, ..., n$.  Since $\tr(AB) = \tr(BA)$, 
\begin{align*}
&\tr \left( \mLp^{-1/2} U \mLp^{-1/2} \right) \\
&\qquad = \tr \left( U \mLp^{-1} \right) = \sum_{k=1}^n \sum_{a=1}^d U \mLp^{-1} u_{k,a}^{\p} \cdot u_{k,a}^{\p} = \sum_{k=1}^n  \sum_{a=1}^d \frac{1}{\lambda^{\p}_k} U u_{k,a}^{\p} \cdot u_{k,a}^{\p} \\
&\qquad = \sum_{k=1}^n \frac{1}{\lambda_k^{\p}} \sum_{a=1}^d \sum_{i,j=1}^n [U]_{i,j} u_{k,a}^j \cdot u_{k,a}^i \\
&\qquad =  \sum_{k=1}^n \frac{1}{\lambda_k^{\p}} \sum_{a=1}^d \sum_{i=1}^n (\betap_k)^2 ( \alpha_k^{i-1} \alpha_k^{i} [U]_{i,i-1} e_a + (\alpha_k^i)^2 [U_{i,i}]e_a + \alpha_k^i \alpha_k^{i+1} [U]_{i,i+1}) \\
&\qquad = \sum_{k=1}^n \sum_{i=1}^n \frac{(\betap_k)^2}{\lambda^{\p}_k}\left[ (\alpha^i_k)^2 \tr([U]_{i,i}) + 2 \alpha_k^{i} \alpha_k^{i+1} \tr([U]_{i,i+1}) \right],
\end{align*}
which is Eq. (\ref{eqn:trLUL}).

 For $1 \leq k \leq n-1$, the estimates for $\betap_{k}, \lambda^{\p}_k$, and $\alpha^m_k$ in Theorem \ref{thm:Lpthm} imply the existence of $C = C(d, \text{curvature}) < \infty$ such that
\begin{align*}
\sum_{m=1}^n\sum_{k=1}^{n-1} \frac{(\betap_k)^2}{\lambda^{\p}_k}\left[ (\alpha^m_k)^2 \tr([U]_{m,m}) + 2 \alpha_k^{m} \alpha_k^{m+1} \tr([U]_{m,m+1}) \right] \\
\leq \frac{C}{\Delta^3}\sum_{m=1}^n\sum_{k=1}^{n-1} \frac{1}{n} \left( |\tr([U]_{m,m})| + |\tr([U]_{m,m+1})| \right)\\
\leq  \frac{C}{\Delta^3} \sum_{m=1}^n \left( |\tr([U]_{m,m})| + |\tr([U]_{m,m+1})| \right).
\end{align*}
For $k=n$, $\alpha_n^m = \gamma_n^m - \gamma_n^{-m}$, so that both $(\alpha_n^m)^2$ and $\alpha_n^{m} \alpha_n^{m+1}$ are $O(\gamma_n^{2m})$. According to Theorem \ref{thm:Lpthm}, $(\betap_n)^2 = O(\gamma_n^{-2n})$. Therefore 
\begin{align*}
\frac{(\betap_n)^2}{\lambda^{\p}_n}\left[ (\alpha^m_n)^2 \tr([U]_{m,m}) + 2 \alpha_n^{m} \alpha_n^{m+1} \tr([U]_{m,m+1}) \right]\\
\leq \frac{C}{\Delta^3}  \left( |\tr([U]_{m,m})| + |\tr([U]_{m,m+1})| \right)
\end{align*}

Combining these cases for $k \leq n-1$ and $k = n$ along with the fact that $| \tr([U]_{i,j}) |\leq d \|[U]_{i,j}\|$ gives the necessary size estimates.
\end{proof}

\section{Size Estimates and Uniform Integrability}
The space on which $\rho_{\p}$ lives, $\Hp(M)$, depends on the partition $\p$ and is changing as $|\p| \to 0$. If we briefly ignore this fact, it appears as though we are attempting an $L^1$-limit in Theorem \ref{thm.mainthm}. We are able to follow this instinct and use an $L^1$-limit type argument by using $\rhot_{\p}$, which for any $\p$ is a map on $W(\re^d)$. Very concisely, the argument for the proof of the main theorem goes as translating from $\rho_{\p}$ to $\rhot_{\p}$, find the $L^1$-limit of $\rhot_{\p}$, and finally translate this argument back to a limit for $\rho_{\p}$. 

With this in mind, the original strategy to achieve the desired $L^1$-limit for $\rhot_{\p}$ was to show that the collection $\{ \rhot_{\p} : \p = \{0, 1/n, 2/n ..., 1\}, n \text{ sufficiently large} \}$ is uniformly integrable and that there is some map $\rho$ such that $\rhot_{\p} \to \rho$ in measure as $|\p| \to 0$. This idea later evolved into showing the $L^1$-limit more directly, yet the artifacts of the previous approach permeate the remaining sections. The main result from this section is Theorem \ref{thm:UniformIntegrability}, which will be used in proving the $L^1$-limit, but is also a result of the uniform integrability of $\rhot_{\p}$.

Define the remainder of $\GFp$ by,
\begin{align}
\Rp := \GFp - \mLp. \label{eqn.Rp}
\end{align}
Using the decomposition of $\mLp = \mAp \mAp^{tr}$ in Theorem \ref{thm:Lpthm} and Eq. (\ref{eqn.rhotp2}), 
\begin{align}
\rhot_{\p} &= \sqrt{ \det\left(I + \mAp^{-1} \Rp (\mAp^{-1})^{tr} \right)}.
\end{align}

\noindent To ease notation for the remainder of this chapter, we also introduce
\begin{align}
\Kp{i} := \sup_{0 \leq s \leq \Delta}  \| \Ap{i}(s) \| \label{eqn:K}
\end{align}
where $\Ap{i}$ is defined in Eq. (\ref{eqn:Ai}) and $\Delta$ is notation defined in Assumption \ref{assumption.Pequal}. Using Eq. (\ref{eqn:Aibound}),
\begin{align}
\Kp{i} \leq \kappa \frac{\| \Delta_i b \|^2}{\Delta^2} \label{eqn:Kibound}.
\end{align}

\subsect{Estimates on the remainder $\Rp$}
Here we give estimates on the non-zero $d \times d$ blocks of the remainder $\Rp$, which will in turn be used to estimate the size of $\rhot_{\p}$ sufficient for proving Theorem \ref{thm:UniformIntegrability}. From Eq. (\ref{eqn.Rp}) along with the definition of $\mLp$ in Eq. (\ref{eqn.Lp2}) and $\GFp$ in Eq. (\ref{eqn:GFpDiag}), the non-zero $d \times d$ blocks of $\Rp$ are of the form
\begin{align}
& [\Rp]_{i,i} =\begin{cases}
\int_0^{\Delta} \left( \Vp{i+1}(s)^{\tr} \Vp{i+1}(s) + \Sp{i}(s)^{\tr} \Sp{i}(s) \right)ds - \frac{2\Delta^3}{3} I  & i < n\\
\int_0^{\Delta}\left( \Sp{i}(s)^{\tr} \Sp{i}(s) \right)ds - \frac{\Delta^3}{3} I  & i = n
\end{cases} \notag \\
&=\begin{cases}
\int_0^{\Delta} \left\{ \left( \Vp{i+1}(s)^{\tr} \Vp{i+1}(s) - (\Delta-s)^2I \right) + \left( \Sp{i}(s)^{\tr} \Sp{i}(s) -s^2 I \right) \right \}ds & i < n \\
\int_0^{\Delta}\left( \Sp{n}(s)^{\tr} \Sp{n}(s) - s^2 I \right)ds  & i = n
\end{cases}. \label{eqn:RpDiag}
\end{align}
and for $1 \leq i < n$,
\begin{align}
[\Rp]_{i,i+1} = [\Rp]_{i+1,i}^{tr} &= \int_0^{\Delta} \Vp{i+1}(s)^{tr} \Sp{i+1}(s)ds - \frac{\Delta^3}{6}I  \notag \\
&= \int_0^{\Delta} ( \Vp{i+1}(s)^{tr} \Sp{i+1}(s) - (\Delta-s)sI)ds \label{eqn:RpOffDiag}
\end{align}
Written suggestively in Eqs. (\ref{eqn:RpDiag}) and (\ref{eqn:RpOffDiag}), we set out to estimate $\|  \Vp{i+1}(s)^{\tr} \Vp{i+1}(s) - (\Delta-s)^2I  \|$, $\| \Sp{i}(s)^{\tr} \Sp{i}(s) -s^2 I  \|$, and $\| \Vp{i+1}(s)^{tr} \Sp{i+1}(s) - (\Delta-s)sI \|$ . 

\begin{lem} 
\label{lem:SVlinear}
For $s \in [0, \Delta]$,
\begin{align}
&\left\| \Sp{i} (s) - sI \right\| \leq s \left( \cosh(\sqrt{\Kp{i}} \Delta) - 1 \right) \quad \text{ and} \label{eqn:Slinear} \\
&\left\| \Vp{i+1}(s) - (\Delta - s) I \right\| \notag \\ 
&\qquad \leq \frac{s}{\Delta} \left( \Delta-s \right) \left(\cosh(\sqrt{\Kp{i}} \Delta) \cosh(4 \sqrt{\Kp{i+1}} \Delta) - 1 \right). \label{eqn:Vlinear2}
\end{align}
\end{lem}
\begin{proof}
Eq. (\ref{eqn:Slinear}) is a direct consequence of Proposition \ref{prop:ODE} along with the fact that $\cosh$ is strictly increasing on $[0,\infty)$. Next, 
\begin{align*}
 \| \Fp{i} \| =  \| \Sp{i+1}(\Delta)^{-1} \Cp{i+1}(\Delta) \Sp{i}(\Delta) \| \leq \left\| \Cp{i+1}(\Delta) \frac{\Sp{i}(\Delta)}{\Delta} \right\| \\
  \leq \cosh(\sqrt{\Kp{i+1}} \Delta )\cosh(\sqrt{\Kp{i}} \Delta), 
\end{align*}
where the first inequality follows from Proposition \ref{prop:PosSemiDefBound} and the second from another application of Proposition \ref{prop:ODE}.
 Hence, by Proposition \ref{prop:GREENest},
\begin{align*}
&\left\| \Vp{i+1}(s) - \frac{\Sp{i} (\Delta)}{\Delta}(\Delta- s) \right\| \\
& \leq s\left( 1 - \frac{s}{\Delta} \right) \left[ \| \Sp{i} (\Delta) \| \Kp{i+1} \Delta \cosh(\sqrt{\Kp{i+1}} \Delta) + \| \Fp{i} \| \left( \cosh(\sqrt{\Kp{i+1}} \Delta) - 1 \right) \right] \\
& \leq s \left( 1 - \frac{s}{\Delta} \right) \cosh(\sqrt{\Kp{i}}\Delta)\cosh(\sqrt{\Kp{i+1}} \Delta) \left( \Kp{i+1} \Delta^2 +  \cosh(\sqrt{\Kp{i+1}} \Delta) - 1 \right) \\
& \leq s \left( 1 - \frac{s}{\Delta} \right) \cosh(\sqrt{\Kp{i}}\Delta)\cosh(\sqrt{\Kp{i+1}} \Delta) \left( \left( \Kp{i+1} \Delta^2 +  1 \right)\cosh(\sqrt{\Kp{i+1}} \Delta) - 1 \right) ~~\\
& \leq s \left( 1 - \frac{s}{\Delta} \right) \cosh(\sqrt{\Kp{i}}\Delta)\cosh(\sqrt{\Kp{i+1}} \Delta) \left( \cosh^2(\sqrt{2\Kp{i+1}} \Delta) - 1 \right)\\
& \leq s \left( 1 - \frac{s}{\Delta} \right) \cosh(\sqrt{\Kp{i}}\Delta) \left( \cosh(4\sqrt{\Kp{i+1}} \Delta) - 1 \right),
\end{align*}
with the final inequality resulting from Eqs. (\ref{eqn.coshacoshb}) and (\ref{eqn.coshabminus1}).
Therefore, for Eq. (\ref{eqn:Vlinear2}),  
\begin{align*}
& \left\| \Vp{i+1}(s) - (\Delta-s)I \right\| \\
& \leq \left\| \Vp{i+1}(s) - \frac{\Sp{i}(\Delta)}{\Delta}(\Delta-s) \right\| + (1 - \frac{s}{\Delta} )\left\| \Sp{i}(\Delta) - \Delta I \right\| \\
& \leq \frac{s}{\Delta} \left( \Delta-s \right) \left[ \cosh(\sqrt{\Kp{i}} \Delta)\left( \cosh(4 \sqrt{\Kp{i+1}} \Delta) - 1 \right) + \left( \cosh(\sqrt{\Kp{i}} \Delta) - 1 \right)  \right]\\
&= \frac{s}{\Delta} \left( \Delta-s \right) \left( \cosh(\sqrt{\Kp{i}}\Delta) \cosh(4\sqrt{\Kp{i+1}}\Delta) - 1 \right)
\end{align*}
\end{proof}

\begin{prop} \label{prop:SVdiag}
 For $s \in [0,\Delta]$, 
\begin{align}
&\left\| \Vp{i+1}(s)^{\tr} \Vp{i+1}(s) - (\Delta-s)^2I \right\| \notag \\
&\qquad \leq 3(\Delta-s)^2 \left( \cosh(2\sqrt{\Kp{i}}\Delta)\cosh(8\sqrt{\Kp{i+1}}\Delta)-1 \right), \\
&\left\| \Sp{i}(s)^{\tr} \Sp{i}(s) - s^2I \right\| \leq 3s^2 \left( \cosh(2 \sqrt{\Kp{i}} \Delta) - 1 \right) 
\end{align}
and
\begin{align}
&\| \Vp{i+1}(s)^{tr} \Sp{i+1}(s) - (\Delta-s)sI \| \notag \\
&\qquad \leq 3s(\Delta-s) \left( \cosh(\sqrt{\Kp{i}} \Delta) \cosh(5 \sqrt{\Kp{i+1}} \Delta) - 1 \right).
\end{align}
\end{prop}
\begin{proof}
We apply Lemma \ref{lem:SVlinear} above by noticing the following. For operators $A$ and $B$ and real numbers $a$ and $b$,
\begin{align*}
A^{tr}B - ab I = (A^{tr}-aI)(B- bI) + a(B-bI) + b(A^{tr} -aI).
\end{align*} 
The asserted inequalities now follow with judicious choices for $A$ and $B$ as well as Eqs. (\ref{eqn.coshacoshb}) and (\ref{eqn.coshabminus1}) along with the fact that $s/\Delta \leq 1$.
\end{proof}

Applying Proposition \ref{prop:SVdiag} to Eqs (\ref{eqn:RpDiag}) and (\ref{eqn:RpOffDiag}) gives the estimates we need on $\Rp$ to continue forward.

\begin{prop} \label{prop:RpDiagSize}
For $1\leq i \leq n$, 
\begin{align}
\left\| \left[ \Rp \right]_{i,i} \right\| &\leq 2 \Delta^3 \left( \cosh( 2 \sqrt{\Kp{i}} \Delta) \cosh(8 \sqrt{\Kp{i+1}} \Delta) - 1\right) \label{eqn:RpDiagSize}
\end{align}
and for $1 \leq i < n$,
\begin{align}
\left\| \left[ \Rp \right]_{i,i+1} \right\| = \left\| \left[ \Rp \right]_{i+1,i} \right\| \leq \frac{\Delta^3}{2} \left( \cosh(\sqrt{\Kp{i}}\Delta) \cosh(5\sqrt{\Kp{i+1}} \Delta) - 1 \right). \label{eqn:RpOffDiagSize}
\end{align}
Moreover, this implies that,
\begin{align}
\left\| \left[ \mAp^{-1} \Rp (\mAp^{-1})^{tr} \right]_{i,i} \right\| \leq \sum_{j=i}^n \lambda_{i,j}\left( \cosh(30\sqrt{\Kp{j}} \Delta) \cosh(120 \sqrt{\Kp{j+1} } \Delta) - 1 \right) \label{eqn:ARAdiagSize}
\end{align}
where $\lambda_{i,j} := \left(\frac{1}{2} \right)^{j-i} \left( \sum_{j=i}^n \left(\frac{1}{2}\right)^{j-i} \right)^{-1} =  \left(\frac{1}{2} \right)^{j-i}\left(2 - \left(\frac{1}{2}\right)^{n-i} \right)^{-1}$ are chosen so that $\sum_{j=i}^n \lambda_{i,j} = 1$.
\end{prop}
\begin{proof}
Eqs. (\ref{eqn:RpDiagSize}) and (\ref{eqn:RpOffDiagSize}) come from integrating the bounds given  in Proposition \ref{prop:SVdiag} with respect to $s$ for $s \in [0, \Delta]$. 

For Eq. (\ref{eqn:ARAdiagSize}), recall that $\mAp$ is upper diagonal, and hence so is $\mAp^{-1}$, and that $\Rp$ is tri-diagonal, yielding
\begin{align*}
&[\mAp^{-1} \Rp (\mAp^{-1})^{tr}]_{i,i} \\
&\qquad = \sum_{j,k = 1}^n [\mAp^{-1}]_{i,j} [\Rp]_{j,k} [(\mAp^{-1})^{tr}]_{k,i} \\
&\qquad =  \sum_{j=i}^n \left\{  [\mAp^{-1}]_{i,j} [\Rp]_{j,j} [(\mAp^{-1})^{tr}]_{j,i} + [\mAp^{-1}]_{i,j+1} [\Rp]_{j+1,j} [(\mAp^{-1})^{tr}]_{j,i}  \right. \\
&\qquad \qquad  \left. + [\mAp^{-1}]_{i,j} [\Rp]_{j,j+1} [(\mAp^{-1})^{tr}]_{j+1,i} \right\}
\end{align*}
where we keep the convention that for $j=n$, $[~\cdot~]_{n,n+1} = [~\cdot~]_{n+1,n} = 0$.
Therefore, from Eq. (\ref{eqn:LpLUsize}),
\begin{align*}
&\|[\mAp^{-1} \Rp (\mAp^{-1})^{tr}]_{i,i} \| \\
& \qquad \leq \sum_{j=i}^n \left\{ \| [ \mAp^{-1} ]_{i,j} \|^2 \| [ \Rp ]_{j,j} \| + 2 \| [\mAp^{-1}]_{i,j+1} \| \| [\mAp^{-1} ]_{i,j}\| \| [\Rp]_{j+1,j} \| \right\} \\
& \qquad \leq \sum_{j=i}^n \left\{ 3 \left(\frac{1}{2}\right)^{j-i} \left[ 2 \left( \cosh( 2 \sqrt{\Kp{j}} \Delta) \cosh(8 \sqrt{\Kp{j+1}} \Delta) - 1\right)  \right. \right. \\
& \qquad \qquad \qquad + \left. \left. \frac{1}{2} \left( \cosh(\sqrt{\Kp{j}}\Delta) \cosh(5\sqrt{\Kp{j+1}} \Delta) - 1 \right) \right] \right\} \\
& \qquad \leq \sum_{j=i}^n \frac{15}{2} \left( \frac{1}{2} \right)^{j-i} \left( \cosh( 2 \sqrt{\Kp{j}} \Delta) \cosh(8 \sqrt{\Kp{j+1}} \Delta) - 1 \right) \\
& \qquad \leq \sum_{j=i}^n 15 \lambda_{i,j}  \left( \cosh( 2 \sqrt{\Kp{j}} \Delta) \cosh(8 \sqrt{\Kp{j+1}} \Delta) - 1 \right) \\
& \qquad \leq \sum_{j=i}^n  \lambda_{i,j}  \left( \cosh( 30 \sqrt{\Kp{j}} \Delta) \cosh(120 \sqrt{\Kp{j+1}} \Delta) - 1 \right)
\end{align*}
wherein the last inequality we used Eq. (\ref{eqn.alphacoshabminus1}).
\end{proof}

\subsect{Bounds on $\rhot_{\p}$ and Uniform Integrability}
\begin{lem} 
\label{lem:reformdet}
Let $\{ \lambda_{i,j} : 1 \leq i \leq n, i \leq j \leq n \}$ be defined as in Proposition \ref{prop:RpDiagSize}. Define $\{ p_{i,j} : 1 \leq i \leq n, i \leq j \leq n\}$ by
\begin{align}
p_{i,j} :=
\begin{cases}
\frac{1}{2}(\lambda_{i,j} + \lambda_{i,j-1}) & j>i \\
\frac{1}{2}(\lambda_{i,i} + \lambda_{i,n}) & j=i
\end{cases}.
\end{align}
Then,
\begin{align}
\det\left( I + \mAp^{-1} \Rp (\mAp^{tr})^{-1}\right) \leq \prod_{i=1}^n\left( \sum_{j=i}^n p_{i,j} \cosh(240 \sqrt{\Kp{j}} \Delta) \right)^d.
\end{align}
Moreover, $\sum_{j=i}^n p_{i,j} = 1$ and $\sum_{i=1}^j p_{i,j} < 3$
\end{lem}
\begin{proof}
The matrix $I + \mAp^{-1} \Rp (\mAp^{tr})^{-1}$ is symmetric positive definite, so we can apply Fischer's inequalty (see \cite[Theorem 7.8.3]{MR832183}),
\begin{align*}
\det(I + \mAp^{-1} \Rp (\mAp^{tr})^{-1}) &\leq  \prod_{i=1}^n \det([I + \mAp^{-1} \Rp (\mAp^{tr})^{-1}]_{i,i}) \\
& \leq \prod_{i=1}^n \left( 1 + \| [ \mAp^{-1} \Rp (\mAp^{tr})^{-1}]_{i,i} \| \right)^d.
\end{align*} 
From Proposition \ref{prop:RpDiagSize},
\begin{align*}
&\leq \prod_{i=1}^n \left( 1+ \sum_{j=i}^n \lambda_{i,j}  \left( \cosh( 30 \sqrt{\Kp{j}} \Delta) \cosh(120 \sqrt{\Kp{j+1}} \Delta) - 1 \right) \right)^d \\
& = \prod_{i=1}^n \left( \sum_{j=i}^n \lambda_{i,j} \cosh( 30 \sqrt{\Kp{j}} \Delta) \cosh(120 \sqrt{\Kp{j+1}} \Delta)  \right)^d.
\end{align*}
Using that $xy \leq \frac{1}{2} (x^2 + y^2)$, 
\begin{align*}
&\leq \prod_{i=1}^n \left( \sum_{j=i}^n \frac{1}{2}(\lambda_{i,j} + \lambda_{i,j-1}) \cosh^2(120 \sqrt{\Kp{j}} \Delta) \right)^d \\
&\leq \prod_{i=1}^n \left( \sum_{j=i}^n p_{i,j} \cosh(240 \sqrt{\Kp{j}} \Delta) \right)^d 
\end{align*}
where we define $\lambda_{i,i-1} := 0$ and use Eq. (\ref{eqn.coshacoshb}). 

\noindent We now calculate, $\sum_{j=i}^n p_{i,j} = \frac{1}{2} \left( 2 \sum_{j=i}^n \lambda_{i,j} \right) = 1.$
From the definition of $\lambda_{i,j} = \left(\frac{1}{2} \right)^{j-i} \left( 2 - \left(\frac{1}{2} \right)^{n-i} \right)^{-1} \leq \left( \frac{1}{2} \right)^{j-i}$, and therefore, $\sum_{i=1}^j p_{i,j} \leq \sum_{i=1}^j \frac{3}{2} \left( \frac{1}{2} \right)^{j-i} < 3.$
\end{proof}

\begin{prop}
\label{prop:limsupPreUI}
Let $\zeta>0$, $p \in \nats$, and $\{ p_{i,j} : 1 \leq i \leq n, i \leq j \leq n\}$ be defined as in Lemma \ref{lem:reformdet}. Then,
\begin{align}
\limsup_{n \to \infty} \E\left[ \prod_{i=1}^n \left( \sum_{j=i}^n p_{i,j} \cosh(\zeta \| \Delta_j b \| \right)^p \right] < \infty.
\end{align}
\end{prop}
\begin{proof}
For convenience define $x_j := \zeta \| \Delta_j b \|$. Using the geometric-arithmetic mean inequality, 
\begin{align*}
\prod_{i=1}^n \left( \sum_{j=i}^n p_{i,j} \cosh(x_j) \right)^p &\leq \left( \sum_{i=1}^n \frac{1}{n}  \left( \sum_{j=i}^n p_{i,j} \cosh(x_j) \right) \right)^{np} \\
&= \left( \sum_{i=1}^n \sum_{j=i}^n \frac{p_{i,j}}{n} \cosh(x_j)\right)^{np} \\
&= \left(1 + \sum_{i=1}^n \sum_{j=i}^n \frac{p_{i,j}}{n} (\cosh(x_j) -1) \right)^{np} \\
&< \left( 1 + \frac{3}{n} \sum_{j=1}^n (\cosh(x_j) - 1) \right)^{np}
\end{align*}
Where we used Lemma \ref{lem:reformdet} to realize $\sum_{i=1}^n \sum_{j=i}^n \frac{p_{i,j}}{n} = \sum_{i=1}^n \frac{1}{n}= 1$ and $\sum_{i=1}^j p_{i,j} < 3$. Estimating the sum on the right hand side of the inequality,
\begin{align*}
\sum_{j=1}^n (\cosh(x_j) - 1) &= \sum_{j=1}^n \sum_{k=1}^{\infty} \frac{x_j^{2k}}{(2k)!} \leq  \sum_{k=1}^{\infty}  \frac{\left( \sum_{j=1}^n x_j^{2}\right)^k}{(2k)!} \\
& \leq  \sum_{k=1}^{\infty} \frac{\| \mathbf{x} \|^{2k}}{(2k)!} 
= \cosh(\|\mathbf{x} \|) - 1
\end{align*}
where $\mathbf{x} := (x_1, ..., x_n)$. Fix some $\alpha \in (0, \frac{1}{4 \zeta^2 p})$ and let $C_{\alpha}< \infty$ such that 
\begin{align*}
\cosh(\|\mathbf{x} \|) - 1 &\leq C_{\alpha}( e^{\alpha \| \mathbf{x} \|^2} - 1) = C_{\alpha}( e^{\tilde{\alpha} \| \mathbf{B}_n \|^2} - 1) 
\end{align*}
where $\mathbf{B}_n := ( \Delta_1 b, \cdots, \Delta_n b)$ and $\tilde{\alpha} = \zeta^2 \alpha \in (0, \frac{1}{4p})$.
Therefore, using the above inequalities and Lemma \ref{lem:limsupGeneral},
\begin{align*}
&\limsup_{n \to \infty} \E\left[ \prod_{i=1}^n \left( \sum_{j=i}^n p_{i,j} \cosh(\zeta \| \Delta_j b \| \right)^p \right] \\
&\leq \limsup_{n \to \infty} \E \left[  \left( 1 + \frac{3}{n} \sum_{j=1}^n (\cosh(\zeta \| \Delta_j b \|) - 1) \right)^{np} \right] \\
&\leq \limsup_{n \to \infty} \E \left[  \left( 1 + \frac{3C_{\alpha}}{n} (e^{\tilde{\alpha}\| \mathbf{B}_n\|^2} - 1) \right)^{np} \right] < \infty.
\end{align*}
\end{proof}

The following Theorem is the main result for this section and is, in fact, just a corollary of what we've shown thus far.

\begin{thm}
\label{thm:UniformIntegrability}
For any $p \in \nats$, 
\begin{align}
\limsup_{|\p| \to 0} \E \left[ \left(\det ( I +\mAp^{-1} \Rp (\mAp^{tr})^{-1} ) \right)^p\right]  < \infty. \label{eqn:rhotsize1}
\end{align}
In particular, given some $p \in \nats$, there exists $N \in \nats$ and $C< \infty$ such that 
\begin{align}
\sup_{n\geq N} \{ \E[(\rhot_{\p})^p] : \#(\p) = n \} <  C. \label{eqn:rhotsize2}
\end{align}
This further implies that $\{ \rhot_{\p} : \#(\p)=n, n \in \nats \}$ are uniformly integrable.
\end{thm}
\begin{proof}
From Lemma \ref{lem:reformdet} and Eq. (\ref{eqn:Kibound}), 
\begin{align*}
\left(\det ( I +\mAp^{-1} \Rp (\mAp^{tr})^{-1} ) \right)^p &\leq \prod_{i=1}^n\left( \sum_{j=i}^n p_{i,j} \cosh(240 \sqrt{\Kp{j}} \Delta) \right)^{dp} \\
&\leq \prod_{i=1}^n\left( \sum_{j=i}^n p_{i,j} \cosh(240 \sqrt{\kappa} \| \Delta_j b \|) \right)^{dp}.
\end{align*}
Applying Proposition \ref{prop:limsupPreUI},
\begin{align*}
\limsup_{|\p| \to 0} \E \left[ \prod_{i=1}^n\left( \sum_{j=i}^n p_{i,j} \cosh(240 \sqrt{\kappa} \| \Delta_j b \|) \right)^{dp} \right] < \infty,
\end{align*}
concluding the proof of Eq. (\ref{eqn:rhotsize1}). From here, Eq. (\ref{eqn:rhotsize2}) is simply a matter of combining the definitions of $\rhot_{\p}$ and $\limsup$.
\end{proof}

\section{The Space $\Hpe$}
We now begin the journey to find an expression for the limit of $\rhot_{\p}$. Lemma \ref{lem:MeasHpe} below will be used to allow us to restrict $\Hp(M)$ to the better-behaved space $\Hpe(M)$, where we control the size of $\| \Delta_i b \|$. This then gives us the ability to Taylor expand the pieces making up $\rhot_{\p}$ in terms of $\| \Delta_i b \|$ and eventually neglect higher order terms in the limit. 

\subsect{Definition of $\Hpe$ and Preliminary Estimates}
Let $\epsi > 0$. Define the subspace $\HpeR \subset \Hp(\re^d)$ by,
\begin{align}
\HpeR :=  \left\{ \vee_{i=1}^n \| \Delta_i b \| \leq \epsi \right\} \cap \Hp(\bbR^d). \label{eqn:Hpe}
\end{align}
From here we can define $\Hpe(M) \subset H_{\p}(M)$ by
\begin{align}
\label{eqn:HpeM}
\Hpe(M) = \phi(\Hpe(\re^d))
\end{align}
where $\phi$ is Cartan's Development discussed in Section \ref{section:CartanDevelopment}. It is worth noting that we might have also chosen to define $\Hpe(M)$ equivalently as,
\begin{align}
\Hpe(M) = \left\{ \sigma \in \Hp(M) : \vee_{i=1}^n \int_{s_{i-1}}^{s_i} \| \sigma'(s) \| ds \leq \epsi \right\}.
\end{align}

Lemma \ref{lem:MeasHpeSp} is proved in \cite[Proposition 5.13]{AndDrive:1999} and left unproved here. Lemma \ref{lem:MeasHpe} is a similar estimate that we need so we can focus on the limiting behavior of $\rhot_{\p}$ on $\Hpe(\bbR^d)$. Recall from Assumption \ref{assumption.Pequal} that we are using an equally spaced partition and $|\p| = \Delta$. 

\begin{lem}
\label{lem:MeasHpeSp}
For any $\epsi > 0$, there is a constant $C = C(d) < \infty$ such that 
\begin{align}
\nuS{\p} ( \Hp(M) \backslash \Hpe(M)) \leq \frac{C}{\epsi^2} \ex\left\{ - \frac{\epsi^2}{4 \Delta} \right\}.
\end{align}
\end{lem}

\begin{lem}
\label{lem:MeasHpe}
For $\epsi > 0$ and sufficiently small $\Delta$, there exists a $C = C(\ddim) < \infty$ such that 
\begin{align}
\nuG{\p}( H_{\p}(M) \backslash H_{\p}^{\epsi}(M) ) \leq  \frac{C}{ \sqrt{\Delta}  \epsi} \ex\{ - \frac{\epsi^2}{8 \Delta} \}.
\end{align}
\end{lem}
\begin{proof}
Set $\Gamma = H_{\p}(M) \backslash H_{\p}^{\epsi}(M)$. Let $B_i := \{ \| \Delta_i b \| > \epsi\} \subset W(\bbR^d)$, and $B := \cup_{i} B_i$. We have, $\phi^{-1}(\Gamma) = \Hp(\re^d) \backslash \Hpe(\re^d) = B \cap \Hp(\re^d),$ and ${(\bp)}^{-1}( B \cap \Hp(\re^d) ) = \cup_i  \{ \| \Delta_i b \| > \epsi \} = B.$ Therefore, $\phi(\bp)^{-1}(\Gamma) = B$ and 
\begin{align*}
\nuG{\p}(\Gamma) &= \int_{\Gamma} d\nuG{\p} = \int_{\Gamma} \rho_{\p} d\nuS{\p} = \int_{B} \rhot_{\p} d\mu \\
&\leq \sum_{i=1}^n \int_{B_i} \rhot_{\p} d\mu \leq \sum_{i=1}^n \left(\E[1_{B_i}] \right)^{1/2} \left( \E[ (\rhot_{\p})^2] \right)^{1/2}.
\end{align*}
Here the second equality comes from $\nuG{\p} = \rho_{\p} \nuS{\p}$, the third equality comes from Theorem \ref{thm:lawofbp}. Applying Theorem \ref{thm:UniformIntegrability} with $n$ sufficiently large (equivalently, $\Delta$ sufficiently small), we find a constant $C < \infty$ such that 
\begin{align*}
\sup\{ \left( \E[ (\rhot_{\p})^2] \right)^{1/2} : \#(\p) = n, n \text{ sufficiently large } \} \leq C.
\end{align*}
Further, from Lemma \ref{lem:GaussBound} with $k = 0$ and $a =  \epsi / \sqrt{|\p|} = \epsi / \sqrt{\Delta}$, 
\begin{align*}
\E[1_{B_i}] \leq \frac{ C  \Delta }{\epsi^2} \ex\{ - \frac{\epsi^2}{4 \Delta} \}.
\end{align*}
Therefore,
\begin{align*}
\nuG{\p}(\Gamma) &\leq \sum_{i=1}^n \left(\E[1_{B_i}] \right)^{1/2} \left( \E[ (\rhot_{\p})^2] \right)^{1/2} \leq  C  \sum_{i=1}^n \frac{\sqrt{  \Delta }}{\epsi} \ex\{ - \frac{\epsi^2}{8  \Delta} \} \\
&=  C \frac{\sqrt{  \Delta }}{ \Delta  \epsi} \ex\{ - \frac{\epsi^2}{8\Delta} \} =  \frac{C}{ \sqrt{\Delta}  \epsi} \ex\{ - \frac{\epsi^2}{8 \Delta} \},
\end{align*}
where we are certainly using the fact that $\Delta = 1/n$. 
\end{proof}

We conclude this section with a few estimates which motivate several of the bounds that we make in the sequel.
\begin{cor}
\label{cor.HpOutHpe}
Given a bounded measurable function $f$ on $\Hp(M)$ and sufficiently small $\Delta$, there exists a $C < \infty$  depending only on $d$ and the bound on $f$ such that
\begin{align*}
\left| \int_{\Hp(M) \backslash \Hpe(M)} f(\sigma) \rho_{\p}(\sigma) d\nuS{\p}(\sigma) \right| \leq \frac{C}{ \sqrt{\Delta}  \epsi} \ex\{ - \frac{\epsi^2}{8 \Delta} \}.
\end{align*}
\end{cor}
\begin{proof}
This is immediate using Lemma \ref{lem:MeasHpe} along with the fact that $d\nuG{\p} = \rho d\nuS{\p}$.
\end{proof}

\begin{cor}
\label{cor:sumbto0}
Take $a, c>0$ and $p \in \bbN$ with $p \geq 2$. Let $\Gamma \subset \{1, 2, ..., n\}$ with $\#(\Gamma) = m$. Suppose  $Y$ is a random variable on $\Hp(\re^d)$ such that $|Y| \leq c \sum_{i \in \Gamma} \left\| \Delta_{i} b \right\|^p$. For sufficiently small $\Delta$ and $\epsi$, there exists a $C = C(d, c, a, p) < \infty$ such that, 
\begin{align}
\int_{\Hpe(\re^d)} (e^Y - 1)^a d\muS{\p} \leq C m \Delta^{\frac{p}{2}}. \label{eqn.sumbto01}
\end{align}
\end{cor}
\begin{proof}
Since  $(e^Y -1)^a \leq e^{a|Y|} - 1$, we will assume that $a=1$ without losing generality. From Eq. (\ref{eqn.eaminus1}) and Theorem (\ref{thm:lawofbp}),
\begin{align*}
&\int_{\Hpe(M)} \left| e^Y - 1 \right| d\muS{\p}\\
&\qquad \leq c \sum_{i \in \Gamma} \int_{\Hpe(\re^d)} \| \Delta_{i} b \|^{p} \ex \left\{ c \epsi^{p-2} \sum_{j \in \Gamma} \| \Delta_{j} b \|^2 \right\} d\muS{\p} \\
&\qquad \leq c \sum_{i \in \Gamma} \int_{\Hp(\re^d)} \| \Delta_{i} b \|^{p} \ex \left\{ c \epsi^{p-2} \sum_{j \in \Gamma} \| \Delta_{j} b \|^2 \right\} d\muS{\p} \\
&\qquad = c \sum_{i\in \Gamma} \E \left[ \| \Delta_{i} b \|^p \ex \left\{ c \epsi^{p-2} \sum_{j \in \Gamma} \| \Delta_{j} b \|^2 \right\} \right]  \\
&\qquad \leq C m \Delta^{\frac{p}{2}}
\end{align*}
where the last inequality follows from Eq. (\ref{eqn.deltabineq2}) in Lemma \ref{lem:sumbpto0}.
%For the second claim use Holder's inequality,
%\begin{align*}
%\int_{\Hpe(\re^d)} |e^Y - e^X| d\muS{\p}  \leq \left[ \int_{\Hp(\re^d)} e^{2X} d\muS{\p} \right]^{1/2}  \left[ \int_{\Hpe(\re^d)} (e^{|Y-X|}-1)^2 d\muS{\p} \right]^{1/2}.
%\end{align*}
%Now the result follows from Eq. (\ref{eqn.sumbto01}).
\end{proof}

\subsect{Taylor Expansions in $\Hpe$} Our approach to come upon the limiting function of $\rhot_{\p}$ as $|\p| \to 0$ is to expand the entries in $\GFp$,  $\int_0^{\Delta}\{ \Sp{i}(s)^{\tr} \Sp{i}(s) + \Vp{i+1}(s)^{\tr} \Vp{i+1}(s) \}ds$ and $\int_0^{\Delta} \Vp{i+1}(s)^{\tr} \Sp{i+1}(s) ds$, in terms of $\| \Delta_i b \|$. Fortunately we will see later that terms of order $\| \Delta_i b \|^3$ and higher will vanish. It is important to remember that these expansions are taking place on $\Hpe(\re^d)$ which insures that $\| \Delta_i b \| \leq \epsi$, which is a fact that is frequently used in the below estimates.

It is convenient here to use the notation $A  = O(r)$ for an operator $A$ and $r >0$ to mean $\| A \| \leq C r$ for some bounding constant $C$ possibly depending on $\epsi$ and the curvature of $M$. Using this notation, Eqs. (\ref{eqn:Aibound}) and (\ref{eqn:Aipbound}) can be restated as $\Ap{i} \Delta^2 = O(\|\Delta_i b\|^2)$ and $(\frac{d}{ds} \Ap{i} ) \Delta^3 = O(\| \Delta_i b\|^3)$ with  $C = \kappa$.

\begin{lem}
\label{lem:SpCpVpest}
For sufficiently small $\epsi > 0$ and $s \in (0, \Delta]$, 
\begin{align}
&\Sp{i}(s) = sI + \frac{s^3}{6}\Ap{i}(0) + O(s \| \Delta_i b \|^3), \label{eqn:Sptaylor}\\
&\Cp{i}(s) = I + \frac{s^2}{2}\Ap{i}(0) + O(\| \Delta_i b \|^3), \label{eqn:Cptaylor} 
\end{align}
and
\begin{align}
\Vp{i+1}(s) = ~(\Delta-s)I +\left( \frac{\Delta^3 - s \Delta^2}{6} \right)\Ap{i}(0) + \left( \frac{3\Delta s^2 - 2\Delta^2 s - s^3}{6} \right)\Ap{i+1}(0) \label{eqn:Vptaylor} \\
 + O(\Delta \{ \| \Delta_i b \|^3 + \| \Delta_{i+1} b \|^3 \}) \notag
\end{align}
on $\HpeR$. 
Moreover, the bounding constant can be taken independent of $i$.
\end{lem}
\begin{proof}
The fact that we can choose the bounding constant independent of $i$ comes from Eqs (\ref{eqn:Aibound}) and (\ref{eqn:Aipbound}), where the bound on $\Ap{i}$ and its derivative can be chosen independent of $i$. From here Eqs. (\ref{eqn:Sptaylor}) and (\ref{eqn:Cptaylor}) are a direct consequence of \ref{prop:ODE}.
Combining Eq. (\ref{eqn:Aibound}) and Eq. (\ref{eqn:Sptaylor}) implies that $s^{-1}\Sp{i}(s) - I  = O( \|\Delta_i b \|^2)$.
Therefore with $\epsi$ sufficiently small, 
\begin{align*}
\left[ \frac{\Sp{i}(s)}{s} \right]^{-1} &= \left[ I + \frac{s^2}{6} \Ap{i}(0) + \left(\frac{\Sp{i}(s)}{s} -I - \frac{s^2}{6}\Ap{i}(0) \right)\right]^{-1} \\
&= I - \frac{s^2}{6} \Ap{i}(0) - \left[ \frac{\Sp{i}(s)}{s} -I - \frac{s^2}{6}\Ap{i}(0) \right]+ \sum_{j=2}^{\infty}(-1)^j\left( \frac{\Sp{i}(s)}{s} - I \right)^j \\
&= I - \frac{s^2}{6} \Ap{i}(0) + O(\| \Delta_i b \|^3),
\end{align*}
and hence,
\begin{align*}
\Fp{i} =  I + \frac{\Delta^2}{6} \Ap{i}(0) + \frac{\Delta^2}{3}\Ap{i+1}(0) + O(\| \Delta_i b \|^3 + \| \Delta_{i+1} b \|^3),
\end{align*}
where $\Fp{i}$ is defined in Eq. (\ref{eqn:Fi}).
Finally,
\begin{align*}
\Vp{i+1}(s) &= \Cp{i+1}(s) \Sp{i}(\Delta) - \Sp{i+1}(s) \Fp{i} \\
&= (\Delta-s)I +\left( \frac{\Delta^3 - s \Delta^2}{6} \right)\Ap{i}(0) + \left( \frac{3\Delta s^2 - 2\Delta^2 s - s^3}{6} \right)\Ap{i+1}(0) \notag \\
& ~~  + O(\Delta \{ \| \Delta_i b \|^3 + \| \Delta_{i+1} b \|^3 \}).
\end{align*}
\end{proof}

\begin{prop}
\label{prop:GFptaylor}
For sufficiently small $\epsi > 0$,
\begin{align}
 \int_0^{\Delta} \Sp{i}(s)^{\tr} \Sp{i}(s) ds =& ~\frac{\Delta^3}{3} I + \frac{\Delta^5}{15} \Ap{i}(0) + O( \Delta^3 \| \Delta_i b \|^3), \\
 \int_0^{\Delta} \Vp{i+1}(s)^{\tr} \Vp{i+1}(s) ds =& ~\frac{\Delta^3}{3} I + \frac{\Delta^5}{9} \Ap{i}(0) - \frac{2 \Delta^5}{45} \Ap{i+1}(0) \notag \\
 & + O( \Delta^3 \{\| \Delta_i b\|^3 + \| \Delta_{i+1} b \|^3 \}) \\
 \int_0^{\Delta} \Vp{i+1}(s)^{\tr} \Sp{i+1}(s) ds =&~ \frac{\Delta^3}{6}I + \frac{13 \Delta^5}{360} \Ap{i}(0) - \frac{7 \Delta^5}{360}\Ap{i+1}(0) \notag \\
 &+ O(\Delta^3 \{ \| \Delta_i b \|^3 + \| \Delta_{i+1} b \|^3 \}) 
\end{align}
on $\HpeR$. Moreover, the bounding constant can be taken independent of $i$. 
\end{prop}
\begin{proof}
This follows from multiplying together the appropriate operators using the estimates from Lemma \ref{lem:SpCpVpest} and integrating over $[0,\Delta]$ keeping in mind that $(\Ap{i})^{\tr} = \Ap{i}$.
\end{proof}

\noindent In light of Proposition \ref{prop:GFptaylor}, we decompose $\Rp$ (on $\HpeR$) in the following way, 
\begin{align}
\Rp =  \Up + \ep \label{eqn:GFpLpUpep}
\end{align}
where $\Up$ is defined by 
\begin{align}
\left[ \Up \right]_{i,j} = 
\frac{\Delta^5}{360} 
\begin{cases}
16(4 \Ap{i}(0) - \Ap{i+1}(0)) & i=j\\
13 \Ap{i}(0) - 7 \Ap{i+1}(0) & i=j+1, i+1=j\\
0 & \text{otherwise}
\end{cases} \label{eqn:Up}
\end{align}
with $\Ap{n+1} \equiv 0$, and the blocks of $\ep$ have size estimates,
\begin{align}
\left[ \ep \right]_{i,j} = 
\begin{cases}
O(\Delta^3 \{\| \Delta_i b \|^3 + \| \Delta_{i+1} b \|^3 \}) & i=j, i=j+1, i+1=j \\
0 & \text{otherwise}
\end{cases} \label{eqn:ep}
\end{align}
where the bounding constant can be taken independent of $i$ and further depends only on $\epsi$ and the curvature of the manifold, and remains bounded as $\epsi \to 0$. 

\section{The Proof of Theorem \ref{thm.mainthm}}
\begin{proof}[Proof of Theorem \ref{thm.mainthm}] We decompose the limit in Eq (\ref{eqn.mainthm}) into three parts, each which we show goes to zero:
\begin{align}
\lim_{|\p|\to 0}\left| \int_{\Hp(M)} f(\sigma) d\nuG{\p}(\sigma) - \int_{\Hpe(\re^d)} f(\phi(\w)) \rhot(\w) d\muS{\p}(\w) \right|, \label{eqn.limit1}
\end{align}
\begin{align}
\lim_{|\p|\to 0}\left| \int_{\Hpe(\re^d)} f(\phi(\w)) \left[ \rhot(\w) - e^{-\frac{2+\sqrt{3}}{20\sqrt{3}} \int_0^1 \Scal(\phi(\w)(s))ds} \right] d\muS{\p}(\w) \right|, \label{eqn.limit2}
\end{align}
and
\begin{align}
\lim_{|\p|\to 0}\left| \int_{\Hpe(\re^d)} f(\phi(\w)) e^{-\frac{2+\sqrt{3}}{20\sqrt{3}} \int_0^1 \Scal(\phi(\w)(s))ds} d\muS{\p}(\w) \right. \qquad \label{eqn.limit3} \\
 - \left. \int_{W(M)} f(\sigma)e^{-\frac{2+\sqrt{3}}{20\sqrt{3}} \int_0^1 \Scal(\sigma(s))ds} d\nu(\sigma)  \right|. \notag 
\end{align}

Proposition \ref{prop.limit13} below shows Eqs. (\ref{eqn.limit1}) and (\ref{eqn.limit3}) are both zero. Proposition \ref{prop.limit2} below shows Eq. (\ref{eqn.limit2}) is zero. This completes the proof of Theorem \ref{thm.mainthm} once these two propositions have been established.
\end{proof}

This proof holds true to our general scheme, where Eq. (\ref{eqn.limit1}) lets us translate from the manifold setting to flat space, Eq. (\ref{eqn.limit2}) finds the limit of $\rhot_{\p}$ in flat space, and Eq. (\ref{eqn.limit3}) then translates the flat space limit back to the manifold setting.

We already have the enough machinery to take care of two of these limits immediately.
\begin{prop}
\label{prop.limit13}
Under the hypotheses of Theorem \ref{thm.mainthm}, the limits in Eqs. (\ref{eqn.limit1}) and (\ref{eqn.limit3}) are zero.
\end{prop}
\begin{proof}
From Corollary \ref{cor.HpOutHpe} it suffices to assume the space $\Hpe(\re^d)$ is replaced with $\Hp(\re^d)$. The fact that Eq. (\ref{eqn.limit1}) vanishes is now just a quick application of Theorem \ref{thm:lawofbp} along with the recollection that on $\Hp(\re^d)$, $\rhot_{\p} = \rho_{\p}(\phi)$. Another application of Theorem \ref{thm:lawofbp} reduces Eq. (\ref{eqn.limit3}) to,
\begin{align*}
\lim_{|\p|\to 0}\left| \int_{\Hp(M)} f(\sigma) e^{-\frac{2+\sqrt{3}}{20\sqrt{3}} \int_0^1 \Scal(\sigma(s))ds} d\nuS{\p}(\sigma) \right. \qquad \\
 - \left. \int_{W(M)} f(\sigma)e^{-\frac{2+\sqrt{3}}{20\sqrt{3}} \int_0^1 \Scal(\sigma(s))ds} d\nu(\sigma)  \right|.  
\end{align*}
However, the fact that this vanishes is the substance of Theorem \ref{thm.AndDrive} and precisely why we chose to compare $\nuG{\p}$ to $\nuS{\p}$.
\end{proof}

 Therefore, we need only focus on the limit in Eq. (\ref{eqn.limit2}) to complete the proof of Theorem \ref{thm.mainthm}.

\subsect{Finding the Limit of Equation (\ref{eqn.limit2})}
It is fruitful for us to rewrite $\rhot_{\p}$ in Eq. (\ref{eqn.rhotp2}) as
\begin{align}
\rhot_{\p} &= \sqrt{ \det \left( I + \mLp^{-1/2} \Rp \mLp^{-1/2} \right) } \notag \\
&= \sqrt{ \det \left( I + \mLp^{-1/2} (\Up + \ep) \mLp^{-1/2} \right) }, \label{eqn.rhotp4}
\end{align}
keeping Eq. (\ref{eqn:GFpLpUpep}) in mind.

At first glance it might seem somewhat mysterious that $\rhot_{\p}$ limits to the desired exponential term, however, we have the following lemma.

\begin{lem}
\label{lem:ExpDetForm}
Let $U \in \fancyL(V)$ with $\| U \| < 1$. Define $\T_k(U) = \sum_{m=1}^k \frac{(-1)^{m+1}}{m} \tr(U^m)$ and $\Psi_{k+1}(U) = \sum_{m=k+1}^{\infty}\frac{(-1)^{m+1}}{m} \tr(U^m)$. Then,
\begin{align}
\det(I + U) = \ex \left\{ \T_k(U) + \Psi_{k+1}(U) \right\} \label{eqn:ExpDetForm}
\end{align}
with
\begin{align}
| \Psi_{k+1}(U) | \leq \frac{N \| U \|^{k+1}}{1- \|U\|}. \label{eqn:ExpDetSize1}
\end{align}
If we further assume that $U$ is normal then,
\begin{align}
| \Psi_{k+1}(U) | \leq  \| U \|_2^2\frac{ \|U\|^{k-1}}{1- \|U\|}. \label{eqn:ExpDetSize2}
\end{align}
\end{lem}
\begin{proof}
With $\| U \| < 1$, $\Re(\det(I+U)) > 0$, and hence we have the familiar formula 
\begin{align}
\log(\det(I + U)) = \sum_{m=1}^{\infty} \frac{(-1)^{m+1}}{m} \tr(U^{m}).
\end{align}
Eq. (\ref{eqn:ExpDetForm}) now results. Eq. (\ref{eqn:ExpDetSize1}) then comes from estimating the trace as $|\tr(U^m)| \leq N \| U \|^m$ which yields,
\begin{align*}
\left| \Psi_{k+1}(U) \right| \leq N \sum_{m=k+1}^{\infty} \frac{\| U \|^m}{ m } \leq N \frac{ \| U \|^{k+1}}{1- \|U\|}.
\end{align*}
If $U$ is normal then, $|U^2| = |U|^2$. For $m \geq 2$,  
\begin{align*}
|\tr(U^m)| \leq \|U\|^{m-2} \tr(|U^2|) =  \|U\|^{m-2} \tr(|U|^2)=  \|U\|^{m-2} \|U\|_2^2.
\end{align*} Thus for $k \geq 1$,
\begin{align*}
\left| \Psi_{k+1}(U) \right| \leq  \sum_{m=k+1}^{\infty} \frac{\| U \|^{m-2}}{ m } \|U\|_2^2 \leq  \| U \|_2^2\frac{ \|U\|^{k-1}}{1- \|U\|}.
\end{align*}
\end{proof}

Hence on the event that $\| \mLp^{-1/2} (\Up + \ep) \mLp^{-1/2} \| < 1$ we can apply Lemma \ref{lem:ExpDetForm}, yielding 
\begin{align}
\rhot_{\p} = \ex \frac{1}{2} \left\{  \tr(\mLp^{-1/2} \Up \mLp^{-1/2}) + \tr(\mLp^{-1/2} \ep \mLp^{-1/2}) +  \Psi_2(\mLp^{-1/2} \Rp \mLp^{-1/2}) \right\} \label{eqn:MainExpForm}
\end{align}
with
\begin{align}
\left| \Psi_2(\mLp^{-1/2} \Rp \mLp^{-1/2}) \right| \leq  \frac{\| \mLp^{-1/2} \Rp \mLp^{-1/2} \|^2_2}{1-\|\mLp^{-1/2} \Rp \mLp^{-1/2}\|}. \label{eqn:Psi2Bound}
\end{align}
Proposition \ref{prop:onlytrU} below shows that $\ex\frac{1}{2}\{ \tr(\mLp^{-1/2}\Up \mLp^{-1/2})\}$  is what contributes in the limit as $| \p | \to 0$, while the other factors vanish. This will then turn our focus to understanding the behavior of $\tr(\mLp^{-1/2} \Up \mLp^{-1/2})$.

\begin{lem}
\label{lem:trLULsize}
Let $\Up$ and $\ep$ be as in Eqs. (\ref{eqn:Up}) and (\ref{eqn:ep}) respectively. Then there exists some $C = C(d, \text{curvature})<\infty$ such that,
\begin{align}
| \tr(\mLp^{-1/2} \Up \mLp^{-1/2}) | \leq C \sum_{i=1}^n \| \Delta_i b \|^2 \label{eqn:trLUpLsize}
\end{align}
and
\begin{align}
| \tr(\mLp^{-1/2} \ep \mLp^{-1/2}) | \leq C \sum_{i=1}^n \| \Delta_i b \|^3. \label{eqn:trLepLsize}
\end{align}
\end{lem}
\begin{proof}
Both $\Up$ and $\ep$ are symmetric, so we can therefore apply Corollary \ref{cor:trLUL} to find some $\Lambda = \Lambda(\ddim)< \infty$ with 
\begin{align*}
| \tr(\mLp^{-1/2} \Up \mLp^{-1/2}) | \leq \frac{\Lambda}{\Delta^3} \sum_{i=1}^n  \left( |\tr([\Up]_{i,i})| + |\tr([\Up]_{i,i+1})| \right)
\end{align*}
and
\begin{align*}
| \tr(\mLp^{-1/2} \ep \mLp^{-1/2}) | \leq \frac{\Lambda}{\Delta^3} \sum_{i=1}^n  \left( \|[\ep]_{i,i}\| + \|[\ep]_{i,i+1}\| \right).
\end{align*}
Eq. (\ref{eqn:ep}) along with the above estimate is enough to imply Eq. (\ref{eqn:trLepLsize}). To finish the proof of Eq. (\ref{eqn:trLUpLsize}), it is sufficient to show that there is some $C = C(d, \text{curvature}) < \infty$ such that $|\tr([\Up]_{m,m})| $ and $|\tr([\Up]_{m,m+1})|$ are bounded by $C(\| \Delta_i b \|^2 + \| \Delta_{m+1} b \|^2)$. 
However, the bound on curvature along with Eqs.  (\ref{eqn:Aibound}) and (\ref{eqn:Up}) imply the existence of just such a $C$.
\end{proof}

\begin{lem}
\label{lem:remsize}
Define the event $\mathpzc{A}_{\p} \subset W(\re^d)$ by 
\begin{align*}
\mathpzc{A}_{\p} = \left\{ \sum_{i=1}^n \| \Delta_i b \|^4 \leq \frac{1}{4} \right\}.
\end{align*} 
For sufficiently small $\epsi$ and $\Delta$, there is a constant $C = C(d, \text{curvature}) < \infty$ such that 
\begin{align}
\left| \tr(\mLp^{-1/2} \ep \mLp^{-1/2}) \right| +  \left| \Psi_2(\mLp^{-1/2} \Rp \mLp^{-1/2}) \right| \leq C \sum_{i=1}^d \| \Delta_i b \|^3 \label{eqn:trPsisize}
\end{align}
on $\Hpe(\re^d) \cap \mathpzc{A}_{\p}$.
\end{lem}
\begin{proof}
Focusing on the second term, from Theorem \ref{thm:Lpthm},
\begin{align*}
\| \mLp^{-1/2} \Rp \mLp^{-1/2} \|_2^2 &\leq \| \mLp^{-1} \|^2 \| \Rp \|_2^2 \leq 16 \Delta^{-6}\| \Rp \|_2^2.
\end{align*}
From Proposition \ref{prop:RpDiagSize},
\begin{align*}
&\left\| [ \Rp]_{i,i}\right\|, \left\| [ \Rp]_{i,i+1}\right\|, \left\| [ \Rp]_{i+1,i}\right\| \\
&\qquad  \leq 2 \Delta^3 \left( \cosh(2 \sqrt{\Kp{i}} \Delta) \cosh(8 \sqrt{\Kp{i+1}} \Delta)-1 \right) \\
&\qquad  \leq 2 \Delta^3 \left( \cosh(2 \kappa \|\Delta_i b \|) \cosh(8 \kappa \|\Delta_{i+1} b \|)-1 \right)
\end{align*}
where the second inequality follows from Eq. (\ref{eqn:Kibound}). Further, on $\Hpe(\re^d)$,
\begin{align*}
\left( \cosh(2 \kappa \|\Delta_i b \|) \cosh(8 \kappa \|\Delta_{i+1} b \|)-1 \right)  \leq C ( \| \Delta_i b\|^2 + \| \Delta_{i+1} b \|^2 ).
\end{align*}
For sufficiently small $\epsi$ we can take $C$ such that $24 \ddim C^2 \leq 1/16$. Now, 
\begin{align*}
\| \Rp \|_2^2 &\leq d \sum_{i=1}^n \left(  \left\| [ \Rp]_{i,i}\right\|^2 + \left\| [ \Rp]_{i,i+1}\right\|^2 + \left\| [ \Rp]_{i+1,i}\right\|^2 \right) \\
&\leq 12 \ddim C^2 \Delta^6 \sum_{i=1}^n \left( \| \Delta_i b \|^2 + \| \Delta_{i+1} b\|^2 \right)^2\\
&\leq 24 \ddim C^2 \Delta^6 \sum_{i=1}^n  \| \Delta_i b \|^4 \leq \frac{1}{16} \Delta^6 \sum_{i=1}^n  \| \Delta_i b \|^4.
\end{align*}
Therefore $\| \mLp^{-1/2} \Rp \mLp^{-1/2} \|_2^2 \leq \sum_{i=1}^n  \| \Delta_i b \|^4$. This further implies that on $\pzc{A}_{\p}$, 
\begin{align*}
\| \mLp^{-1/2} \Rp \mLp^{-1/2} \| &\leq \| \mLp^{-1/2} \Rp \mLp^{-1/2} \|_2 \leq \sqrt{\sum_{i=1}^n  \| \Delta_i b \|^4} \leq \frac{1}{2}.
\end{align*}
Hence, $1 - \| \mLp^{-1/2} \Rp \mLp^{-1/2} \| \geq \frac{1}{2}$ and we have
\begin{align*}
\left| \Psi_2(\mLp^{-1/2} \Rp \mLp^{-1/2}) \right| &\leq \frac{\| \mLp^{-1/2} \Rp \mLp^{-1/2} \|_2^2}{1 - \| \mLp^{-1/2} \Rp \mLp^{-1/2} \|} \leq 2 \sum_{i=1}^n \| \Delta_i b \|^4.
\end{align*}

Eq. (\ref{eqn:trLepLsize}) in Lemma \ref{lem:trLULsize} above gives the existence of some $\Lambda <\infty$ depending only on curvature and $d$ such that,
\begin{align*}
 |\tr(\mLp^{-1/2} \ep \mLp^{-1/2}) | &\leq \Lambda \sum_{i=1}^n \| \Delta_i b \|^3.
\end{align*}
Therefore on $\Hpe(\re^d) \cap \mathpzc{A}_{\p}$,
\begin{align*}
&\left| \tr(\mLp^{-1/2} \ep \mLp^{-1/2}) \right| + \left| \Psi_2( \mLp^{-1/2} \Rp \mLp^{-1/2}) \right| \\
 &\leq \Lambda \sum_{i=1}^n \| \Delta_i b \|^3 +2 \sum_{i=1}^n \| \Delta_i b \|^4 \leq ( \Lambda + 2 \epsi)\sum_{i=1}^n \| \Delta_i b \|^3 \leq C \sum_{i=1}^n \| \Delta_i b \|^3
\end{align*}
which concludes the proof.
\end{proof}

\begin{prop}
\label{prop:onlytrU}
Let $X_{\p} :=  \frac{1}{2} \tr \left( \mLp^{-1/2} \Up \mLp^{-1/2} \right)$. Then for sufficiently small $\Delta$ and $\epsi$, there exists a $C < \infty$ depending only on $d$ and the bound on the curvature of $M$ such that,
\begin{align}
\int_{\Hpe(\re^d)} \left| \rhot_{\p} - e^{X_{\p}} \right| d\muS{\p}  \leq C \Delta^{1/4}.
\end{align}
\end{prop}
\begin{proof} Let $\fancyA_{\p}$ be as in Lemma \ref{lem:remsize}. 
Define the random variable $y_{\p}$ by 
\begin{align*} 
y_{\p} := \frac{1}{2} \left( \tr(\mLp^{-1/2} \ep \mLp^{-1/2}) + \Psi_2(\mLp^{-1/2} \Rp \mLp^{-1/2})\right).
\end{align*} 
By Theorem \ref{thm:UniformIntegrability}, $\E[\rhot_{\p}^2] < \infty$ independent of $n$. By Lemma \ref{lem.EeXminusY} and Eq. (\ref{eqn:MainExpForm}), we only need to show that $\left| \E[ e^{2|y_{\p}|} - 1]\right| \leq C \Delta^{1/2}$, which follows from  Eq. (\ref{eqn:trPsisize}) along with Corollary \ref{cor:sumbto0}. 

On the compliment $F := \Hpe(\re^d) \cap \fancyA_{\p}^c$,
\begin{align*}
 \int_{F} \left| \rhot_{\p} - e^{X_{\p}} \right| d\muS{\p} &\leq   \int_{F}\left( \left| \rhot_{\p}\right| + \left| e^{X_{\p}} \right| \right) d\muS{\p} \\
 &\leq \left( \E \left[ |\rhot_{\p} |^2 \right]^{1/2}  + \E \left[ e^{2X_{\p}} \right]^{1/2} \right)  \E \left[ 1_{\left\{ \sum_{i=1}^n \| \Delta_i b \|^4 > 1/4 \right\}}\right]^{1/2}.
\end{align*}
Another application of Lemma \ref{lem:EDb2} and using Theorem \ref{thm:UniformIntegrability} ensures that $ \E \left[ |\rhot_{\p} |^2 \right]^{1/2}  + \E \left[ e^{2X_{\p}} \right]^{1/2}$ stays bounded for sufficiently small $\Delta$. Noticing that 
\begin{align*}
\left\{ \sum_{i=1}^n \| \Delta_i b \|^4 > \frac{1}{4} \right\} &\subset \bigcup_{i=1}^n  \left\{ \| \Delta_i b \|^4 > \frac{1}{4n} \right\} \\
&\ed \bigcup_{i=1}^n\left\{ \| Z_i \| > \left( \frac{n}{4} \right)^{1/4} \right\},
\end{align*}
where $\{ Z_i \}$ are i.i.d. with $Z_i \ed b_1$.
Lemma \ref{lem:GaussBound} gives,
\begin{align*}
\E \left[ 1_{\left\{ \sum_{i=1}^n \| \Delta_i b \|^4 > 1/4 \right\}}\right] &\leq \sum_{i=1}^n \frac{C}{\sqrt{n}} e^{-\frac{\sqrt{n}}{16}} = \frac{C}{\sqrt{ \Delta}} e^{-\frac{1}{16\sqrt{\Delta}}}.
\end{align*}
Combining these facts implies the claim.
\end{proof}

$\newline$
$\newline$
With Proposition \ref{prop:onlytrU}, we turn our attention to $\tr \left( \mLp^{-1/2} \Up \mLp^{-1/2} \right)$.

\begin{lem} 
\label{lem:LpUptrace}
For $\betap_k, \alpha^m_k$ and $\lambda^{\p}_k$ as in Theorem \ref{thm:Lpthm},
\begin{align}
\tr \left( \mLp^{-1/2} \Up \mLp^{-1/2} \right) &= \sum_{m,k=1}^n \frac{(\betap_k)^2}{\lambda^{\p}_k}\left[ (\alpha^m_k)^2 \tr([\Up]_{m,m}) + 2 \alpha_k^{m} \alpha_k^{m+1} \tr([\Up]_{m,m+1}) \right] \label{eqn:trformula1} \\
&= - \sum_{m,k=1}^n \frac{(\betap_k)^2}{\lambda^{\p}_k} \Delta^3 \langle \Ric_{u(s_{m-1})} \Delta_m b, \Delta_m b \rangle \xi_{k,m} \label{eqn:trformula2}
\end{align}
where for $1 \leq m \leq n$ (with $\alpha^{n+1}_k = \alpha^0_k := 0$),
\begin{align}
\xi_{k,m} = \frac{2}{45}\left[ 4 (\alpha_k^m)^2 -  (\alpha_k^{m-1})^2\right] + \frac{ 1 }{180} \alpha_k^m \left[ 13\alpha_k^{m+1} - 7\alpha_k^{m-1} \right]. \label{eqn.xikm}
\end{align}
\end{lem}
\begin{proof}
Eq. (\ref{eqn:trformula1}) follows by applying Corollary \ref{cor:trLUL}. By the definition of $\Up$ in Eq \ref{eqn:Up}, 
\begin{align*}
\tr \left( [\Up]_{m,m} \right) &= \frac{2 \Delta^5}{45}   \tr( 4 \Ap{m}(0) - \Ap{m+1}(0) ) \quad \text{ and,} \\
\tr\left( [\Up]_{m,m+1} \right) &= \frac{\Delta^5}{360}   \tr( 13 \Ap{m}(0) - 7 \Ap{m+1}(0) ).
\end{align*}
This implies,
\begin{align*}
(\alpha^m_k)^2 \tr([\Up]_{m,m}) &+ 2 \alpha_k^{m} \alpha_k^{m+1} \tr([\Up]_{m,m+1})\\
&= \Delta^5 \left[ \tr(\Ap{m}(0))\left( \frac{8}{45} (\alpha_k^m)^2 + \frac{13}{180} \alpha_k^m \alpha_k^{m+1} \right) \right. \\
&\qquad \qquad - \left. \tr(\Ap{m+1}(0)) \left( \frac{2}{45} (\alpha_k^m)^2 + \frac{7}{180} \alpha_k^m \alpha_k^{m+1} \right)\right].
\end{align*}
Hence, 
\begin{align*}
\sum_{m=1}^n (\alpha^m_k)^2 \tr([\Up]_{m,m}) + 2 \alpha_k^{m} \alpha_k^{m+1} \tr([\Up]_{m,m+1}) = \sum_{m=1}^n \Delta^5 \tr(\Ap{m}(0)) \xi_{k,m}.
\end{align*}
This leads to Eq. (\ref{eqn:trformula2}) by Eqs. (\ref{eqn:Ai}) and (\ref{eqn:trAp}).
\end{proof}

\begin{cor}
\label{cor.bndrytermsto0}
Using the same notation as Corollary \ref{cor:ksumsize}, there is some $C = C(\text{curvature})$ such that
\begin{align}
\sum_{m \in \W_{\delta}} \frac{(\betap_n)^2}{\lambda_n^{\p}}\Delta^3 \langle \Ric_{u(s_{m-1})} \Delta_m b, \Delta_m b \rangle \xi_{n,m} < C\left(\frac{4}{9}\right)^{\frac{\delta n}{ \pi}} \sum_{m\in\bdryc{\delta}} \| \Delta_m b\|^2 \label{eqn.bdrycn}
\end{align}
and
\begin{align}
\sum_{m \in \bdry{\delta}} \frac{(\betap_n)^2}{\lambda_n^{\p}}\Delta^3 \langle \Ric_{u(s_{m-1})} \Delta_m b, \Delta_m b \rangle \xi_{n,m} \leq C \sum_{m\in\bdry{\delta}} \| \Delta_m b\|^2 \label{eqn.bdryn}
\end{align}
\end{cor}
\begin{proof}
Since $\Ric$ and $\Delta^3 / \lambda_n^{\p}$ are bounded independent of $n$, this result comes down to understanding the size of $(\betap_n)^2 \xi_{n,m}$. Recall that in Theorem \ref{thm:Lpthm}, for $n \geq 2$, $\gamma_n \in (-2, -3/2)$ and $\alpha_n^m = \gamma_n^m - \gamma_n^{-m}$. From Eq. (\ref{eqn.bn}), $(\betap_n)^2 = O(\gamma_n^{-2n})$, and from Eq. (\ref{eqn.xikm}), $\xi_{n,m} = O( \gamma_n^{2m})$. Therefore $(\betap_n)^2 \xi_{n,m} = O(\gamma_n^{-2(n-m)})$, which is enough to conclude Eq. (\ref{eqn.bdryn}).

For $m \in \W_{\delta}$, $n - m > (\delta/\pi) n$ and thusly, 
\begin{align*}
\sum_{m \in \W_{\delta}} \frac{(\betap_n)^2}{\lambda_n^{\p}}\Delta^3 \langle \Ric_{u(s_{m-1})} \Delta_m b, \Delta_m b \rangle \xi_{n,m} \leq C \left( \gamma_n^{-2} \right)^{\frac{\delta n}{\pi}} \sum_{m \in \W_{\delta}} \| \Delta_m b\|^2 \\
 <  C \left( \frac{4}{9} \right)^{\frac{\delta n}{\pi}} \sum_{m \in \bdryc{\delta}} \| \Delta_m b\|^2
\end{align*}
which is Eq.\,(\ref{eqn.bdrycn}).
\end{proof}

An important fact that will be used is that from the definition of $\alpha_k^m$ in Theorem \ref{thm:Lpthm} for $2 \leq m \leq n-2$ and $1 \leq k \leq n-1$,
\begin{align}
\xi_{k,m} &= \frac{1}{15} + \frac{1}{60}\cos(\tp_{k}) + \rem_{k,m} \label{eqn:xikm}
\end{align}
where the remainder term $\rem_{k,m}$ is given by,
\begin{align}
\rem_{k,m} = \cos(2m\tp_k)\left\{ \frac{1}{45}( \cos(2\tp_k) - 4 ) - \frac{1}{60}\cos(\tp_k) \right\} \notag \\
 + \sin(2m\tp_k)\left\{ \frac{1}{45} \sin(2\tp_k) - \frac{1}{60}\sin(\tp_k) \right\} \label{eqn:remxi}.
\end{align}
It's useful to write $\xi_{k,m}$ in this fashion since what follows shows that as $| \p | \to 0$, the terms involving $\rem_{k,m}$ vanish.

\begin{lem} 
\label{lem:ksumsize}
Let $\delta > 0$. Using the notation from Corollary \ref{cor:ksumsize}, there exists a $C = C(\text{curvature}) < \infty$ such that,
\begin{align}
& \left| \sum_{m = 2}^{n-1} \sum_{k=1}^{n-1} \frac{(\betap_k)^2}{\lambda^{\p}_k} \Delta^3 \langle \Ric_{u(s_{m-1})} \Delta_m b, \Delta_m b \rangle  \rem_{k,m}  \right| \notag \\
& \qquad \leq \frac{C}{n \sin(\delta)} \sum_{m=1}^n \| \Delta_m b \|^2 + C \sum_{m \in \bdry{\delta}} \| \Delta_m b \|^2.
\end{align}
\end{lem}
\begin{proof}
Using Corollary \ref{cor:ksumsize} and Eq. (\ref{eqn:remxi}),
\begin{align*}
&\left| \sum_{m=2}^{n-1} \sum_{k=1}^{n-1} \frac{(\betap_k)^2}{\lambda^{\p}_k} \Delta^3 \langle \Ric_{u(s_{m-1})} \Delta_m b, \Delta_m b \rangle  \rem_{k,m}  \right| \\
& \qquad \leq \sum_{m =2}^{n-1} \left|  \langle \Ric_{u(s_{m-1})} \Delta_m b, \Delta_m b \rangle  \sum_{k=1}^{n-1}  \frac{(\betap_k)^2}{\lambda^{\p}_k} \Delta^3 \rem_{k,m} \right| \\
&\qquad \leq \frac{C}{n \sin(\delta)} \sum_{m\in \bdryc{\delta}} \| \Delta_m b \|^2  + C \sum_{m \in \bdry{\delta}} \| \Delta_m b \|^2\\
&\qquad \leq \frac{C}{n \sin(\delta)} \sum_{m=1}^n \| \Delta_m b \|^2 + C \sum_{m \in \bdry{\delta}} \| \Delta_m b \|^2.
\end{align*}
\end{proof}

From here we are able to show that what actually contributes in the limit allows us to ignore the $\rem_{k,m}$ term. Moreover, we will see that those boundary cases of $\xi_{k,m}$ for $m = 1, n-1, n$ are also negligible and allow us to further simplify our expression when passing to the limit.

\begin{prop} 
\label{prop:LpUptrSimp2}
For sufficiently small $\epsi$ and $\Delta$,
\begin{align}
\int_{\Hpe(\re^d)} \left| e^{X_{\p}} - e^{Y_{\p}} \right| d\muS{\p} \leq  C\Delta^{1/4}
\end{align} 
where $C = C(d, \text{curvature})<\infty$, 
\begin{align*}
Y_{\p} := - \frac{1}{2} \sum_{m=2}^{n-2}\sum_{k=1}^{n-1} \frac{(\betap_k)^2}{\lambda^{\p}_k} \Delta^3  \left( \frac{1}{15} + \frac{1}{60}\cos(\tp_k) \right) \langle \Ric_{u(s_{m-1})} \Delta_m b, \Delta_m b \rangle
\end{align*}
and as before $X_{\p} := \frac{1}{2} \tr(\mLp^{-1/2} \Up \mLp^{-1/2})$. 
\end{prop}
\begin{proof} We first have the estimate,
\begin{align*}
| e^{X_{\p}} - e^{Y_{\p}} | = e^{Y_{\p}} | e^{X_{\p}-Y_{\p}} - 1 | &\leq e^{Y_{\p}}\left( e^{\partial X_{\p}} | e^{Z_{\p}} - 1 | + |e^{\partial X_{\p}}-1| \right) \\
&= e^{y_{\p}} | e^{Z_{\p}} - 1 | + e^{Y_{\p}}|e^{\partial X_{\p}}-1|
\end{align*} 
where,
\begin{align*}
&\partial X_{\p} := - \frac{1}{2} \sum_{m \in \{1,n-1,n\}} \sum_{k=1}^n \frac{(\betap_k)^2}{\lambda^{\p}_k} \Delta^3 \langle \Ric_{u(s_{m-1})} \Delta_m b, \Delta_m b \rangle \xi_{k,m} \\
&y_{\p} := Y_{\p} + \partial X_{\p}, 
\end{align*}
and
\begin{align*}
Z_{\p} &:= (X_{\p} - \partial X_{\p}) - Y_{\p} \\
&= - \frac{1}{2} \sum_{m=2}^{n-2} \bigg( \sum_{k=1}^{n-1} \frac{(\betap_k)^2}{\lambda^{\p}_k} \Delta^3 \langle \Ric_{u(s_{m-1})} \Delta_m b, \Delta_m b \rangle \rem_{k,m} \\ 
& \qquad \qquad \qquad   + \frac{(\betap_n)^2}{\lambda^{\p}_n} \Delta^3 \langle \Ric_{u(s_{m-1})} \Delta_m b, \Delta_m b \rangle \xi_{n,m} \bigg)
\end{align*}
There is some $\Lambda = \Lambda(\text{curvature})$ so that we have the following size estimates,
\begin{align*}
|\partial X_{\p}| &\leq \sum_{k=1}^n \frac{\Lambda}{n} \left( \| \Delta_1 b \|^2 + \| \Delta_{n-1} b \|^2 + \| \Delta_{n} b \|^2 \right)\\
&\leq \Lambda  \left( \| \Delta_1 b \|^2 + \| \Delta_{n-1} b \|^2 + \| \Delta_{n} b \|^2 \right), 
\end{align*}
and from Eq. (\ref{eqn.bk}),
\begin{align*}
|Y_{\p}| \leq \Lambda \sum_{m=2}^{n-2} \sum_{k=1}^{n-1} \frac{1}{n} \| \Delta_m b\|^2 \leq \Lambda \sum_{m=1}^n \| \Delta_m b \|^2,
\end{align*}
and hence,
\begin{align*}
|y_{\p}| \leq \Lambda \sum_{m=1}^n \| \Delta_m b\|^2.
\end{align*}
From Lemma \ref{lem:ksumsize} and Corollary \ref{cor.bndrytermsto0}, with $\delta = 1/\sqrt{n}$,
\begin{align*}
|Z_{\p}| &\leq \Lambda \bigg(\frac{1}{\sqrt{n}} + \left( \frac{4}{9}\right)^{\frac{\sqrt{n}}{\pi}} \bigg) \sum_{m=1}^n \| \Delta_m b\|^2 + \Lambda \sum_{m \in \bdry{\delta}} \| \Delta_m b \|^2 \\
&\leq \frac{\Lambda}{\sqrt{n}}\sum_{m=1}^n \| \Delta_m b\|^2 + \Lambda \sum_{m \in \bdry{\delta}} \| \Delta_m b \|^2
\end{align*}
where we allowed $\Lambda$ to grow in the second inequality to absorb the exponentially decaying term. Using H\"{o}lder's inequality,
\begin{align*}
\int_{\Hpe(\re^d)} | e^{X} - e^{Y} | d\muS{\p} \leq \left[ \int_{\Hpe(\re^d)} e^{2 y} d\muS{\p} \right]^{1/2} \left[ \int_{\Hpe(\re^d)} (e^{Z}- 1)^2 d\muS{\p} \right]^{1/2} \\ 
+ \left[ \int_{\Hpe(\re^d)} e^{2Y}  d\muS{\p} \right]^{1/2}  \left[ \int_{\Hpe(\re^d)}(e^{|\partial X|} - 1)^2 d\muS{\p} \right]^{1/2}
\end{align*}
where we have simplified notation by dropping the subscript $\p$ from  $y, \tilde{y}, X, Y$, and $Z$.
For sufficiently small $\Delta$ Lemma \ref{lem:EDb2} ensures that, 
\begin{align*}
\int_{\Hpe(\re^d)} e^{2 Y} d\muS{\p} \leq C ~~ \text{ and } ~~\int_{\Hpe(\re^d)} e^{2 y} d\muS{\p} \leq C.
\end{align*}
From Corollary \ref{cor:sumbto0},
\begin{align*}
\int_{\Hpe(\re^d)}(e^{|\partial X|} - 1)^2 d\muS{\p} \leq C \Delta.
\end{align*}
For the term involving $Z$, 
\begin{align*}
& \int_{\Hpe(\re^d)} (e^Z - 1)^2 d\muS{\p} \leq 2\E\left[ |Z| e^{2|Z|} \right] \\
 &\qquad \leq  \E \left[ \left( \frac{\Lambda}{\sqrt{n}} \sum_{m=1}^n \| \Delta_m b \|^2 +\Lambda \sum_{m \in \bdry{\delta}} \| \Delta_m b \|^2 \right) \ex\left\{ \Lambda \sum_{m=1}^n \| \Delta_m b \|^2 \right\} \right]^{1/2} 
 \end{align*}
Using Lemma \ref{lem:sumbpto0}, 
\begin{align*}
\E \left[ \frac{\Lambda}{\sqrt{n}} \sum_{m=1}^n \| \Delta_m b \|^2  e^{ 2 \Lambda \sum_{m=1}^n \| \Delta_m b \|^2} \right] \leq \frac{C}{\sqrt{n}} =C \Delta^{1/2}, 
\end{align*}
and 
\begin{align*}
\E \left[ \Lambda \sum_{m \in \bdry{\delta}} \| \Delta_m b \|^2  e^{ 2 \Lambda \sum_{m=1}^n \| \Delta_m b \|^2} \right] \leq C\delta = \frac{C}{\sqrt{n}} = C \Delta^{1/2}\\
\end{align*}
This is enough to conclude the proof.
\end{proof}

Defining $Y_{\p}$ as in  Proposition \ref{prop:LpUptrSimp2}, we rearrange the expression using the definition of $\lambda_k^{\p}$ from Theorem \ref{thm:Lpthm},
\begin{align}
Y_{\p} &= - \frac{1}{40} \left\{ \sum_{k=1}^{n-1} (\betap_k)^2\frac{4 + \cos(\tp_k)}{2 + \cos(\tp_k)} \right\} \left\{ \sum_{m=2}^{n-2} \langle \Ric_{u(s_{m-1})} \Delta_m b, \Delta_m b \rangle\right\}. \label{eqn:yp}\\
&=: - \tau_{\p} \sum_{m=2}^{n-2} \langle \Ric_{u(s_{m-1})} \Delta_m b, \Delta_m b \rangle \label{eqn:taup}.
\end{align}
Corollary \ref{cor:BetaInt} above shows that 
\begin{align}
\tau_{\p} \to \frac{1}{20} \int_0^1 \frac{4 + \cos(\pi x)}{2+\cos(\pi x)} dx = \frac{1}{20}\left( 1 + \frac{2}{\pi} \int_0^{\pi} \frac{dx}{2 + \cos(x)}\right). \label{eqn.tauplimit}
\end{align}
To evaluate the integral $\int_0^{\pi} (2 + \cos(x))^{-1} dx$, we use the residue theorem from complex variables with the substitution $z = e^{i x}$,
\begin{align*}
\int_0^{\pi} \frac{dx}{2+\cos( x)} dx = \frac{1}{2} \int_0^{2 \pi} \frac{dx}{2 + \cos(x)}  = \frac{1}{2} \oint_{ | z | = 1} \frac{- i z^{-1} dz}{2 + \frac{1}{2}(z + z^{-1})} \\ 
= - i \oint_{|z|=1} \frac{dz}{z^2 + 4z + 1} =  - i \oint_{|z|=1} \frac{dz}{(z + 2 - \sqrt{3})(z+2 + \sqrt{3})} \\ 
= 2 \pi i \times \left.\left( \frac{-i}{z + 2 + \sqrt{3}} \right)\right|_{z = - 2 + \sqrt{3}} = \frac{\pi}{\sqrt{3}}.
\end{align*}
From here we define 
\begin{align}
\tau_G := \frac{1}{20} \int_0^1 \frac{4 + \cos(\pi x)}{2+\cos(\pi x)} dx = \frac{2 + \sqrt{3}}{20\sqrt{3}}. \label{eqn:tauG}
\end{align}
By Eq. (\ref{eqn.tauplimit}), $\tau_{\p} \to \tau_G$ as $n \to \infty$.
\begin{prop}
\label{prop:LpUptrSimp3}
Let $Y_{\p}$ be as in Eq. (\ref{eqn:yp}), $\tau_{\p}$ be as in Eq. (\ref{eqn:taup}), and $\tau_{G}$ be as in Eq. (\ref{eqn:tauG}). There is a constant $C = C(d, \text{curvature}) < \infty$ such that,
\begin{align}
\int_{\Hpe(\re^d)} | e^{Y_{\p}(\w)} - e^{-\tau_{G}  \int_0^1 \Scal(\phi(\w)(s))} | d\muS{\p}(\w) \leq C ( \sqrt{ | \tau_{\p} - \tau_{G} |} + \Delta^{1/4}).
\end{align}
\end{prop}
\begin{proof}
Breaking the integrand into pieces we consider,
\begin{align*}
 \underbrace{(e^{Y_{\p}(\w)} - e^{-\tau_{G} \fancyR_{\p}})}_{I} + \underbrace{(e^{-\tau_{G} \fancyR_{\p}}- e^{-\tau_{G} \fancyS_{\p}})}_{II} + \underbrace{(e^{-\tau_{G} \fancyS_{\p}} - e^{-\tau_{G}  \int_0^1 \Scal(\phi(\w)(s))})}_{III}.
\end{align*}
Let $\Lambda  = \Lambda(\text{curvature}) < \infty$ be given such that $|\Scal| \leq \Lambda$. Then,
\begin{align*}
|e^{-\tau_{G} \fancyR_{\p}}- e^{-\tau_{G} \fancyS_{\p}}| \leq e^{\tau_G \Lambda} | e^{-\tau_{G} (\fancyR_{\p}- \fancyS_{\p}) } - 1 |
\end{align*}
Now applying Lemma \ref{lem.EeXminusY} and Lemma \ref{lem:fancyRS1}, 
 \[ \int_{\Hpe(\bbR^d)}  \big| II \big| \leq e^{\tau_G \Lambda} \int_{\Hp(\bbR^d)}  \big| e^{-\tau_{G} (\fancyR_{\p}- \fancyS_{\p}) } - 1 \big| \leq C(e^{C \Delta} - 1)^{1/2}.\]  
 Similarly, with
\begin{align*}
&\left| e^{-\tau_{G} \fancyS_{\p}} - e^{-\tau_{G}  \int_0^1 \Scal(\phi(\w)(s))}\right| \\
&\qquad \leq e^{\tau_G \Lambda} \left(\ex\left\{ \tau_G \left|\fancyS_{\p} - \int_0^1 \Scal(\phi(\w)(s))ds \right|\right\} - 1 \right),
\end{align*}
another application of Lemma \ref{lem.EeXminusY} to the right hand side along with Lemma \ref{lem:fancyS} gives $\int \big|III\big| \leq C \Delta^{1/4}$.

What remains then is to bound $\int \big| I \big|$. To start, we will assume that $\Lambda$ is also a bound on $\Ric$ so that  $| \langle \Ric_{u(s_{i-1})} \Delta_i b, \Delta_i b \rangle | \leq \Lambda \| \Delta_i b \|^2$ for each $i = 1, 2, ..., n$.
% Another application of Lemma \ref{lem.EeXminusY} ensures that we need only show the appropriate bound for
%\begin{align*}
%\E \left[ |\tau_{G} \fancyR_{\p} + Y_{\p} | e^{|\tau_{G} \fancyR_{\p} + Y_{\p} |} \right].
%\end{align*}
From here,
\begin{align*}
\tau_{G} \fancyR_{\p} + Y_{\p} = (\tau_{G} - \tau_{\p}) \fancyR_{\p} + \tau_{\p} \partial \fancyR_{\p}
\end{align*}
where 
\begin{align*}
\partial \fancyR_{\p} := \langle \Ric_{u(s_{0})} \Delta_1 b, \Delta_1 b \rangle + \langle \Ric_{u(s_{n-2})} \Delta_{n-1} b, \Delta_{n-1} b \rangle \\ + \langle \Ric_{u(s_{n-1})} \Delta_n b, \Delta_n b \rangle.
\end{align*}
Using the bounds
\begin{align*}
|\fancyR_{\p}| \leq \Lambda \sum_{i=1}^n \|\Delta_i b\|^2 ~~ \text{ and }~~|\partial \fancyR_{\p}| \leq \Lambda ( \| \Delta_1 b \|^2 + \| \Delta_{n-1} b \|^2 + \| \Delta_n b \|^2 )
\end{align*}
along with Eq. (\ref{eqn.eaminus1}), Eq. (\ref{eqn.deltabineq3}) in Lemma \ref{lem:sumbpto0}, and Theorem \ref{thm:lawofbp} we have,
\begin{align}
\int_{\Hpe(M)} e^{2 \tau_G |\fancyR_{\p}| } \, d\muS{\p} \leq  2 \tau_G \Lambda \sum_{i=1}^n\E \left[\|\Delta_i b \|^2 e^{2 \tau_G \Lambda \sum_{j=1}^n\|\Delta_j b \|^2 } \right] \leq C. \label{eqn.intI1}
\end{align}
Along these same lines, from Eq. (\ref{eqn.deltabineq3}),
\begin{align*}
&\int_{\Hpe(M)} \big| (\tau_{G} - \tau_{\p}) \fancyR_{\p}  \big| e^{| (\tau_{G} - \tau_{\p}) \fancyR_{\p} + \tau_{\p} \partial \fancyR_{\p}|} \, d\muS{\p} \\
&\qquad \leq  \E \left[ \big| (\tau_{G} - \tau_{\p}) \fancyR_{\p}  \big| e^{| (\tau_{G} - \tau_{\p}) \fancyR_{\p} + \tau_{\p} \partial \fancyR_{\p}|}\right] \\
&\qquad \leq |\tau_{G} - \tau_{\p}|  \Lambda \sum_{i=1}^n\E \left[   \|\Delta_i b\|^2 e^{ 2\Lambda \sum_{j=1}^n\|\Delta_j b\|^2}\right] \leq C( |\tau_{G} - \tau_{\p}|),
\end{align*}
and arguing similarly using Eq. (\ref{eqn.deltabineq2}),
\begin{align*}
\int_{\Hpe(M)} \big| \tau_{\p} \partial \fancyR_{\p} \big| e^{| (\tau_{G} - \tau_{\p}) \fancyR_{\p} + \tau_{\p} \partial \fancyR_{\p}|} \, d\muS{\p} \leq  \E \left[ \big| \tau_{\p} \partial \fancyR_{\p} \big| e^{| (\tau_{G} - \tau_{\p}) \fancyR_{\p} + \tau_{\p} \partial \fancyR_{\p}|}\right] \\
\leq C\Delta.
\end{align*}
In particular,
\begin{align}
 \E \left[ \big| (\tau_{G} - \tau_{\p}) \fancyR_{\p} + \tau_{\p} \partial \fancyR_{\p} \big| e^{| (\tau_{G} - \tau_{\p}) \fancyR_{\p} + \tau_{\p} \partial \fancyR_{\p}|}\right] \leq C( |\tau_{G} - \tau_{\p}| + \Delta). \label{eqn.intI2}
\end{align}
With Eqs.\,(\ref{eqn.intI1}) and (\ref{eqn.intI2}), Lemma \ref{lem.EeXminusY} implies $\int \big| I \big| \leq C (|\tau_G - \tau_{\p}| + \Delta)^{1/2}$.
Combining the bounds on $\int \big| I \big|$, $\int \big| II \big|$, and $\int \big| III \big|$ concludes the proof.
\end{proof}

\begin{prop}
\label{prop.limit2}
Under the assumptions of Theorem \ref{thm.mainthm}, the limit in Eq. (\ref{eqn.limit3}) is zero.
\end{prop}
\begin{proof} Combining Propositions \ref{prop:onlytrU}, \ref{prop:LpUptrSimp2},  \ref{prop:LpUptrSimp3}, and Eqs.\,(\ref{eqn.tauplimit}) and (\ref{eqn:tauG})  shows that the limit in Eq.\,(\ref{eqn.limit2}) vanishes when $\Delta \to 0$.
\end{proof}

\section{Appendix}
\subsect{Frequently Referenced Inequalities}
Here we collect several inequalities which are straight forward to show, but the frequency of use warrants their mention. For any $a \in \bbR$ and $p\in \nats$,
\begin{align}
\left| e^a - 1 \right|^p \leq e^{p|a|} - 1 \leq p|a|e^{p|a|}. \label{eqn.eaminus1}
\end{align}
If $a, b > 0$ and $\alpha \geq 1$,
\begin{align}
&\frac{\sinh(a)}{a} \leq \cosh(a), \label{eqn.sinhleqcosh}\\ 
&\cosh(a)\cosh(b) \leq \cosh(a+b),\label{eqn.coshacoshb} \\
&\cosh(a)(\cosh(b) - 1) \leq \cosh(a)\cosh(b) - 1 \label{eqn.coshabminus1} \\
&\alpha(\cosh(a) \cosh(b) - 1) \leq \cosh(\alpha a) \cosh( \alpha b) - 1. \label{eqn.alphacoshabminus1}
\end{align}

%------------------------------------------------------
\subsect{ODE Estimates}
\begin{prop}
\label{prop:ODE}
Let $s > 0$ and $J$ be an interval of $\re$ containing $[0,s]$.  Suppose that $z: J \to \Hom(\re^N \to \re^N)$ satisfies $z''(r) = A(r)z(r)$ where $A \in C^1(J \to \Hom(\re^N\to \re^N))$. We also suppose that there exist $K_0, K_1 > 0$ such that $\sup_{r \in [0,s]} \| A(r) \| \leq K_0$ and $\sup_{r \in [0,s]} \| A'(r) \| \leq K_1$. Then,
\begin{align}
\| z(s) - z(0) - s z'(0) \| \leq \| z(0) \| (\cosh(\sqrt{K_0} s) - 1) + \| z'(0) \| s \left( \frac{\sinh(\sqrt{K_0} s)}{\sqrt{K_0} s} - 1 \right). \label{eqn:ODE1}
\end{align}
If we assume that $z(0)=0$ and $z'(0)=I$, then
\begin{align}
\| z(s) - s I - \frac{s^3}{6}A(0) \| \leq \frac{s^4}{12} K_1 + \frac{s}{6}\left( \frac{\sinh(\sqrt{K_0}s)}{\sqrt{K_0}s} - 1 - \frac{1}{6}K_0 s^2 \right).
\end{align}
If instead we assume that $z(0)=I$ and $z'(0)=0$, then
\begin{align}
\| z(s) - I - \frac{s^2}{2}A(0) \| \leq \frac{s^3}{6} K_1 + \frac{1}{2} \left( \cosh(\sqrt{K_0} s)-1-\frac{1}{2} K_0 s^2 \right). 
\end{align}
\end{prop}
\begin{proof}
By Taylor's theorem with integral remainder,
\begin{align}
z(s) &= z(0) + s z'(0) + \int_0^{s}(s-r)z''(r)dr \notag \\
&= z(0) + s z'(0) + \int_0^s(s-r)A(r)z(r) dr. \notag
\end{align}
From here, iterating Talyor's theorem, 
\begin{align*}
&z(s) -z(0) - sz'(0)  \\
&=  \int_0^{s}(s-r)A(r)z(r)dr  \\
&= \int_0^s(s-r)A(r)\left\{ z(0) + r z'(0) + \int_0^{r}(r-t)A(t)z(t)dt \right\} dr  \\
& \cdots  \\
&= \left(\overbrace{\sum_{j=1}^m \int_{0< s_1 < \cdots < s_j < s} (s-s_j) \cdots (s_2-s_1) A(s_j) \cdots A(s_1) ds_1 \cdots ds_j}^{I} \right) z(0)  \\
&~ + \left( \overbrace{\sum_{j=1}^m \int_{0< s_1 < \cdots < s_j < s} (s-s_j) \cdots (s_2-s_1) s_1 A(s_j) \cdots A(s_1) ds_1 \cdots ds_j}^{II} \right) z'(0) \\
&~ + \overbrace{ \int_{0< s_1 < \cdots < s_{m+1} < s} (s-s_{m+1}) \cdots (s_2-s_1) A(s_{m+1}) \cdots A(s_1) z(s_1) ds_1 \cdots ds_{m+1}}^{III}.
\end{align*}
Using the bound on $A$, 
\begin{align}
\| I \| &\leq \sum_{j=1}^m \frac{s^{2j} K_0^{j}}{(2j)!} \leq \cosh(\sqrt{K_0} s)-1\\
\| II \| &\leq \sum_{j=1}^m \frac{s^{2j+1} K_0^j}{(2j+1)!} \leq \frac{\sinh(\sqrt{K_0} s)}{\sqrt{K_0}} - s
\end{align}
and
\begin{align}
\| III \| &\leq \sup_{r \in [0,s]} \| z(r) \| \frac{s^{2(m+1)}}{[2(m+1)]!}.
\end{align}
Taking $m \to \infty$ completes the proof of Eq. (\ref{eqn:ODE1}).

If $z(0)=0$ and $z'(0) = I$, then 
\begin{align*}
\| z(r) \| &\leq \frac{\sinh(\sqrt{K_0} r)}{\sqrt{K_0}}.
\end{align*}
Again iterating Taylor's theorem,
\begin{align*}
z(s) &= sI + \int_0^s(s-r)A(r)\left\{ rI + \int_0^r (r-t)A(t) z(t) dt \right\} dr \\
&= sI + \int_0^s (s-r)r A(0) dr + \int_0^s \int_0^r (s-r)r A'(t) dt dr \\
&~~ + \int_0^s \int_0^r (s-r)(r-t) A(r) A(t) z(t) dt dr \\
&= sI + \frac{s^3}{6} A(0) + \int_0^s \int_0^r (s-r)r A'(t) dt dr \\
&~~ + \int_0^s \int_0^r (s-r)(r-t) A(r) A(t) z(t) dt dr \\
\end{align*}
where the second equality came from $A(r) = A(0) + \int_0^r A'(t)dt$. Hence
\begin{align*}
\| z(s) - sI - \frac{s^3}{6}A(0) \| &\leq \frac{s^4}{12} K_1 + K_0^2 \int_0^s \int_0^r (s-r)(r-t)\frac{\sinh(\sqrt{K_0}t)}{\sqrt{K_0}} dt dr\\
&= \frac{s^4}{12} K_1 + \frac{s}{6}\left( \frac{\sinh(\sqrt{K_0}s)}{\sqrt{K_0}s} - 1 - \frac{1}{6}K_0 s^2 \right).
\end{align*}
If $z(0)=I$ and $z'(0)=0$, then $\| z(r) \| \leq \cosh(\sqrt{K_0} r)$, a similar expansion as above shows
\begin{align*}
z(s) &= I + \int_0^s(s-r)A(0)dr + \int_0^s \int_0^r (s-r)A'(t)dtdr \\
&~~ + \int_0^s \int_0^r (s-r)(r-t)A(r)A(t)z(t)dt dr\\
&= I + \frac{s^2}{2} A(0) + \int_0^s \int_0^r (s-r)A'(t)dtdr \\
&~~ + \int_0^s \int_0^r (s-r)(r-t)A(r)A(t)z(t)dt dr. 
\end{align*}
Therefore,
\begin{align*}
\| z(s) - I - \frac{s^2}{2} A(0) \| &\leq \frac{s^3}{6}K_1  + K_0^2 \int_0^s \int_0^r (s-r)(r-t)\cosh(\sqrt{K_0} t)dt dr\\
&= \frac{s^3}{6} K_1 + \frac{1}{2} \left( \cosh(\sqrt{K_0} s)-1-\frac{1}{2} K_0 s^2 \right). 
\end{align*}
\end{proof}

\begin{prop}
\label{prop:GREENest}
Let $z(s)$ be as in Proposition \ref{prop:ODE} and $f(s) := \frac{z(\Delta) - z(0)}{\Delta}s+z(0)$. Then,
\begin{align*}
&||z(s) - f(s)|| \\
&\qquad \leq s \left(1- \frac{s}{\Delta} \right) \left\{ ||z(0)|| K \Delta \cosh(\sqrt{K} \Delta)  +||z'(0)||\left( \cosh(\sqrt{K} \Delta) - 1 \right) \right\}
\end{align*}
\end{prop}
\begin{proof}
Let $g(s) := z(s) - f(s)$. Then $g''(s) = A(s)z(s)$ and satisfies the zero Dirichlet boundary condition $g(0) = g(\Delta) = 0$. For $s, t \in [0,1]$, let $G(s,t)$ be defined by
\[ G(s,t) := t \left( 1-\frac{s}{\Delta} \right) 1_{[0,s)}(t)  + s \left( 1-\frac{t}{\Delta} \right) 1_{[s,\Delta]}(t). \]
Then,
\begin{align*}
g(s) & =  - \int_0^{\Delta} G(s,t)g''(t)dt = - \int_0^{\Delta} G(s,t) A(t) z(t) dt.
\end{align*}
Therefore, using the fact that $0 \leq G(s,t) \leq s \left( 1-\frac{s}{\Delta} \right)$ and the estimate from proposition \ref{prop:ODE} , we have 
\begin{align*}
||g(s)|| & \leq s \left( 1-\frac{s}{\Delta} \right) K \int_0^{\Delta} \left(||z(0)|| \cosh(\sqrt{K} t) + ||z'(0)||\frac{\sinh(\sqrt{K} t)}{\sqrt{K}} \right) dt \\
& =  s \left( 1-\frac{s}{\Delta} \right)K\left(||z(0)||\frac{\sinh(\sqrt{K} \Delta)}{\sqrt{K}} + ||z'(0)|| \frac{\cosh(\sqrt{K} \Delta) -1}{K} \right)\\
&\leq s \left(1- \frac{s}{\Delta} \right) \left\{ ||z(0)|| K \Delta \cosh(\sqrt{K} \Delta) + ||z'(0)||\left( \cosh(\sqrt{K} \Delta) - 1 \right) \right\}.
\end{align*}
\end{proof}

\begin{rmk*}
Let $V$ be a finite-dimensional vector space and $\mc{C}_{0}$ be the collection of those functions in $C^2([0,\Delta] \to V)$  satisfying the zero Dirichlet boundary condition. The function $G(s,t)$ as defined in the proof of Proposition \ref{prop:GREENest} is the Green's function of the operator $- d^2/dt^2$ acting on $\mc{C}_0$. Indeed, for any $k \in \mc{C}_0$, $k(s) = - \int_0^{\Delta} G(s,t) k''(t) dt$.
\end{rmk*}

\begin{prop}
\label{prop:PosSemiDefBound}
Suppose that $t \mapsto A(t)$ is a smooth map where $A(t)$ is an $N \times N$ positive semi-definite matrix.  Let $t \mapsto \xi(t) \in \re^N$ for $t \geq 0$ be a smooth map with $\ddot{\xi}(t) = A(t) \xi(t), \xi(0) = 0,$ and $\dot{\xi}(0)\neq 0$. Then, 
\begin{align}
\| \xi(t) \| \geq t \|\dot{\xi}(0) \|.
\end{align}
\end{prop}
\begin{proof}
If $x(t):= \| \xi(t) \|$, then we want to show $x(t) \geq t \dot{x}(0)$. It suffices to show that $\ddot{x}(t) \geq 0$ for all $t$, since then $x(t) = \int_0^t \dot{x}(s)ds \geq t \dot{x}(0)$. By assumption $\dot{x}(0) > 0$ and since $x(0)=0$, there exists some $\epsi > 0$ with $ x(t) > 0$ for all $t \in (0,\epsi)$. For $t \in (0,\epsi)$, we calculate, 
\begin{align*}
\ddot{x}(t) &= \frac{\langle \ddot{\xi}(t), \xi(t) \rangle + \| \dot{\xi}(t)\|^2}{\| \xi(t) \|} - \frac{ \langle \xi(t), \dot{\xi}(t) \rangle^2}{ \| \xi(t) \|^3 } \\
&= \frac{ \langle A(t) \xi(t), \xi(t) \rangle}{\| \xi(t) \|} + \frac{\| \dot{\xi}(t) \|^2 \| \xi(t) \|^2 - \langle \xi(t), \dot{\xi}(t) \rangle^2}{\| \xi(t) \|^3 }  \geq 0
\end{align*}
where the last inequality comes from the positivity of $A(t)$ and the Cauchy-Schwarz inequality. Now, suppose we set $\tau = \sup\{ \epsi > 0 : x(t) > 0 \text{ for all } t \in (0,\epsi)\}$. Then a continuity argument reveals that if $\tau < \infty$, $x(\tau) \geq \tau \dot{x}(0) > 0$, but by the definition of $\tau$, $x(\tau)=0$, a contradiction. Therefore $\tau = \infty$ and the above argument shows that $\ddot{x}(t)\geq 0$ for all $t \in (0,\infty)$.
\end{proof}

%------------------------------------------------------$$$$$$$
\subsect{Probabilistic Inequalities}
For the following, $\p=\{ 0=s_0 < s_1 < s_2 < \cdots < s_n=1 \}$ is a partition of $[0,1]$ and $\{b_s : s \in [0,1]\}$ is a standard $\bbR^d$-valued brownian motion. The symbol $\ed$ means equal in distribution, and $N^d(0,1)$ is a $d$-dimensional standard normal random variable. 

\begin{lem}
\label{lem:EDb2}
Given any $C \geq 0$ and $p \in [1, \infty)$, if $p C \Delta_i s < 1$ for each $i$, then
\begin{align}
\E \left[ e^{\frac{p}{2} C \sum_{i=1}^n \| \Delta_i b \|^2 } \right] &= \prod_{i=1}^n (1 - p C \Delta_i s)^{-d/2} \\
& \to e^{d p C / 2} \text{ as } | \p | \to 0.
\end{align}
\end{lem}
\begin{proof} If $Z \ed N^d(0,1)$, the $\E[ \ex\{ pC \Delta_i s \|Z\|^2 / 2 \}] = (1-pC \Delta_i s)^{-d/2}$. Therefore the above equalities are elementary using independent increments and scaling of Brownian motion.
\end{proof}

\begin{lem} 
\label{lem:sumbpto0}
Take $c > 0$ and $p \in \nats$ with $p \geq 2$. For sufficiently small $| \p |$, there exists a $C = C(\ddim, c, p) < \infty$ such that,
\begin{align}
\E\left[ \| \Delta_i b \|^p \ex \left\{ c \sum_{j=1}^n \| \Delta_j b \|^2 \right\} \right] \leq C \Delta_i s |\p|^{\frac{p-2}{2}} \leq C | \p |^{\frac{p}{2}}. \label{eqn.deltabineq1}
\end{align} 
In particular, if $\Gamma \subset \{ 1, ..., n \}$ with $\#(\Gamma)=m$,
\begin{align}
\sum_{i\in \Gamma} \E \left[ \| \Delta_i b \|^p \ex \left\{ c \sum_{j \in \Gamma} \| \Delta_j b \|^2 \right\} \right] \leq C \left( \sum_{i\in \Gamma} \Delta_i s \right) |\p|^{\frac{p-2}{2}} \leq C m |\p|^{\frac{p}{2}}, \label{eqn.deltabineq2}
\end{align}
and,
\begin{align}
\sum_{i = 1}^n \E \left[ \| \Delta_i b \|^p \ex \left\{ c \sum_{j=1}^n \| \Delta_j b \|^2 \right\} \right] \leq C  |\p|^{\frac{p-2}{2}}. \label{eqn.deltabineq3}
\end{align}
\end{lem}
\begin{proof}
Notice first,
\begin{align*}
&\E \left[ \| \Delta_i b \|^p \ex \left\{ c \sum_{i=1}^n \| \Delta_i b \|^2 \right\} \right]\\
&\qquad =   \E \left[ \| \Delta_i b \|^p \ex\left\{ c \| \Delta_i b \|^2 \right\} \right] \E \left[ \ex\left\{ c \sum_{j\neq i} \|\Delta_j b\|^2\right\} \right].
\end{align*}
By Lemma \ref{lem:EDb2}, $\limsup_{| \p | \to 0} \E \left[ \ex\left\{ c \sum_{j\neq i} \|\Delta_j b\|^2\right\} \right] = e^{\ddim c }$ and thus
\begin{align*}
&  \E \left[ \| \Delta_i b \|^p \ex\left\{ c \| \Delta_i b \|^2 \right\} \right] \E \left[ \ex\left\{ c  \sum_{j\neq i} \|\Delta_j b\|^2\right\} \right] \\
&\qquad \leq 2  e^{\ddim c} \E\left[ \| \Delta_i b \|^p \ex\left\{ c \| \Delta_i b \|^2 \right\} \right] \\
&\qquad = 2 c e^{\ddim c} (\Delta_i s)^{p/2} \E\left[ \| b_1 \|^p \ex\left\{ c \Delta_i s \| b_1 \|^2 \right\} \right] \\
&\qquad \leq 2 c \Delta_i s e^{\ddim c} |\p|^{\frac{p-2}{2}} \E\left[ \| b_1 \|^p \ex\left\{ c |\p| \| b_1 \|^2 \right\} \right] \\
&\qquad \leq C \Delta_i s |\p|^{\frac{p-2}{2}}
\end{align*}
where $C$ is as desired. 

Using Eq.  (\ref{eqn.deltabineq1}), 
\begin{align*}
\sum_{i\in \Gamma} \E \left[ \| \Delta_i b \|^p \ex \left\{ c \sum_{j \in \Gamma} \| \Delta_j b \|^2 \right\} \right]   \leq \sum_{i\in \Gamma} \E \left[ \| \Delta_i b \|^p \ex \left\{ c \sum_{j=1}^n \| \Delta_j b \|^2 \right\} \right] \\
\leq C \left( \sum_{i\in \Gamma} \Delta_i s \right) |\p|^{\frac{p-2}{2}} 
\end{align*} 
If $\Gamma = \{ 1, 2, ..., n\}$, then $\sum_{i \in \Gamma} \Delta_i s = 1$ from which Eq. (\ref{eqn.deltabineq3}) follows. Otherwise, $\sum_{i \in \Gamma} \Delta_is \leq m|\p|$, from which the Eq.  (\ref{eqn.deltabineq2}) follows.
\end{proof}

The proof of Eq. (\ref{eqn:Zga}) below in Lemma \ref{lem:GaussBound} can be found in \cite[Lemma 8.6]{AndDrive:1999}, but a full proof is included here, since we need more for our purposes.
\begin{lem}
\label{lem:GaussBound}
Let $Z \ed N^{d}(0,1)$, $k \geq 0$, and $a > 0$. Then there exists a $C<\infty$ depending only on $k$ and $d$ such that 
\begin{align}
\E[ e^{k \| Z \|} : \| Z \| \geq a] \leq \frac{C}{a^2} e^{-\frac{1}{4}a^2}. \label{eqn:Zga}
\end{align}
If we further restrict $k < 1/2$, then we can take $C$ such that 
\begin{align}
\E[ e^{k \| Z \|^2} : \| Z \| \geq a] \leq \frac{C}{a^2} e^{-\frac{1-2k}{4}a^2}. \label{eqn:Z2ga}
\end{align}
\end{lem}
\begin{proof}
We can find a $\Lambda = \Lambda(k,\ddim) < \infty$ such that $r^{\ddim -1} e^{kr} e^{-r^2/2} \leq \Lambda e^{-3r^2/8}$ and $r^{d-1} e^{-(1-2k)r^2/2} \leq \Lambda e^{-3(1-2k)r^2/8}$. Now, let $\w_{d-1}$ be the volume of the unit sphere in $\re^d$, and for $b > 0$ we have,
\begin{align*}
\Lambda \w_{d-1}(2 \pi)^{d/2} \int_a^{\infty} e^{-\frac{3b}{8} r^2} dr &\leq \Lambda \w_{d-1}(2 \pi)^{d/2} \int_a^{\infty} \frac{r}{a} e^{-\frac{3b}{8} r^2} dr \\
&= \frac{8 \Lambda \w_{d-1}(2\pi)^{d/2}}{6 b a}e^{-\frac{3b}{8}a^2} \\
&\leq \frac{C}{a^2} e^{-\frac{b}{4}a^2}.
\end{align*}
Realizing that 
\begin{align*}
\E[ e^{k \| Z \|} : \| Z \| \geq a] &= \w_{d-1}(2 \pi)^{d/2} \int_a^{\infty} r^{d-1} e^{kr}e^{-\frac{1}{2} r^2} dr \\
&\leq \Lambda \w_{d-1}(2 \pi)^{d/2} \int_a^{\infty} e^{-\frac{3}{8} r^2} dr
\end{align*}
and
\begin{align*}
\E[ e^{k \| Z \|^2} : \| Z \| \geq a] &= \w_{d-1}(2 \pi)^{d/2} \int_a^{\infty} r^{d-1} e^{-\frac{1-2k}{2} r^2} dr \\
&\leq \Lambda \w_{d-1}(2 \pi)^{d/2} \int_a^{\infty} e^{-\frac{3(1-2k)}{8} r^2} dr
\end{align*}
implies the result.
\end{proof}

\begin{lem}
\label{lem.EeXminusY}
Let $X$ and $Y$ be random variables on the probability space $\W$. Suppose that  for some  $p \geq 0$, there is a constant $M(p)$ such that $\E[e^{2pY}] \leq M(p)$, and some $R(p) < \infty$ such that any of the following hold,
\begin{align}
&\E \left[ 2 | X - Y| e^{2 p |X-Y|} \right] \leq R(p), \label{eqn.EeXminusY1}\\
&\E \left[ e^{2 p|X-Y|} - 1 \right] \leq R(p). \label{eqn.EeXminusY2} \\
&\left| \E \left[ e^{p(X-Y)} - 1 \right] \right| \vee \left| \E \left[ e^{2p(X-Y)} - 1 \right] \right|  \leq R(p).  \label{eqn.EeXminusY3}
\end{align}
Then, given any measurable subset $A \subset \W$, 
\begin{align}
\E \left[ |e^{pX} - e^{pY}| : A \right] \leq \sqrt{3M(p)R(p)}.  \label{eqn.EeXminusY4}
\end{align}
\end{lem}
\begin{proof} 
The fact that  Eq. (\ref{eqn.EeXminusY1}) implies Eq. (\ref{eqn.EeXminusY2}) comes from Eq. (\ref{eqn.eaminus1}), and that Eq.\,(\ref{eqn.EeXminusY2}) implies Eq.\,(\ref{eqn.EeXminusY3}) is clear. So we will assume Eq. (\ref{eqn.EeXminusY3}), and without losing generality also  assume that  $A = \W$.  To start, 
\[ (e^{p(X-Y)} - 1)^2 = (e^{2p(X-Y)}-1) - 2(e^{p(X-Y)}-1).\]
Therefore, using Holder's inequality and Eq. (\ref{eqn.eaminus1}),
\begin{align*}
&\E[|e^{pX} - e^{pY}|]\\
 &\leq \left( \E[e^{2pY}] \right)^{1/2} \left( \E[(e^{p(X-Y)} - 1)^2] \right)^{1/2}\\
 & \leq \sqrt{M(p)} \cdot  \left( \E \big[(e^{2p(X-Y)}-1) - 2(e^{p(X-Y)}-1)\big] \right)^{1/2} \\
 &\leq \sqrt{M(p)} \sqrt{3 R(p)}.
\end{align*}
which concludes the proof.
%At the heart of this is the equality 
%\begin{align*}
%(e^{X} - e^{Y})^2 = e^{2Y} \left[(e^{2(X-Y)} - 1) - 2(e^{X-Y} - 1)\right].
%\end{align*}
%Applying Jensen's inequality yields 
%\begin{align*}
%| \E (e^{X} - e^Y) | \leq \left[ \E (e^X - e^Y)^2 \right]^{1/2}.
%\end{align*} 
%Combining this with the above equality and performing the obvious bounds yields the result.
\end{proof}

 The following result is proved in \cite[Proposition 8.8]{AndDrive:1999}.
\begin{prop}
\label{prop.ItoTraceOnlyImportant}
Let $\{ R_i \}_{i=1}^n$ be a collection of random $\ddim \times \ddim$ symmetric matrices such that $R_i$ is $\sigma( b_r : 0 \leq r \leq s_{i-1} )$-measurable. Suppose there is some $K < \infty$ with $\|R_i \|\leq K$ for all $i$. Then, given any $p \in \re$,
\begin{align}
1 \leq \E\left[ \ex\{ p \sum_{i=1}^n \left( \langle R_i \Delta_i b, \Delta_i b \rangle - \tr(R_i)\Delta_i s \right) \}\right] \leq e^{2 \ddim p^2 K^2 | \p |}.
\end{align}
\end{prop}

%----------------------------$$$$$$$$$$$$------$$$$$$$$$$$$
\subsect{A Lemma in Stochastic Integration}
\begin{lem}
\label{lem:limsupGeneral}
Let $p \in \nats$ and $\alpha, \beta, \gamma \in \re$ with $\alpha < \frac{1}{4p}, 0 \leq\beta,$ and $-1 \leq \gamma$. Set $\mathbf{B}_n := (\Delta_1 b, \cdots, \Delta_n b)$ where $\{b_s : s \in [0,1]\}$ is an $\re^{\ddim}$-valued Brownian motion. Then 
\begin{align}
\limsup_{n \to \infty} \E \left[ \left(1 + \frac{\beta}{n}\left( e^{\alpha \| \mathbf{B}_n \|^2}  + \gamma \right) \right)^{np} \right] < \infty. \label{eqn:limsupGeneral}
\end{align}
\end{lem}
\begin{proof}
We define the deterministic functions $g(x) := 1 + \frac{\beta}{n}\left( e^{\alpha x}  + \gamma \right)$ and $f(x) := g(x)^{np} = (1 + \frac{\beta}{n}\left( e^{\alpha x}  + \gamma \right))^{np}$. We also define the stochastic process $Q_t^n := \sum_{i=1}^n \| b_{t \wedge s_i} - b_{t \wedge s_{i-1}} \|^2$. With this notation, Eq. (\ref{eqn:limsupGeneral}) becomes
\begin{align}
\limsup_{n\to \infty} \E[ f(Q_1^n) ] < \infty.
\end{align}
We now use It\^{o}'s Lemma to get an estimate on $\E[ f(Q_t^n) ]$. To start, $d Q^n_s = 2 ( b_s - b_{\underline{s}} ) db_s + \ddim ds$ and $d[Q^n]_s = 4 \| b_s - b_{\underline{s}} \|^2 ds$ where $\underline{s} = s_{i-1}$ whenever $s \in (s_{i-1}, s_i]$. Therefore,
\begin{align*}
\E[f(Q_t^n)- f(Q_0^n)] &= \E\left[ \int_0^t f'(Q_s^n) dQ^n_s + \frac{1}{2} \int_0^t f''(Q_s^n) d[Q^n]_s \right] \\
&= \E \left[ \ddim \int_0^t f'(Q_s^n) ds + 2 \int_0^t f''(Q_s^n)\|b_s - b_{\underline{s}}\|^2 ds\right] \\
&= \ddim \int_0^t \E[ f'(Q_s^n) ] ds + 2 \int_0^t \E[f''(Q_s^n)\|b_s - b_{\underline{s}}\|^2 ] ds
\end{align*}
where in the second equality we dropped the martingale term. Calculating the derivatives of $f$,
\begin{align*}
& f'(x) = \beta \alpha p e^{\alpha x} g(x)^{np-1} \\
& f''(x) = \beta^2 \alpha^2 p ( p- \frac{1}{n}) e^{2 \alpha x} g(x)^{np-2} + \beta \alpha^2 p e^{\alpha x} g(x)^{np-1}.
\end{align*}
By our choices for $\beta$ and $\gamma$, $g(x) \geq 1$ for $x \geq 0$. This implies that there exists a constant $C < \infty$ independent of $n$ such that
\begin{align}
\E[f(Q_t^n)- f(Q_0^n)] &\leq C \int_0^t \E[ e^{2 \alpha Q_s^n} (1 + \|b_s - b_{\underline{s}}\|^2 ) g(Q_s^n)^{np-1}] ds.
\label{eqn:limsupGeneralIneq1}
\end{align}
From here we want to show that there exists another constant $\tilde{C} < \infty$ independent of $n$ such that,
\begin{align*}
\E[f(Q_t^n)- f(Q_0^n)] &\leq \tilde{C} \int_0^t \E[ g(Q_s^n)^{np}] ds \\
&= \tilde{C} \int_0^t \E[ f(Q_s^n)] ds.
\end{align*}
Before we do this, let us first understand why this will be enough to finish the proof. If such a $\tilde{C}$ exists, then we will have
\begin{align*}
\E[f(Q_t^n)] &\leq \E[f(Q_0^n)] + \tilde{C} \int_0^t \E[ f(Q_s^n)] ds \\
&= \left( 1 + \frac{\beta}{n}(1 + \gamma)\right)^{np} + \tilde{C} \int_0^t \E[ f(Q_s^n)] ds \\
&\leq e^{\beta p (1 + \gamma)} + \tilde{C} \int_0^t \E[ f(Q_s^n)] ds.
\end{align*}
Applying Gronwall's inequality to the function $t \mapsto \E[f(Q_t^n)]$, 
\begin{align*}
\E[f(Q_t^n)] &\leq e^{\beta p (1 + \gamma) + \tilde{C}t},
\end{align*}
noting the the right hand side is independent of $n$. In particular, for any $t \in [0,1]$,
\begin{align}
\limsup_{n \to \infty} \E[f(Q_t^n)] &\leq e^{\beta p (1 + \gamma) + \tilde{C}t} < \infty,
\end{align}
which concludes the proof as soon as the existence of $\tilde{C}$ is established.

From Eq. (\ref{eqn:limsupGeneralIneq1}), to prove the existence of $\tilde{C}$ it will suffice to show that there exists a constant $\Lambda$ independent of $n$ such that $\E[ e^{2 \alpha Q_s^n} (1 + \|b_s - b_{\underline{s}}\|^2 ) g(Q_s^n)^{np-1}] \leq \Lambda \E[g(Q_s^n)^{np}]$. Using Holder's inequality,
\begin{align*}
&\E[ e^{2 \alpha Q_s^n} (1 + \|b_s - b_{\underline{s}}\|^2 ) g(Q_s^n)^{np-1}] \\
\qquad &\leq \E[ e^{2 \alpha np Q_s^n} (1 + \|b_s - b_{\underline{s}}\|^{2} )^{np}]^{\frac{1}{np}} \E[ g(Q_s^n)^{np}]^{\frac{np-1}{np}} \\
\qquad &\leq \E[ e^{2 \alpha np Q_s^n} (1 + \|b_s - b_{\underline{s}}\|^{2} )^{np}]^{\frac{1}{np}} \E[ g(Q_s^n)^{np}],
\end{align*}
where we once again used that $g \geq 1$ for the second inequality. Therefore it is sufficient to find such a $\Lambda$ with $\E[ e^{2 \alpha np Q_s^n} (1 + \|b_s - b_{\underline{s}}\|^{2} )^{np}]^{\frac{1}{np}} \leq \Lambda$. If $\{ Z_i \}_{i=1}^{\infty}$ are i.i.d. $N^{\ddim}(0,1)$ random variables and $s \in (s_{j-1}, s_j]$, then 
\begin{align*}
Q_s^n &= \sum_{i=1}^{j-1} \| \Delta_i b \|^2 +  \| b_s - b_{\underline{s}} \|^2 \\
&\ed \sum_{i=1}^{j-1} \frac{1}{n} \| Z_i \|^2 +  (s- \underline{s}) \|Z_j \|^2
\end{align*} 
and
\begin{align*}
1 + \| b_s - b_{\underline{s}} \|^2 \ed 1 +(s- \underline{s})  \| Z_j \|^2.
\end{align*}
Therefore,
\begin{align*}
&\E[ e^{2 \alpha np Q_s^n} (1 + \|b_s - b_{\underline{s}}\|^{2} )^{np}]^{\frac{1}{np}} \\
&\qquad \leq \E[ \left( \prod_{i=1}^{j-1} e^{2 \alpha p \|Z_i\|^2} \right) e^{2 \alpha p \|Z_j\|^2} (1 + (s- \underline{s})\| Z_j \|^{2} )^{np}]^{\frac{1}{np}} \\
&\qquad =  \left( \prod_{i=1}^{j-1}\E[ e^{2 \alpha p \|Z_i\|^2}] \right)^{\frac{1}{np}} \E[e^{2 \alpha p \|Z_j\|^2} (1 + (s- \underline{s}) \| Z_j \|^{2} )^{np}]^{\frac{1}{np}} \\
&\qquad = \E[ e^{2 \alpha p \|Z_1\|^2}]^{\frac{j-1}{np}} \E[e^{2 \alpha p \|Z_j\|^2} (1 + (s- \underline{s}) \| Z_j \|^{2} )^{np}]^{\frac{1}{np}} \\
&\qquad \leq  \E[ e^{2 \alpha p \|Z_1\|^2}]^{\frac{1}{p}} \E[e^{2 \alpha p \|Z_j\|^2} (1 + \frac{1}{n} \| Z_j \|^{2} )^{np}]^{\frac{1}{np}}.
\end{align*}
With $\alpha < \frac{1}{4p}$, $\E[e^{2 \alpha p \|Z_1\|^2}]^{\frac{1}{p}} = (1-4 \alpha p)^{- \frac{1}{p}}$. For the second term, find some $\delta$ with $\alpha < \delta < \frac{1}{4p}$ and set $m = \inf\{ l \in \nats : l \geq \frac{\delta}{\delta-\alpha} \}$, the ceiling of $\frac{\delta}{\delta-\alpha}$. Again using Holder's inequality,
\begin{align*}
\E[e^{2 \alpha p \|Z_j\|^2} (1 + \frac{1}{n}\| Z_j \|^{2} )^{np}] &\leq \E[e^{2 \delta p \|Z_j\|^2}]^{\frac{\alpha}{\delta}} \E[(1+ \frac{1}{n}\| Z_j \|^{2})^{np \frac{\delta}{\delta-\alpha}} ]^{\frac{\delta - \alpha}{\delta}} \\
&\leq (1 - 4 \delta p)^{-\frac{\alpha}{\delta}} \E[(1+ \frac{1}{n}\| Z_j \|^{2})^{npm} ]^{\frac{\delta - \alpha}{\delta}}
\end{align*}
Using the binomial formula,
\begin{align*}
\E[(1+ \frac{1}{n}\| Z_j \|^{2})^{npm} ] &= \sum_{k=0}^{npm} \chs{npm}{k} \left(\frac{1}{n}\right)^k \E[ \| Z_j \|^{2k}] \\
&=  \sum_{k=0}^{npm} \chs{npm}{k}  \left(\frac{1}{n}\right)^k \frac{(2k)!}{2^k k!} \\
&\leq  \sum_{k=0}^{npm} \chs{npm}{k}  \left(\frac{1}{n}\right)^k \frac{1}{2^k}\frac{e}{\sqrt{\pi}} \left( \frac{4}{e} \right)^k k^k \\
&\leq \sum_{k=0}^{npm} \chs{npm}{k} \frac{e}{\sqrt{\pi}} \left( \frac{2pm}{e} \right)^k \\
&= \frac{e}{\sqrt{\pi}} \left(1 +  \frac{2pm}{e} \right)^{npm}.
\end{align*}
Here the third line comes from Stirling's approximation. Putting these pieces together,
\begin{align*}
&\E[ e^{2 \alpha np Q_s^n} (1 + \|b_s - b_{\underline{s}}\|^{2} )^{np}]^{\frac{1}{np}} \\
&\leq \E[ e^{2 \alpha p \|Z_1\|^2}]^{\frac{1}{p}} \E[e^{2 \alpha p \|Z_j\|^2} (1 + \frac{1}{n} \| Z_j \|^{2} )^{np}]^{\frac{1}{np}}\\
&\leq (1-4 \alpha p)^{- \frac{1}{p}} \left[ (1 - 4 \delta p)^{-\frac{\alpha}{\delta}}  \left( \frac{e}{\sqrt{\pi}} \left(1 +  \frac{2pm}{e} \right)^{npm} \right)^{\frac{\delta-\alpha}{\delta}} \right]^{\frac{1}{np}} \\
&= (1-4 \alpha p)^{- \frac{1}{p}} \left[ (1 - 4 \delta p)^{-\frac{\alpha}{\delta}} \left( \frac{e}{\sqrt{\pi}}\right)^{\frac{\delta-\alpha}{\delta}}  \right]^{\frac{1}{np}} \left(1 +  \frac{2pm}{e} \right)^{m\frac{\delta-\alpha}{\delta}} \\
&\leq (1-4 \alpha p)^{- \frac{1}{p}} (1 - 4 \delta p)^{-\frac{\alpha}{\delta}} \left( \frac{e}{\sqrt{\pi}}\right)^{\frac{\delta-\alpha}{\delta}} \left(1 +  \frac{2pm}{e} \right)^{m\frac{\delta-\alpha}{\delta}}\\
&=: \Lambda < \infty
\end{align*}
where as desired $\Lambda$ is independent of $n$.
\end{proof}

%---------------------------------------------------$$$$$$$$$$$

\section*{Acknowledgements}
The author would like to acknowledge Dr.\,Bruce K. Driver for direction and countless hours of conversation and assistance in completing this paper. The author is also grateful to Dr.\,Maria Gordina for her helpful remarks, and to the anonymous referee whose careful reading and useful commentary improved the quality of this paper.

\bibliographystyle{plain}	% (uses file "plain.bst")
\bibliography{refs}
\end{document}